\documentclass[twoside]{article}
\usepackage[accepted]{aistats2022}
\usepackage{macros}  
\usepackage{natbib}

\begin{document}

%

%

\twocolumn[

\aistatstitle{Acceleration in Distributed Optimization under Similarity}

\aistatsauthor{ Ye Tian$^\ast$ \And Gesualdo Scutari$^\ast$ \And Tianyu Cao$^\ast$  \And    Alexander Gasnikov$^{\dag}$ }

\aistatsaddress{ $^\ast$Purdue University \\ \small{$^\dag$ MIPT, ISP RAS Research Center for Trusted Artificial Intelligence} } ]

\begin{abstract}\vspace{-0.3cm}
We study distributed (strongly convex) optimization problems over a network of agents, with no centralized nodes. The loss functions of the agents are assumed to be \textit{similar}, due to statistical data similarity or otherwise.  In order to reduce the number of  communications  to reach a solution accuracy, we proposed a {\it preconditioned, accelerated} distributed  method. 
An $\varepsilon$-solution is achieved in $\tilde{\mathcal{O}}\big(\sqrt{\frac{\beta/\mu}{1-\rho}}\log1/\varepsilon\big)$ number of communications steps, where $\beta/\mu$ is the relative condition number between the global and local loss functions, and $\rho$ characterizes the connectivity of the network. This rate  matches (up to poly-log factors)  lower complexity communication bounds   of distributed gossip-algorithms applied to  the class of problems of interest. Numerical results  show significant communication savings with respect to existing accelerated distributed schemes, especially when solving ill-conditioned problems. \vspace{-0.3cm} 



\end{abstract}

\section{INTRODUCTION}\vspace{-0.1cm} 
We study distributed optimization over a network of $m$ agents, in the form \vspace{-0.2cm}
\begin{equation}\label{eq:problem}
\!\!\tag{P}
\min_{x\in \mathbb{R}^d}\,\, u(x) \triangleq f(x) + r(x),\quad f(x)\triangleq \frac{1}{m} \sum_{i=1}^m f_i(x), 
\end{equation}
where $f_i:\mathbb{R}^d\to \mathbb{R}$ is the loss function of agent $i$, known only to that agent; and $r:\mathbb{R}^d\to [-\infty, \infty]$ is an extended-value function (known to all agents), which is instrumental to enforce further conditions on the solution, such as sparsity or constraints. The network of agents is modelled as undirected, fixed graph, with no centralized node; 
we refer to such architectures as {\it mesh} networks.\vspace{-0.1cm}

An instance of \eqref{eq:problem} of particular interest to this work is the distributed Empirical Risk Minimization (ERM) whereby the goal is to minimize the average loss over some dataset, distributed across the nodes of the network. Specifically, denoting  by $\mathcal{D}_i=\{z_i^{(1)},\ldots, z_i^{(n)}\}$ the set of $n$ samples owned by agent $i$, 
the empirical risk $f_i$ in (\ref{eq:problem}) reads\vspace{-0.1cm}
   \begin{equation}
       \label{eq:local_risk} f_i(x)=\frac{1}{n}\sum_{j=1}^n \ell\big(x;z_i^{j}\big),\vspace{-0.2cm}
   \end{equation}
  where $\ell(x;z_i^{j})$  measures the mismatch between the   parameter $x$ and the sample $z_i^{j}$.  
  
The lack of global knowledge of $f$  from the agents and of a centralized node in the network calls for the    design of distributed algorithms, whereby agents  alternate   a computation procedure  based on local information and   communication round(s)  with  neighboring nodes. 
Since the cost of communications is often  the bottleneck in distributed computing, if compared with local (parallel) computations (e.g., \cite{Bekkerman_book11,Lian17}), a lot of research has been devoted to designing distributed algorithms that are \emph{communication efficient}. \vspace{-0.1cm} 

   

\textbf{Communication-saving  via acceleration:} Acceleration (in the sense of Nesterov) has been extensively investigated as a procedure to reduce the communication burden of distributed algorithms--Table~\ref{table} summarizes  the communication complexity of  existing first-order methods  over mesh networks (see Sec.~\ref{sec:related_works} for a discussion  of these works).  For $L$-smooth and $\mu$-strongly convex functions $f$ in (\ref{eq:problem}), linear convergence rate is certified, with a constant factor scaling with  $\sqrt{\kappa}$   ($\kappa\triangleq L/\mu$ is the condition number of $f$). This dependence  is not improvable, meaning that it matches lower communication complexity bounds for the class of  $L$-smooth and $\mu$-strongly convex function $f_i$'s \citep{scaman2017optimal}. However, for ill-conditioned problems--e.g., the typical setting of many ERM problems wherein the optimal regularization parameter   for test predictive performance is very small--$\kappa$ can be extremely large; hence the aforementioned scaling   of the number of communications with  $\sqrt{\kappa}$ is unsatisfactory.

\noindent \textbf{Exploiting function similarity:} Further improvements can be obtained if extra structure is postulated for the $f_i$'s. This is, e.g., the case of ERM problems wherein each $f_i$ [see   (\ref{eq:local_risk})] has an additional finite-sum structure. 
This is an instance of the property known as {\it function similarity} (see, e.g.,  \cite{shamir2014communication,arjevani2015communication,zhang2015disco,hendrikx2020statistically}):
 \begin{equation}
 \label{eq:intro_stat_sim}
 	 \left\|\nabla^2f_i(\bx) - \nabla^2 f(\bx)\right\|\leq  {\beta},\end{equation}      
 for all $x$ in a proper domain of interest and all   $i=1,\ldots, m$, where $\beta>0$ measures the degree of similarity between the  Hessian matrices of    local and global losses. For instance,   in the aforementioned ERM setting, when data are i.i.d. among machines,  the $f_i$'s in (\ref{eq:local_risk}) reflect  statistical similarities in the data residing at different nodes, resulting in  $\beta=\tilde{O}(1/\sqrt{n})$ w.h.p., where $n$ is the local sample size   ($\tilde{O}$ hides log-factors and dependence on  $d$).

The situation  $1+\beta/\mu\ll\kappa$, happens in several scenarios. For instance, consider some ill-conditioned functions. Another example are ERMs with optimal regularization  $\mu=\mathcal{O}(1/\sqrt{mn})$ and $L=\mathcal{O}(1)$ (e.g., see  \cite[Table 1]{zhang2015disco} for ridge regression), we have: $\kappa=\mathcal{O}(\sqrt{m\cdot n})$ while $\beta/\mu=\mathcal{O}(\sqrt{m})$--the former  grows with the local sample size $n$, while  the latter is independent. 
This motivated a surge of studies aiming at exploiting function similarity coupled with acceleration to boost communication efficiency  (see Sec.~\ref{sec:related_works}):   linear convergence is certified with a  number of communication steps scaling (asymptotically \citep{hendrikx2020statistically}) with $1+\sqrt{\beta/\mu}$ \citep{zhang2015disco}, which can be significantly smaller than $\sqrt{\kappa}$, and matches lower complexity bounds \citep{arjevani2015communication} up to log-factors. 
These algorithms however are {\it centralized} and    {\it cannot} be  implemented over {\it mesh} networks. 
This  suggests  the following open question: 

\fbox{\begin{minipage}{22em}
Is linear convergence with  a number of communications  scaling with $1+\sqrt{\beta/\mu}$   achievable by any distributed algorithm over {\it mesh} networks? \end{minipage}}

\renewcommand{\arraystretch}{1.8}
\begin{table*}[t!]
\caption{\small \textbf{Distributed Accelerated Algorithms over Mesh Networks}: \texttt{SSDA}/\texttt{MSDA} \citep{scaman2017optimal}, \texttt{OPAPC} \citep{kovalev2020optimal}, \texttt{Accelerated Dual Ascent} \citep[Alg. 3]{uribe2020dual},  \texttt{APM-C} \citep{li2018sharp},  \texttt{Mudag} \citep{ye2020multi},  \texttt{Accelerated EXTRA} \citep{doi:10.1137/18M122902X}, \texttt{DAccGD} \citep{rogozin2020towards}, and \texttt{DPAG} \citep{ye2020decentralized}.  $L$ (resp. $\mu$) denotes the smoothness (resp. strong convexity) constant of $F$ while $L_{\textnormal mx}$ (resp. $\mu_{\textnormal mn}$) is the largest smoothness (smallest strong convexity) constant of the $f_i$'s; $\rho$ is the network connectivity;  $\widetilde{\mathcal{O}}$ hides poly-log factors.}
\medskip 
\begin{small}
\resizebox{2.1\columnwidth}{!}{\begin{tabular}{c|c|c|c|c}
 {\bf Algorithm} &  {\bf Problem} &  {\bf Similarity} &  {\bf Gossip Matrix} &  {\bf Rate (\# comm.)} \\ \hline 
                    \texttt{SSDA}/\texttt{MSDA}, \texttt{OPAPC},  \texttt{Accelerated Dual Ascent}         &   $f_i$ scv, $r\equiv 0$              &     \XSolidBrush             &                                & $\mathcal{O}\left(\sqrt{\frac{L_{\textnormal mx}}{\mu_{\textnormal mn}}}\sqrt{\frac{1}{1-\rho}}\log(\frac{1}{\varepsilon})\right)$    \rule{0pt}{2.6ex} \rule[-2.6ex]{0pt}{0pt}\\ \hline
                       \texttt{APM-C}    &      $f_i$ scv, $r\equiv 0$          &         \XSolidBrush         &    PSD                            &     $\mathcal{O}\left(\sqrt{\frac{L_{\textnormal mx}}{\mu_{\textnormal mn}}}\sqrt{\frac{1}{1-\rho}}\log^2(\frac{1}{\varepsilon})\right)$   \rule{0pt}{2.6ex} \rule[-2.6ex]{0pt}{0pt}               \\ \hline
                     \texttt{ Mudag}      &     $f$ scv, $r\equiv 0$              &          \XSolidBrush         &   PSD                             &       $\widetilde{O}\left(\sqrt{\frac{L}{\mu}}\sqrt{\frac{1}{1-\rho}}\log(\frac{1}{\varepsilon})\right)$      \rule{0pt}{2.6ex} \rule[-2.6ex]{0pt}{0pt}          \\ \hline
                       \texttt{Accelerated EXTRA}      &      $f_i$ scv, $r\equiv 0$          &         \XSolidBrush                   &                   &  $\widetilde{O}\left(\sqrt{\frac{L_{\text{mx}}}{\mu_{\text{mn}}}}\sqrt{\frac{1}{1-\rho}}\log(\frac{1}{\varepsilon})\right)$                                  \rule{0pt}{2.6ex} \rule[-2.6ex]{0pt}{0pt} \\ \hline
                    \texttt{ DAccGD}       &    $f_i$ scv, $r\equiv 0$                  &           \XSolidBrush        &                                &      $\widetilde{O}\left(\sqrt{\frac{L}{\mu}}{\frac{1}{1-\rho}}\log^2(\frac{1}{\varepsilon})\right)$            \rule{0pt}{2.6ex} \rule[-2.6ex]{0pt}{0pt}      \\ \hline
                      \texttt{ DPAG }     &      $f$ scv, $r\not\equiv 0$            &       \XSolidBrush            &   PSD                             &                  $\widetilde{O}\left(\sqrt{\frac{L}{\mu}}\sqrt{\frac{1}{1-\rho}}\log(\frac{1}{\varepsilon})\right)$                                  \rule{0pt}{2.6ex} \rule[-2.6ex]{0pt}{0pt}       \\ \hline
{\bf ACC-SONATA (this work)}                  &    $f$ scv, $r\not\equiv 0$              &   \Checkmark               &                                &               $\widetilde{O}\left(\sqrt{\frac{\beta}{\mu}}\sqrt{\frac{1}{1-\rho}}\log(\frac{1}{\varepsilon})\right)$        \\ 
\end{tabular}}\end{small}\label{table}\vspace{-0.2cm}
\end{table*}

  \subsection{Contributions}\vspace{-0.2cm}
 We provide a positive answer to the above question.\smallskip \\ 
\textbf{(i) Algorithm design:} We proposed Accelerated-SONATA (\texttt{ACC-SONATA}), an 
inexact accelerated proximal   method (outer loop) for \eqref{eq:problem}, embedded with the distributed algorithm known as \texttt{SONATA} \citep{sun2019distributed} (inner loop),  which approximatively solves     the proximal subproblem over mesh networks, according to a   properly  designed   notion of inexactness. \texttt{SONATA} couples local preconditioning via  surrogation of $f_i$  
with a gradient tracking mechanism \citep{di2016next,xu2017convergence}, estimating locally the  gradient of the global loss $f$.  At high level,  the outer loop ensures acceleration while \texttt{SONATA} exploits    function similarity to boost the convergence rate of the inner loop. 
A direct acceleration of the mirror method, achieving $\widetilde{\mathcal{O}}(\beta/\mu)$ over star-networks  \citep{LuFreundNesterov2020},  does not seem possible  in general \citep{Dragomir19}. \vspace{-.1cm}



\textbf{(ii)\! New analysis:}\! While \texttt{ACC-}\texttt{SONATA} is  inspired by   proximal acceleration      for centralized optimization \citep{dAspremont2021AccelerationM}, such as Catalyst   \citep{lin2015Catalyst}, formally  it is not an instance of  any of  existing methods. It is also different from distributed algorithms accelerated \'{a} la 
Catalyst \citep{doi:10.1137/18M122902X,Hendrikx_DualFree20}, which neither   exploit   function similarity nor use   gradient-tracking,  and    deal with smooth optimization ($r\equiv 0$). Hence, a new convergence   analysis is needed for \texttt{ACC-SONATA}, which represents the technical contribution of this work. We hinge on 
  new potential functions (for the inner and outer loop) that incorporate consensus errors, 
  extrapolation variables, and   gradient-tracking variables.   Such   potentials also shed light on the  appropriate    warm-restart strategy and  termination criterion of the inner-loop \texttt{SONATA}, which are   implementable in a distributed setting (this is not the case if using  criteria in \cite{lin2015Catalyst-journal}, in particular when $r\neq 0$). We  remark that the proposed analysis, although  developed for \texttt{ACC-SONATA},  is fairly general and potentially applicable to a variety of other distributed algorithms replacing \texttt{SONATA} in the inner loop. This will be the subject of future investigation.  

\textbf{(iii) Guarantees:} By   a proper choice of   local surrogations (mirror-prox-like), \texttt{ACC-SONATA} provably achieves an $\varepsilon$-solution (on the objective value) of \eqref{eq:problem} in $\tilde{\mathcal{O}}\big(\sqrt{\frac{\beta/\mu}{1-\rho}}\log1/\varepsilon\big)$ number of communications steps, where  $\rho$ characterizes the connectivity of the network. This matches for the first time lower communication complexity bounds, up to poly-log factors.  On the other hand, when $1+\sqrt{\beta/\mu}>\sqrt{\kappa}$
   a different choice of surrogate (linearization of each $f_i$) is possible, which guarantees linear convergence with communication scaling of $\sqrt{\kappa}$--this compares  favorable   with  existing first-order accelerated methods (see Table~\ref{table}), with \texttt{ACC-SONATA} being applicable on general \eqref{eq:problem} with $r\not\equiv 0$ and with no restrictions on gossip matrices. Our numerical results on synthetic and real data supports the competitive performance of our method over the state of the art. 
\subsection{Related Works}\label{sec:related_works}\vspace{-0.2cm}
Acceleration under function similarity (\ref{eq:intro_stat_sim}) has been extensively investigated  over   master/workers architectures 
while it remains unexplored  over {\it mesh networks}, as inferred by the following discussion.

$\bullet$ {\bf Master/workers systems:}
Several papers employed acceleration (in the sense of Nesterov)  to solve      \eqref{eq:problem} over   such architectures; see, e.g., \cite{gorbunov2020recent,dAspremont2021AccelerationM,lin2015Catalyst,Rabbat_SlowMo,smith_fedLen_SPMag20} 
and references therein for a comprehensive description of state-of-the-art methods, including their application to federated learning systems and stochastic optimization. Given the focus on the paper, next we comment in details  works     exploring  the idea of statistical preconditioning to further decrease 
the communication complexity of  solving       \eqref{eq:problem}. 

  \texttt{DANE}  \cite{shamir2014communication} is a   mirror-descent type algorithm for $\eqref{eq:problem}$ with $r=0$ whereby   workers   perform a local data preconditioning via a suitably chosen Bregman divergence, and the master averages the solutions of the workers. For quadratic losses, \texttt{DANE} achieves communication complexity $\widetilde{\mathcal{O}}((\beta/\mu)^2\log1/\varepsilon)$. 
More recently, \cite{Fan2020} proposed \texttt{CEASE}, which achieves  \texttt{DANE}'s complexity  but  for nonquadratic losses and $r\neq 0$. Applying the  convergence analysis of mirror descent in   \cite{LuFreundNesterov2020} to   \texttt{CEASE}  
enhances   its rate 
to   $\widetilde{\mathcal{O}}((\beta/\mu)\log1/\varepsilon)$.   

 Further improvements are achievable employing   acceleration. Efforts include:  \texttt{DiSCO} \citep{zhang2015disco}, an inexact
damped Newton method coupled with a preconditioned conjugate gradient (to compute the Newton direction), which   achieves  communication complexity $\widetilde{\mathcal{O}}((1+\sqrt{\beta/\mu})\log1/\varepsilon)$ for self-concordant losses (and $r\neq 0$); \texttt{AIDE} \citep{reddi2016aide}, which uses the Catalyst framework  {\citep{lin2015Catalyst}}, matching the rate of  \texttt{DiSCO} for quadratic losses (and $r=0$); \texttt{DANE-HB} \citep{yuan2020convergence}, a variant of  \texttt{DANE} equipped with Heavy Ball momentum and matching for quadratic functions the communication complexity of \texttt{DiSCO} and \texttt{AIDE}; and \texttt{SPAG}\citep{hendrikx2020statistically}, a preconditioned direct accelerated method,    achieving for nonquadradic losses {\it asymptotically} the convergence rate   ${\mathcal{O}}((1-1/\sqrt{\beta/\mu})^k)$   ($k$ is the iteration index). 

None of above methods are implementable over  mesh networks, because they all  rely on the presence of  a master node.  Notice also that, although designed  for mesh networks,  our proposed method, \texttt{ACC-SONATA} 
compares favorably also with the aforementioned schemes (specifically designed for star networks) by   achieving communication complexity $\mathcal{O}(\sqrt{\beta/\mu}\log(\beta/\mu)\log1/\varepsilon)$ for nonquadratic losses.

  $\bullet$ {\bf Acceleration over mesh networks:} 
  Given the focus of this work, we discuss  next only (provably convergent) distributed  algorithms over mesh networks employing some form of acceleration--they are summarized in Table~\ref{table}. Although substantially different--some are primal  \citep{ye2020multi, ye2020decentralized,doi:10.1137/18M122902X,kovalev2020optimal,rogozin2020towards} others are dual or penalty-based \citep{scaman2017optimal,uribe2020dual,li2018sharp} methods, and applicable to special instances of \eqref{eq:problem} (mainly with $r=0$) and subject to special design constraints (e.g., positive semidefinite gossip matrix)--they all  achieve linear convergence rate, with  communication complexity scaling some with $\sqrt{\kappa_\ell}$ ($\kappa_\ell=L_{\textnormal mx}/\mu_{\textnormal mn}$ is the  ``local'' condition number)   and others  with  $\sqrt{\kappa}$    ($\kappa=L/\mu$ is the  condition number of $f$). Note that in general $\kappa  \ll \kappa_{\ell}$; hence the latter group is preferable to the former. 
   By using only gradient information of the local $f_i$'s, none of these methods can take advantage of function similarity.  This means that their rates still scale as $\sqrt{\kappa}$ no matter how small $\beta$ is (even $\beta=0$, i.e., all $f_i$'s identical), which is highly sub-optimal  when considering, e.g.,  ill-conditioned losses, and contrasts with the lower complexity bound $\mathcal{O}(\sqrt{\beta/\mu})$ (cf.~Sec.~\ref{sec:LB}).

To our knowledge, \texttt{Network-DANE} \citep{NetDane} and \texttt{SONATA} \citep{sun2019distributed} are the  only two  methods that leverage statistical similarity to enhance convergence over mesh networks, with the latter achieving communication complexity scaling with  $\widetilde{\mathcal{O}}(\beta/\mu)$ for nonquadratic losses  and $r\neq 0$.   These methods however are not accelerated, missing thus the more favorable scaling $\sqrt{\beta/\mu}$.
The proposed accelerated method, \texttt{ACC-SONATA},   fills exactly this gap. 
\vspace{-0.4cm}

 \section{SETUP AND BACKGROUND}\vspace{-0.1cm}
 We study \eqref{eq:problem} under the following assumptions.

\begin{assumption}\label{assump:class}
Given problem~\eqref{eq:problem},\vspace{-0.3cm}
\begin{enumerate}\itemsep-0.05cm
    \item[(i)]  $r:\mathbb{R}^d\to \mathbb{R}$ is a proper, closed, convex function; let $\textnormal{dom}\,r$ denote its domain;
  \item[(ii)]   Each $f_i$ is convex and twice differentiable over (an open set containing)    $\textnormal{dom}g$;
    \item[(iii)] $f$ is  $\mu$-strongly convex and $L$-smooth on  $\textnormal{dom}g$, with $0<\mu \leq L < \infty$. Define $\kappa=L/\mu$.\vspace{-0.1cm} 
\end{enumerate}\vspace{-0.2cm}
\end{assumption}
Note that (i)-(ii) implies that each $f_i$ is $\mu_i$-strongly convex and $L_i$-smooth, with $0\leq \mu_i \leq L_i < \infty$ (not each $f_i$ need to  be strongly convex); let $L_{\max} = \max_{i\in[m]} L_i,$, were we denoted $[m]\triangleq \{1,\ldots, m\}.$

Function similarity is captured by the following. 
\begin{definition}\label{assump:sim} Under Assumption~\ref{assump:class}, let $\beta\geq 0$ the smallest number such that  \vspace{-0.1cm}
$$ \max_{i\in[m]} \,\, \sup_{x\in \textnormal{dom}\,r} \norm{\nabla^2 f_i(x) - \nabla^2 f(x)} \leq \beta.\vspace{-0.2cm}
$$ 
\end{definition}
The smaller $\beta$, the more similar $f_i$'s are.  Note that the following bound holds for $\beta$: 
\vspace{-0.1cm}
$$\beta\leq \max_{i\in [m]}\max\big\{|L_i-\mu|,\,\, |\mu-L_i|\big\}.\vspace{-0.1cm}$$
The case of interest is of course when $1+\beta/\mu<L/\mu$, which is the typical situation of ill-conditioned $f$'s.

\noindent \textbf{Network  model:}
The   network of   agents is  an undirected,    graph $\Gh=(\Eg,\Vx)$, where $\Vx = [m]$ denotes the vertex set (the set of agents) while  $\Eg$ is the set of edges; $\{i,j\}\in\Eg$ if there is   a communication link between agent $i$ and agent $j$. For the sake of notation, we assume   $\{i,i\}\in\Eg$ for any $i\in[m].$
We  make the blanket assumption that the graph $\mathcal G$ is connected.

The distributed algorithms of interest employ gossip communications--each node   averages the values of its neighbors' variables. The weights of this averaging process (collected into a matrix $W\in \mathbb{R}^{m\times m}$)   satisfy the following standard assumptions.  
\begin{assumption}\label{assump:weight}    The matrix    ${W}\in \mathbb{R}^{m\times m}$ belongs to the class $W=P_M(\overline{{W}})$, for some $M\in \mathbb{N}_{++}$ and $\overline{{W}}\in \mathcal{R}^{m\times m}_+$, where $P_M$ is a polynomial with degree at most $M$  with $P_M(1)=1$, and $\overline{{W}}$  satisfies the following conditions: \vspace{-0.3cm}
\begin{itemize}\itemsep-0.1cm
\item[(i)] $\overline{W}$ is compliant with $\mathcal{G}$, that is its $(i,j)$-entries $\bar{w}_{ij}$ satisfies: $\bar{w}_{ii} > 0$, for all $i\in [m]$;  $\bar{w}_{ij} > 0$, if $(i,j) \in \mathcal{E}$; and $\bar{w}_{ij}=0$ otherwise;\item[(ii)]  ${1}^\top \overline{W} = {1}^\top$ and $\overline{W}  {1} = 1$ (doubly stochasticity).   \vspace{-0.7cm} 
\end{itemize}Define  $\rho\triangleq \norm{W-11^\top/m}<1$.
\end{assumption}

  The above class of weight matrices captures single ($M=1$) and multiple ($M>1$) rounds of communications per optimization step (notice that $P_M(1)=1$  is to ensure the doubly stochasticity of $W$ when $\overline{W}$ is so). 
Several rules have been proposed in the literature fulfilling Assumption \ref{assump:weight}, including    the Laplacian,  the Metropolis-Hasting, and the maximum-degree weights rules   as well as Chebyshev \citep{auzinger2011iterative,scaman2017optimal} or   Jacobi \citep{Berthier2020} polynomials-based accelerations.

\vspace{-0.1cm}
 \subsection{Lower Complexity Bounds over Mesh Networks under Similarity}\label{sec:LB}\vspace{-0.2cm}
 We informally  state here  lower communication complexity bounds over mesh networks for the class of problems \eqref{eq:problem} satisfying Assumptions~\ref{assump:class} and~\ref{assump:weight}, and certain distributed gossip algorithms   of interest (see Definition~\ref{app_def_oracle} in Appendix~\ref{App_LB}): 
 \begin{equation}\label{eq:LB} \Omega \left( \sqrt{\frac{\beta/\mu}{1-\rho}} \log \left( \frac{\mu\norm{x^\star}^2}{\varepsilon} \right) \right).\end{equation}
 The formal statement of this result can be found in the supplementary material (cf. Theorem~\ref{th_LB},  Appendix~\ref{App_LB}).
 The next section is devoted to the design of the  first distributed algorithm matching such a lower complexity bound (up to poly-log factors). 
 As anticipated, our scheme  hinges on the \texttt{SONATA} algorithm \citep{sun2019distributed}, which we recall next.  \vspace{-0.1cm}
  
\subsection{A Building Block:  SONATA Algorithm}\vspace{-0.1cm}
The instance of \texttt{SONATA} used in this work is summarized in Algorithm~\ref{alg:SONATA} (assumed to be applied to \eqref{eq:problem}). Each agent $i$ owns local copies $x_i$ of the shared optimization variable $x$ along with the auxiliary variable $y_i$ that is  a local proxy of $\nabla f$, which is not available at the agents' sides. In parallel and iteratively, agents  update their  $x$-variables, first solving in \texttt{(S.1)} a local approximation of \eqref{eq:problem} wherein $\tilde{f}_i(x;x_i^k)$ is a  surrogate of $f_i$ at $x_i^k$ and the linear term $y_i^k-\nabla f_i(x_i^k)$ is an estimate of $\sum_{j\neq i} \nabla f_j(x_i^k)$. This is followed by a communication step,  \texttt{(S.2)}, instrumental to enforce asymptotic consensus among the $x$-variables and track  $\nabla f$ via the $y$-ones.  

Several surrogate functions are feasible, see \cite{sun2019distributed}. Here, we focus on the following two: \vspace{-0.2cm}
\begin{align}
    \label{surrogates_f_i}
   & \widetilde{f}_i(x;x_i^k)= f_i(x)+\frac{\beta}{2}\|x-x_i^k\|^2;\\
   &\widetilde{f}_i(x;x_i^k)= f_i(x_i^k)+\inn{\nabla f_i(x_i^k)}{x - x_i^k}+\frac{L}{2}\|x-x_i^k\|^2.\label{surrogates_linearization}
\end{align} 
Note that the use of the linearization (\ref{surrogates_linearization}) corresponds to perform  at each agent's side a proximal gradient step; this is the typical update of the majority of existing distributed algorithms (as those in Table~\ref{table}). Such a choice does not permit to take advantage of function similarity, if any. In fact, when (\ref{surrogates_linearization}) is employed and a weight matrix $W$ satisfying Assumption~\ref{assump:weight} is used for the consensus and tracking steps, \texttt{SONATA} applied to \eqref{eq:problem} achieves an $\varepsilon$-solution (in terms of objective value) in $\widetilde{\mathcal{O}}\big(\kappa\, \frac{1}{1-\rho}\log1/\varepsilon\big)$ number of communications. Communication saving under    function similarity is provably achievable instead using  surrogate
(\ref{surrogates_f_i}), resulting in a  communication complexity of 
 $\widetilde{\mathcal{O}}\big(\frac{\beta}{\mu}\cdot \frac{1}{1-\rho}\log1/\varepsilon\big).$  This motivated us to  use \texttt{SONATA} as building block of our accelerated method aiming at exploiting function similarity.  
\vspace{-0.1cm}
 \begin{algorithm}[h]
\caption{\texttt{SONATA}$(\left\{ f_i \right\}_{i\in[m]},\, x^0, \, y^0,\, T)$}\label{alg:SONATA} 
 {\bf Input}: $\left\{ f_i(x) \right\}_{i\in[m]},\, \,r(x)$ [cf. \eqref{eq:problem}];
 
 \hspace{1.2cm} $x^0=(x_i^0)_{i\in [m]}$ 
 [init. points],
 
 \hspace{1.2cm} $y^0=(y_i^0)_{i\in [m]}$ [grad.-tracking init.],
 
 \hspace{1.2cm} $T>0$ 
 [\# iterations];\\
 {\bf Output}: $x^T=\big(x_i^T\big)_{i\in [m]},\,  \, y^T=\big( y_i^T\big)_{i\in [m]}$; 

\textbf{for} {$k=0,1,2,\ldots,T-1$} \textbf{do}

  \texttt{ (S.1) Local computations:} for all $i\in [m]$, \vspace{-0.2cm} \begin{align*}
      {x}_i^{k+1/2} = \argmin_{x\in \mathbb{R}^d} & \,\,\tilde{f}_i(x; x_i^k) \\& + \inn{y_i^k - \nabla f_i(x_i^k)}{x - x_i^k}+r(x);
  \end{align*}
 
 \texttt{ (S.2) Communications:}  for all $i\in [m]$,\vspace{-0.2cm}
   \begin{align*}
    & x_i^{k+1} = \sum_{j=1}^m w_{ij} x_j^{k+1/2}, \\
    & y_j^{k+1} = \sum_{j=1}^m w_{ij} \big( y_j^k + \nabla f_j(x_j^{k+1}) - \nabla f_j(x_j^k)\big).
   \end{align*} 
\textbf{end for}
\end{algorithm}
 
\vspace{-0.1cm}
 
 \section{ACCELERATED SONATA}\label{sec:alg-design}\vspace{-0.2cm}
 We are ready to introduce our main algorithm, Algorithm~\ref{alg:ACCSONATA}. At high level, the scheme can be interpreted as a successive application of \texttt{SONATA} for the inexact minimization of the function \vspace{-0.1cm} \begin{equation}\label{eq:u_k}
u_{k}(x) \triangleq \frac{1}{m}\sum_{i=1}^m  f_i^k(x) + r(x), 
\end{equation} with $f_i^k(x) = f_i(x) +  ({\delta}/{2})\|x-z_i^k\|^2$, wherein the $z$-variable in the quadratic term plays the role of the extrapolation \`{a} la Nesterov, to gain acceleration. The use of \texttt{SONATA} in the inner loop allows us to  take advantage of function similarity, if any, by choosing surrogates as in (\ref{surrogates_f_i}) and a suitable value for  $\delta>0$ (see Theorem \ref{thm:cata_net}).   Notice the warm restart of \texttt{SONATA} every $T$ iterations, with in particular the gradient tracking initialization unconventionally  chosen, as recommended by our convergence analysis.\vspace{-0.1cm} 


 \begin{algorithm}[!ht]
\caption{\texttt{Accelerated SONATA}}\label{alg:ACCSONATA} 
 {\bf Input}: $\beta$, $\mu$, $\delta>0$, $\alpha = \sqrt{\mu/(\mu+\delta)}$; 
 
 \hspace{1.1cm} $x_i^0=z_i^0 =z_i^{-1} =0$, $y_i^0 = \nabla f_i(x_i^0)$ \\
 {\bf Output}: $x^K = (x_i^K)_{i\in [m]}$ \\
\textbf{for} {$k=0,1,2,\ldots,K-1$} \textbf{do}\smallskip 

\,\,\,\texttt{Set:} \qquad $ f_i^k(x) = f_i(x) + \frac{\delta}{2}\norm{x-z_i^k}^2;$\medskip 
 
\,\,\texttt{(S.1) Inner loop via  \texttt{SONATA}:} 
    \begin{align*} & \big(x^{k+1},\,  y^{k+1}\big) = \\
    & \texttt{SONATA} \Big( \left\{ f_i^k \right\}_{i\in[m]},\,  x^k, \, y^k+ \delta\,  \mybrace{z^{k-1} - z^k}, \, T \Big);
    \end{align*}
    \,\,\texttt{(S.2)     Extrapolation step:} 
\[ z_i^{k+1} = x_i^{k+1} + \frac{1-\alpha}{1+\alpha} \,(x_i^{k+1} - x_i^k),\quad \forall i \in[m].
\]
\textbf{end for}
\end{algorithm}


\textbf{Algorithm rationale:}   The genesis of the algorithmic design can be traced back to the idea of acceleration of a centralized inexact proximal method (outer loop) (see, e.g., \cite{dAspremont2021AccelerationM}), whose proximal subproblems     are approximately solved  in a  distributed fashion via the \texttt{SONATA} algorithm (inner loop),  satisfying a suitable notion of inexactness (defined in this paper) for proximal operations.      
In fact, assuming  one can absorb consensus errors on  the agents' variables $x_i$'s and  momentum vectors $z_i$'s into such a   criterion of solution approximation, we can approximate \texttt{(S.1)} and \texttt{(S.2)} as\vspace{-0.1cm}
  \begin{align*}
    &  x_i^k \approx \overline{x}^k\triangleq \frac{1}{m}\sum_{i=1}^m x_i^k\quad \text{and}\quad  z_i^k \approx \bar{z}^k\triangleq\frac{1}{m}\sum_{i=1}^m z_i^k,\\
   & \texttt{(S.1)$^\prime$:}\quad \bar{x}^{k+1} \approx   \argmin_{x\in \mathbb{R}^d}   u(x)+   \frac{\delta}{2} \norm{x-\bar{z}^k}^2,\\
    & \texttt{(S.2)$^\prime$:}\quad \bar{z}^{k+1} = \bar{x}^{k+1} + \frac{1-\alpha}{1+\alpha} \,(\bar{x}^{k+1} - \bar{x}^k),
  \end{align*}
  where we used the fact that the minimization of $u_k$ in (\ref{eq:u_k}) and that of the function in \texttt{(S.1)$^\prime$} have the same solution. The dynamics above are a resemble of an inexact proximal acceleration \citep{dAspremont2021AccelerationM,lin2015Catalyst,lin2015Catalyst-journal}. 
  
  \textbf{Challenges:} Despite the above  connection,   existing convergence analyses of centralized accelerated methods  break down when applied to \texttt{ACC-SONATA}. 
  Specifically,    (i) the notions of approximate solutions for proximal problems as in  \cite{dAspremont2021AccelerationM,lin2015Catalyst-journal} cannot be satisfied here,  because of the aforementioned  consensus errors,  let alone their practical verification in a distributed setting  and in the presence of the nonsmooth function $r$;  and (ii)  the potential functions used therein are not adequate to certify   linear convergence  of the outer loop of \texttt{ACC-SONATA} at the desired accelerated rate,   they do not capture  unavoidable  consensus and gradient tracking errors coming out of the inexact, distributed  minimization of $u_k$ in (\ref{eq:u_k}) via \texttt{SONATA}. 
  Furthermore,    the  convergence proof of \texttt{SONATA}  as in  \cite{sun2019distributed} is not directly applicable to study the inner loop, due to the  unconventional restart of the gradient tracking variables. Also,   R-linear convergence of the objective-value gap and consensus/tracking errors  therein  seems no longer  sufficient  to provably obtain acceleration of  the outer loop.     
 Our convergence analysis addresses  these challenges--we refer to Appendix~\ref{app_proof_Th4}
 for the complete proof (and Appendix~\ref{sec:app_sketch} for  a sketch). 
  \vspace{-0.3cm}
  


\subsection{Convergence Analysis}\vspace{-0.1cm}
Communication complexity of \texttt{Acc-SONATA} is established in Theorem~\ref{thm:cata_net}  and Theorem~\ref{crl:cata_net_lin} below, pertaining  to  the use in the inner algorithm \texttt{SONATA} of the surrogates \eqref{surrogates_f_i} and  (\ref{surrogates_linearization}), respectively. 
The explicit expression of the constants hidden in the big-O notation can be found in the supplementary material.   

\begin{theorem}\label{thm:cata_net}
Consider problem \eqref{eq:problem} under Assumption~\ref{assump:class}, with optimal value function $u^\star$ and $\beta> \mu$ w.l.o.g.. 
Let $\{x^k \triangleq (x_i^k)_{i\in [m]}\}$ be the sequence generated by \textnormal{\texttt{ACC-SONATA}} under Assumption~\ref{assump:weight}, with\vspace{-0.1cm} \begin{equation}\label{eq:thm_rho_sim}
\rho \leq  \mathcal{O} \left(\left(1+\frac{\kappa-1}{\beta/\mu}\right)^{-2}\right),\vspace{-0.1cm}\end{equation}   and the following tuning: \vspace{-0.1cm}$$  \delta=\beta-\mu,\quad 
  T = \mathcal{O} \mybrace{ \log {\beta}/{\mu}},$$
  and agents' surrogate functions \eqref{surrogates_f_i} in \textnormal{\texttt{SONATA}}. Define  $\bar{x}^k \triangleq \frac{1}{m}\sum_{i=1}^m x_i^k$ and  the optimality gap\vspace{-0.1cm} 
  \begin{align}\label{eq:opt_gap_def}
  \Delta(x^k)\triangleq \max\mybrace{\frac{1}{m}\sum_{i=1}^m u(x_i^k) - u^\star,\frac{1}{m}\sum_{i=1}^m \norm{x_i^k - \bar{x}^k}^2 }.
  \end{align}
 Then, there holds \vspace{-0.2cm}
\begin{equation*}   \Delta(x^k) = \mathcal{O} \mybrace{\mybrace{1-c\,\sqrt{ {\mu}/{\beta}}}^k},
\end{equation*} where $c\in (0,1)$ is some universal constant. 
Therefore,    $\Delta(x^K) \leq \varepsilon$, $\varepsilon>0$, in   
\begin{equation}\label{eq:num_comm_sim}
\mathcal{O}\left(\sqrt{ \frac{\beta}{\mu}}\cdot T\cdot \log  \frac{1}{\varepsilon}\right)
\end{equation}
total  (inner plus outer loop) communication steps. 
\end{theorem} 

Note that  (\ref{eq:num_comm_sim}) states linear convergence with  optimal dependence on  $\beta/\mu$, up to the log-factor $T=\log(\beta/\mu)$. This is achieved under \eqref{eq:thm_rho_sim}, which requires
the 
network to be sufficiently connected. 
If the network is not part of the design, \eqref{eq:thm_rho_sim} might not be satisfied by the topology under consideration. Still,  \eqref{eq:thm_rho_sim} can be enforced by running multiple communication rounds        per iteration (computation steps)   
in the inner loop \texttt{SONATA}.  
Specifically, let $\bar{\rho}=\|\overline{W}-11^\top/m\|$ be the connectivity  of the   graph associated with a given weight matrix $\overline{W}$ (satisfying Assumption~\ref{assump:weight}); suppose we run $M$ steps  of communications per iteration (computation) in Step  \texttt{(S.2)} of the \texttt{SONATA} algorithm,   each time using the weight matrix $\overline{W}$.  This yields an effective network with matrix $W=\overline{W}^M$ ($P_M(x)=x^M$) and improved connectivity $\rho = \bar{\rho}^M$. One can then choose $M$ so that  $\bar{\rho}^M$ satisfies \eqref{eq:thm_rho_sim}, resulting in $M=\mathcal{O}\big(\log(1+(\kappa-1)/(\beta/\mu))/\log(1/\bar{\rho})\big)=\mathcal{O}\big(\log(1+(\kappa-1)/(\beta/\mu))/(1-\bar{\rho})\big)$ rounds of communications. The dependence on $\bar{\rho}$ can be improved   leveraging    
  Chebyshev  polynomials of   the gossip matrix $\overline{W}$ (further assumed to be symmetric) of order at most $M$, that is, $W=P_M(\overline{W})$, where $P_M$ are   Chebyshev  polynomials. At the agents' side, the updates  $P_M(\overline{W})$ are implemented via a  shift-register gossip protocol, resulting in   time-varying weights $w_{ij}$'s in $\texttt{(S.2)}$ of the \texttt{SONATA} algorithm--see \cite{Berthier2020} for details.  
  It is not difficult to show that $M=\mathcal{O}\big(\log(1+(\kappa-1)/(\beta/\mu))/(\sqrt{1-\bar{\rho}})\big)$     suffices to satisfy \eqref{eq:thm_rho_sim}, yielding an overall communication complexity
\begin{equation}\label{eq:total_comm_cost_acceleration}
\mathcal{O} \mybrace{\,\sqrt{\frac{{\beta}/{\mu}}{\ {1-\bar{\rho}}}}\,  \log \bigg(1+\frac{\kappa - 1}{\beta/\mu}\bigg) \, \log\left(\frac{\beta}{\mu}\right)\,  \log  \frac{1}{\varepsilon}}.
\end{equation}
 This   matches the lower complexity bound in \eqref{eq:LB} (cf.  Theorem~\ref{th_LB} in Appendix~\ref{App_LB}) up to poly-log factors.

 Consistently with the behaviour of \texttt{SONATA}, if surrogates in the form (\ref{surrogates_linearization}) are used by the agents, 
 acceleration is still achieved, but with a linear rate scaling with $\sqrt{\kappa}$ rather than $\sqrt{\beta/\mu}$. This is formalized next.

\begin{theorem}\label{crl:cata_net_lin}
Consider problem \eqref{eq:problem} under Assumption~\ref{assump:class}, with optimal value function $u^\star$ and $\kappa > 1$ w.l.o.g.. 
Let $\{x^k \triangleq (x_i^k)_{i\in [m]}\}$ be the sequence generated by \textnormal{\texttt{ACC-SONATA}} under Assumption~\ref{assump:weight}, with  \begin{equation}\label{eq:thm_rho}
\rho \leq   \mathcal{O} \left(\left(2 +\frac{\beta/\mu-1}{\kappa}\right)^{-2}\right),\vspace{-0.1cm}\end{equation}  and the following tuning:  $$ \delta=L-\mu,\quad 
  T = \mathcal{O} \mybrace{ \log \kappa},$$
and agents' surrogate functions \eqref{surrogates_linearization} in \textnormal{\texttt{SONATA}}. 
 Then, there holds \vspace{-0.2cm}
\begin{equation}   \Delta(x^k) = \mathcal{O} \mybrace{\mybrace{1-c\,\frac{1}{\sqrt{\kappa}}}^k},
\end{equation} where $c\in (0,1)$ is some universal constant. 
Therefore,    $\Delta(x^K) \leq \varepsilon$, $\varepsilon>0$, in   
\begin{equation}\label{eq:num_comm}
\mathcal{O}\left(\sqrt{\kappa}\,T \log  \frac{1}{\varepsilon}\right)
\end{equation}
total  (inner plus outer loop) communication steps.
\end{theorem}

Enforcing \eqref{eq:thm_rho} via multiple rounds of communications based on  Chebyshev polynomials (applied to a symmetric matrix $\overline{W}$, with $\bar{\rho}=\|\overline{W}-11^\top/m\|$), the   communication complexity   in Theorem~\ref{crl:cata_net_lin} becomes 
\begin{equation}\label{eq:total_comm_cost_acceleration_linearization}
\mathcal{O} \mybrace{\,\sqrt{\frac{\kappa}{\ {1-\bar{\rho}}}}\,  \log \bigg(2+\frac{\beta/\mu - 1}{\kappa}\bigg) \, \log\left(\kappa\right)\,  \log  \frac{1}{\varepsilon}}.
\end{equation}
Comparing (\ref{eq:total_comm_cost_acceleration}) with  (\ref{eq:total_comm_cost_acceleration_linearization}) shows that, if properly exploited,  function similarity provably leads to communication saving. This calls for the use of the surrogate \eqref{surrogates_f_i} over \eqref{surrogates_linearization}; and thus it comes generally  at the cost of solving   computationally more demanding subproblems  at   agents' sides.   This is common to all the methods (including centralized) exploiting similarity and discussed in Sec.~\ref{sec:related_works}, and  seems an unavoidable tradeoff. These methods are   in fact designed with the goal of saving communications, at the cost of more computations. 

\textbf{\texttt{Inexact ACC-SONATA}:} To alleviate the computation cost of solving agent's subproblems with surrogate \eqref{surrogates_f_i} when a closed form solution is not available, in
Appendix~\ref{app_inexact}, we   discuss how to modify  \texttt{ACC-SONATA}   to accommodate    inexact  solutions of   agents' subproblems in Step \texttt{(S.1)} of  \texttt{SONATA}. We defer to the appendix for details; here we only point out that, by carefully choosing the {  inexact  criterion} for solving  approximately  the local optimization subproblems, the communication complexity of the resulting inexact \texttt{ACC-SONATA}, termed \texttt{Inexact ACC-SONATA} (Algorithm~\ref{alg:inexact_SF}), matches that of \texttt{ACC-SONATA} as in   \eqref{eq:num_comm_sim} (see Theorem~\ref{thm:ine_cata_f}).  We also study  the computational complexity of \texttt{Inexact-ACC-SONATA} (see Theorem~\ref{thm:ine_cata_f_cp}). For instance, if each agent's subproblem with surrogate \eqref{surrogates_f_i} is solved (up to a suitably chosen accuracy)  via  accelerated proximal gradient, \texttt{Inexact ACC-SONATA} reaches an $\varepsilon$-solution  of (P) after\begin{equation}\label{eq:compexity_acc_M_main_text}
\tilde{\mathcal{O}}\left( \sqrt{1+ \frac{\kappa+\beta/\mu}{ 2\beta/\mu - 1}}\,\cdot \frac{\beta}{\mu} \cdot \mybrace{\log  \frac{1}{\varepsilon}}^2 \right)
\end{equation}
  total gradient evaluations/agent, where $\tilde{\mathcal{O}}$ hides log-factors. On the other hand, if    surrogates \eqref{surrogates_linearization} are  used in the agents' subproblem, the total computation complexity of \texttt{Inexact ACC-SONATA} is still given by Theorem~\ref{crl:cata_net_lin}, and thus reads  $\tilde{\mathcal{O}}(\sqrt{\kappa}\log (1/\varepsilon))$, which might be more favorable than (\ref{eq:compexity_acc_M_main_text}). 
 
 Quite interestingly, the proposed accelerated framework offers the flexibility, within the same algorithm, to privilege computation or communication savings, based upon the choice of the right surrogate function, achieving (up to poly-log factors) either optimal computation or communication complexity (under similarity).

\textbf{\texttt{ACC-SONATA} on star-networks:} 
Although \texttt{ACC-} \texttt{SONATA} has been designed specifically for mesh networks, it readily applies to master/workers architectures;   details can be found in the supplementary material.  Here we only remark that a direct application of Theorem~\ref{thm:cata_net} and Theorem~\ref{crl:cata_net_lin} leads to the following  communication complexity to solve \eqref{eq:problem} over master/workers architectures \vspace{-0.1cm} $$\mathcal{O}\left(\sqrt{\frac{\beta}{\mu}}\log\bigg(\frac{\beta}{\mu}\bigg)\log\frac{1}{\varepsilon}\right)\,\, \text{and}\,\,  \mathcal{O}\left(\sqrt{\kappa}\log(\kappa)\log\frac{1}{\varepsilon}\right),$$ 
respectively. 
Quite interestingly, these rates compare favorably with those of the centralized methods discussed in Sec.~\ref{sec:related_works}. \vspace{-0.2cm}

\begin{figure*}[ht]\vspace{-0.3cm}
\centering{\includegraphics[width=5.4cm, height=4.6cm]{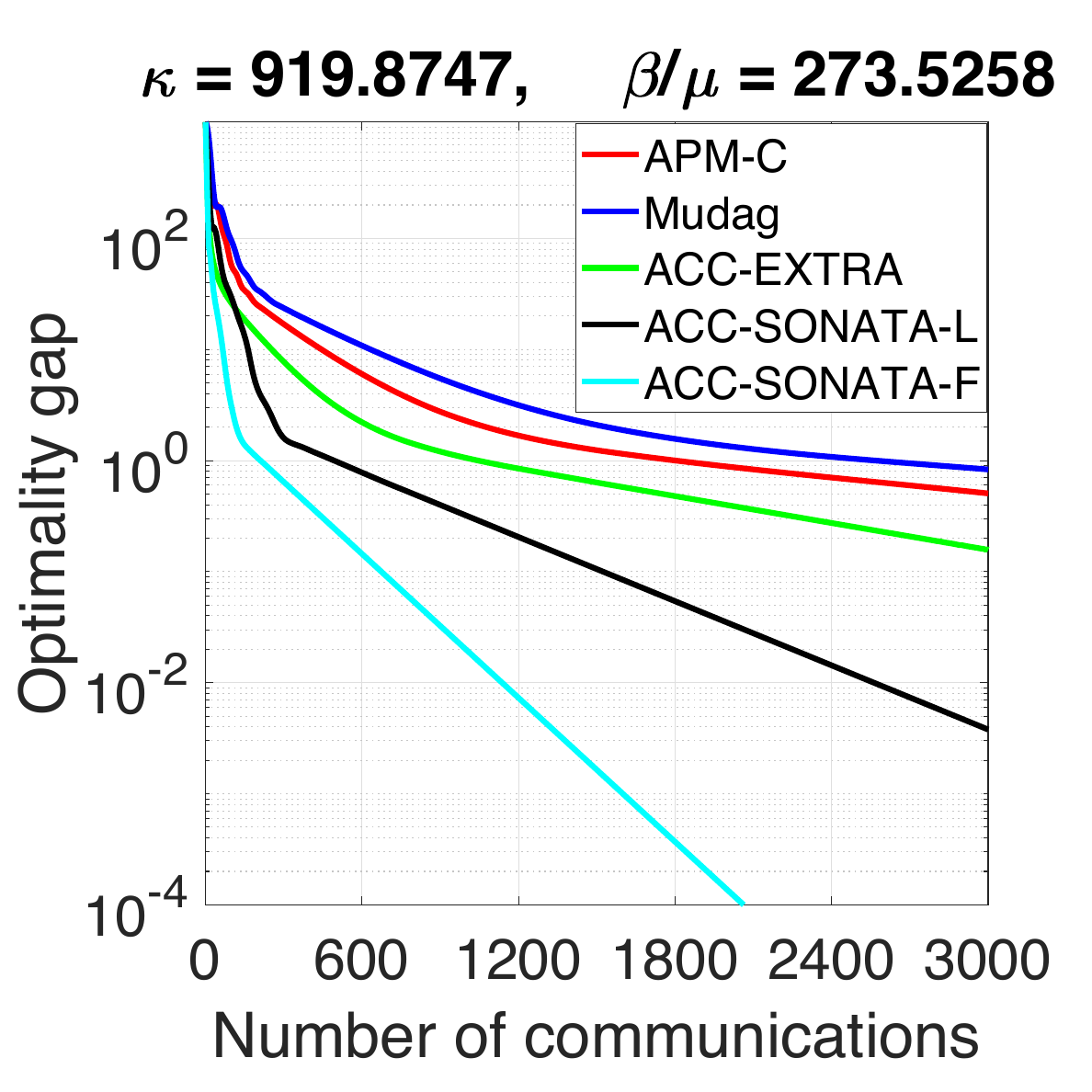}}\hspace{-.1cm}
\centering{\includegraphics[width=5.2cm, height=4.4cm]{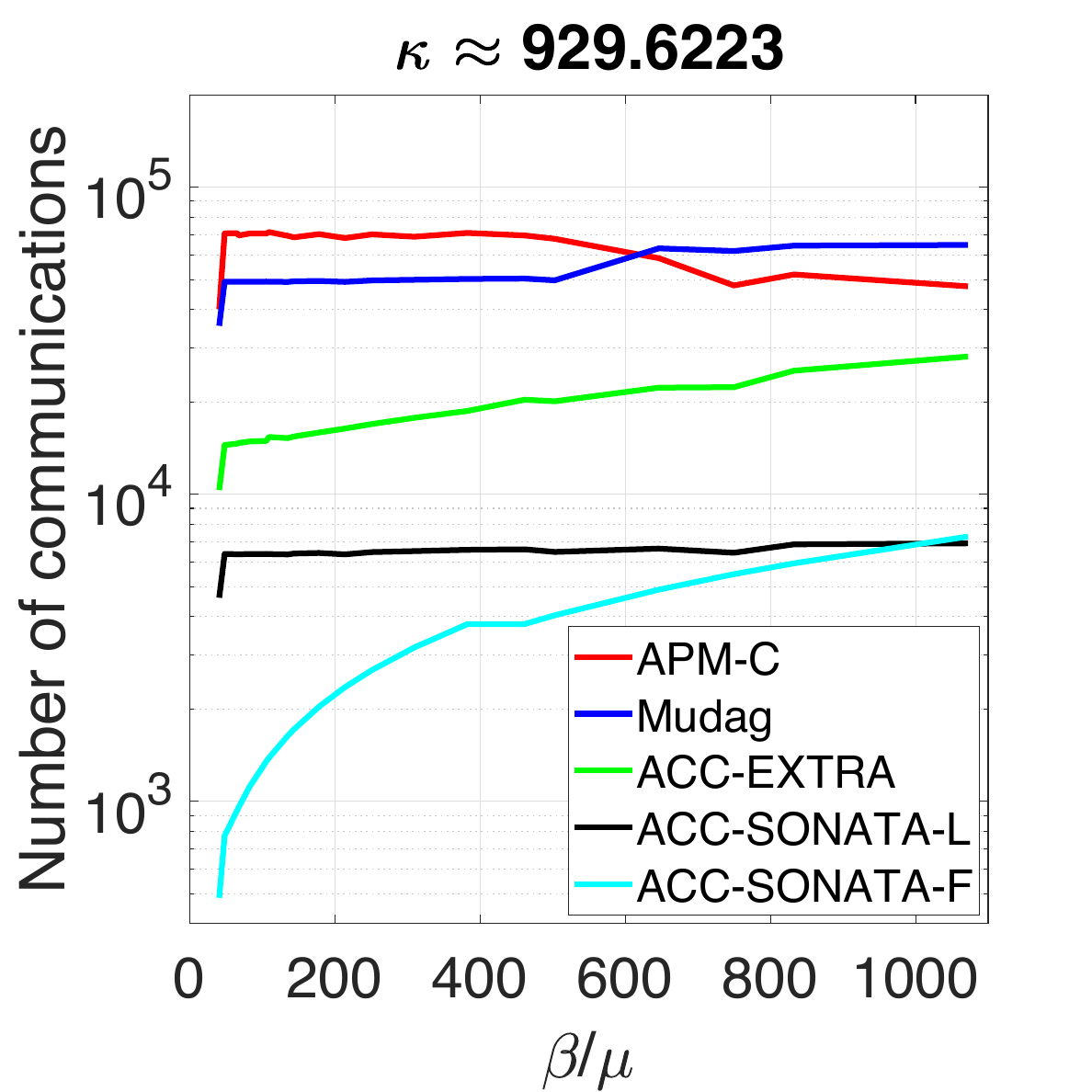}}
\centering{\includegraphics[width=5.2cm, height=4.4cm]{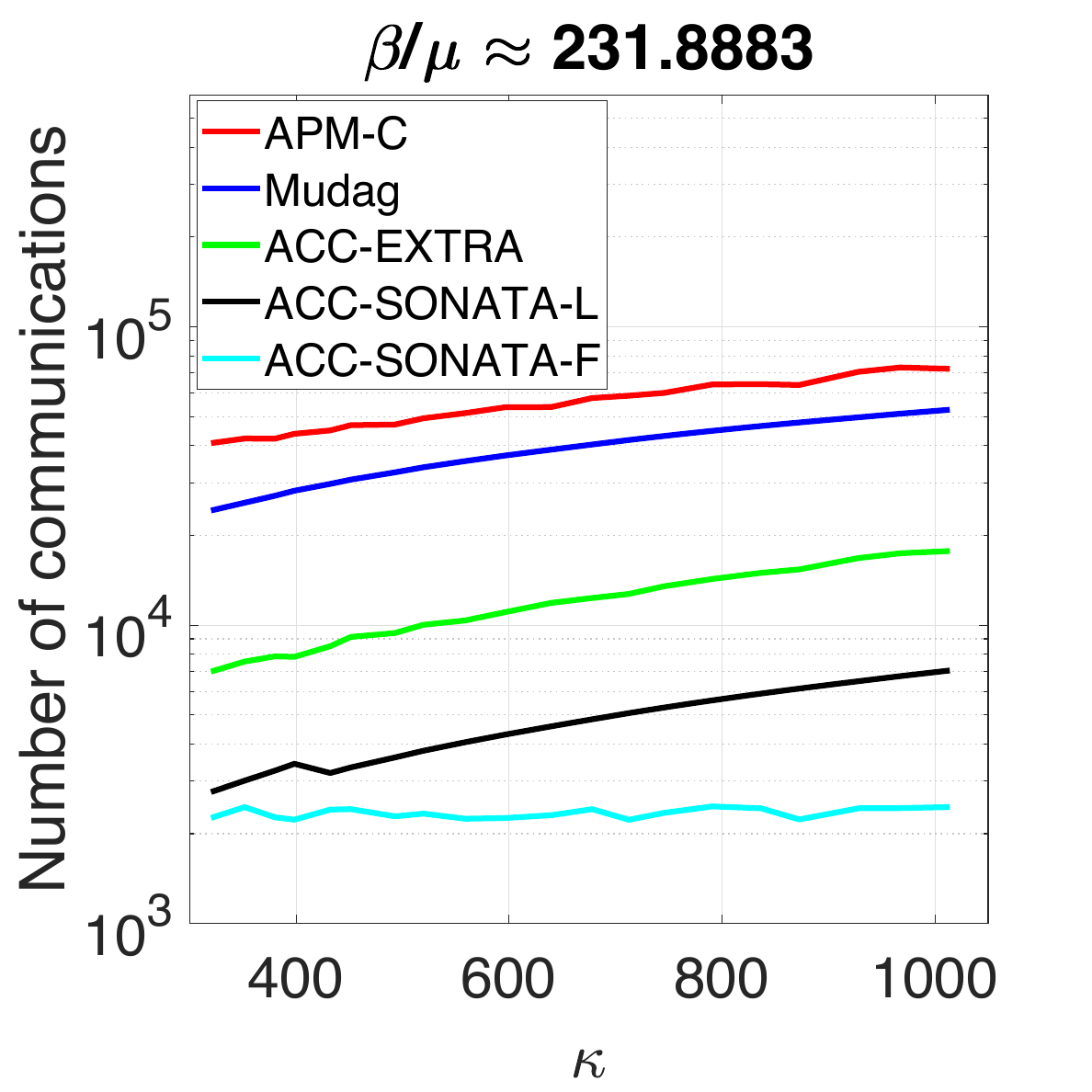}}\vspace{-0.2cm}

\caption{\small Comparison of distributed accelerated algorithms on ridge regression (synthetic data).     \textbf{(left panel)}:  optimality gap  versus     total number of communications, for given $\beta/\mu$ and $\kappa$; \textbf{(mid  panel)}:     number of communications  to reach a precision of $10^{-4}$ versus $\beta/\mu$, for fixed $\kappa$; \textbf{(right  panel)}:  same quantity versus $\kappa$, for fixed $\beta/\mu$. }
\label{fig:synthetic_data}
\end{figure*}

\begin{figure*}[ht] 
\centering{\includegraphics[width=5.4cm, height=4.5cm]{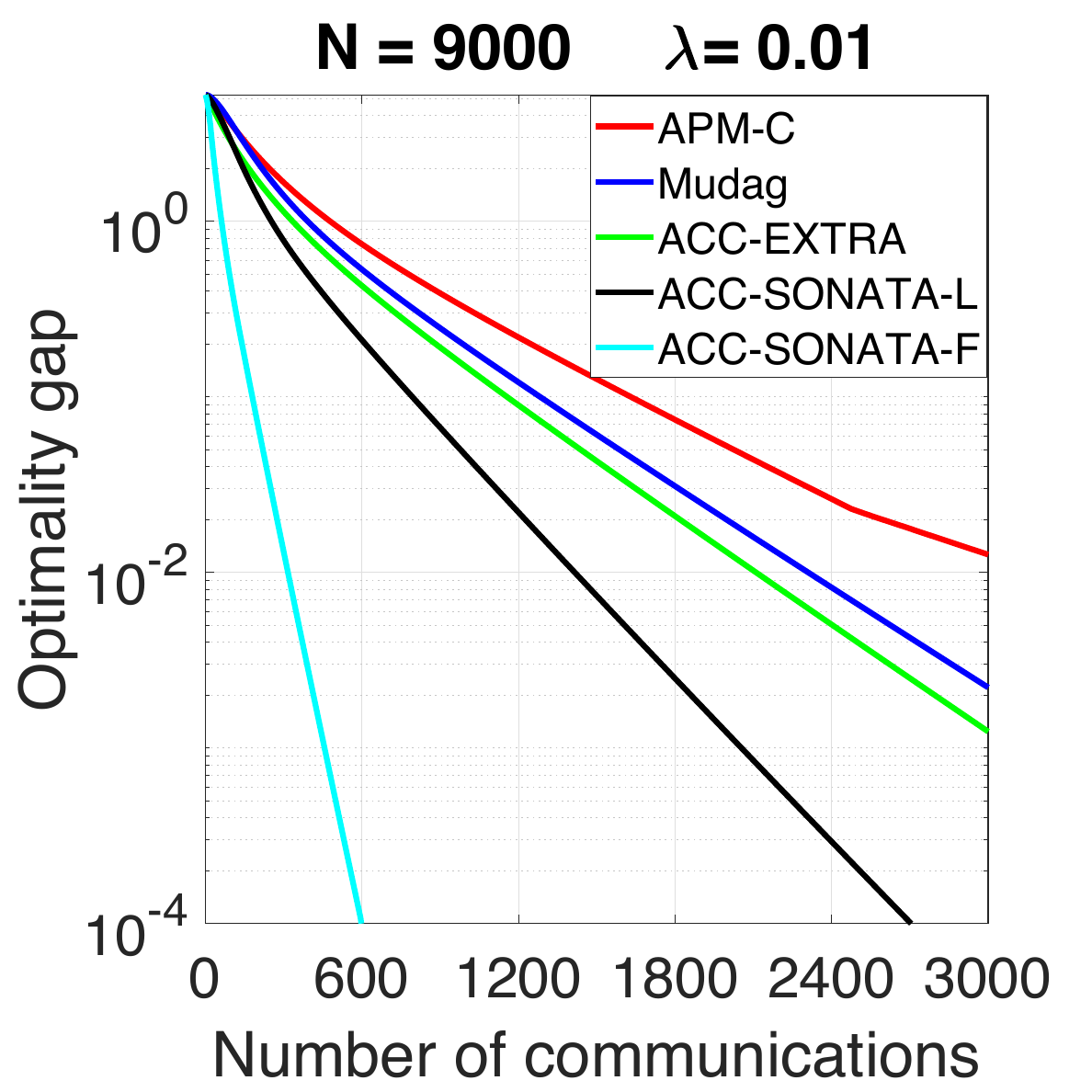}}\hspace{-.1cm}
\centering{\includegraphics[width=5.4cm, height=4.5cm]{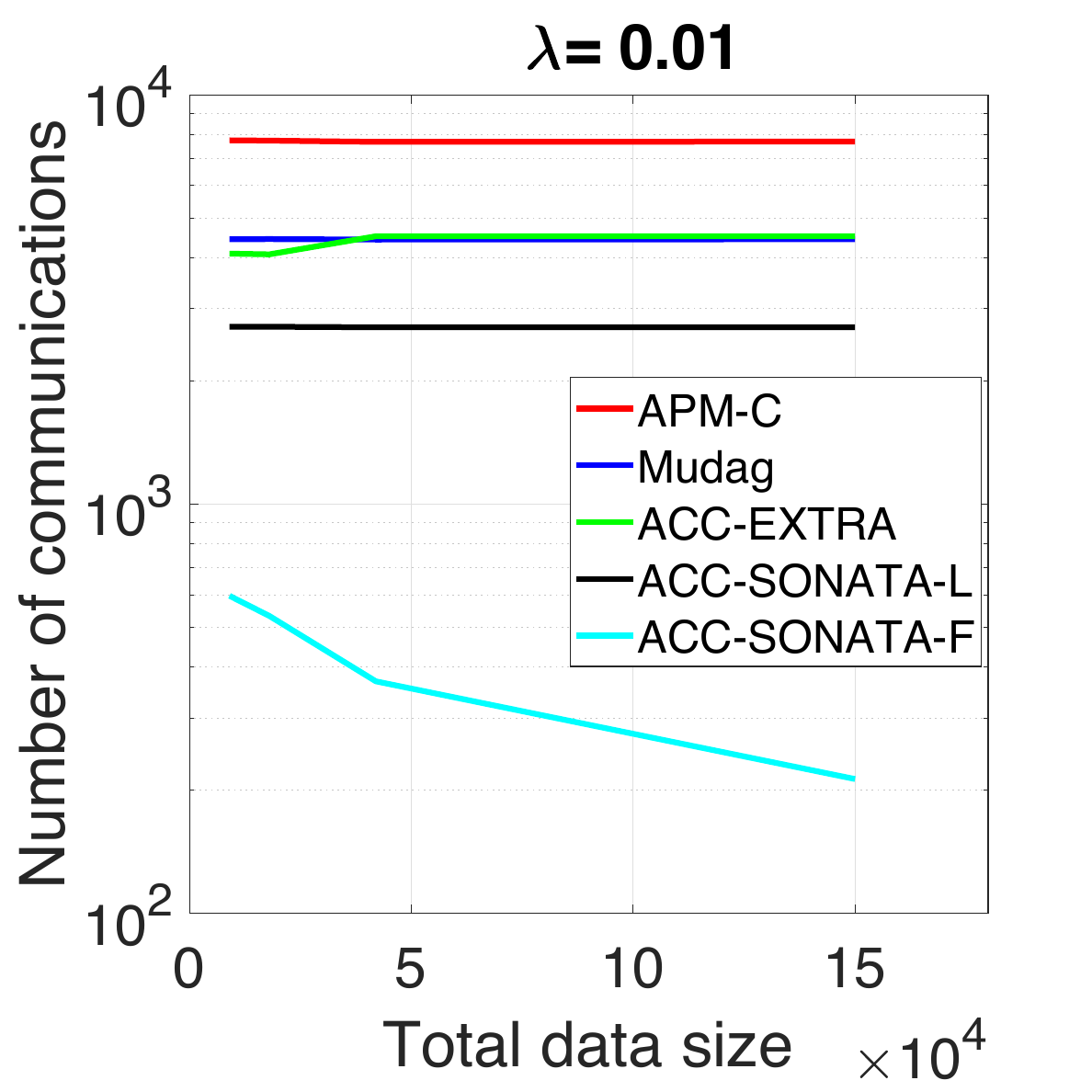}}
\vspace{-0.2cm}

\caption{\small Comparison of distributed accelerated algorithms on hinge loss minimization (\texttt{COV1} dataset).     \textbf{(left panel)}:  optimality gap  versus     total number of communications; \textbf{(right  panel)}:     number of communications  to reach a precision of $10^{-4}$ versus (total) sample. }
\label{fig:real_data}\vspace{-0.1cm}
\end{figure*}

\section{NUMERICAL RESULTS}\label{sec:num_result}\vspace{-0.2cm}
 We present  numerical results  on synthetic and real data,  corroborating our complexity analysis (Theorems~\ref{thm:cata_net}  and~\ref{crl:cata_net_lin}).  Additional experiments on different problem classes and data sets as well as including more algorithms  are reported in the supplementary material. 
 
 \textbf{1) Ridge regression:} Our first experiment concerns a ridge regression problem over a network of agents, modeled as a Erdos-Renyi graph  with   $m = 30$ nodes and  edge probability $p = 0.5$. The problem is an instance of \eqref{eq:problem} with $f_i(x)=1/(2n) \|A_i x-b_i\|^2+\lambda \|x\|^2$ [agent $i$ owns data $A_i \in \mathbb{R}^{n\times d},\,b_i\in \mathbb{R}^n$] and $r=0$. Data are generated as follows  \citep{sun2019distributed}:   Each row of ${A}_i$ is i.i.d, drawn from   $\mathcal{N}({0},{\Sigma})$, where  ${\Sigma} = \sum_{j=1}^{d}\lambda_{j}{u_j}{u_j}^{\top}$. The $\lambda_j$'s are uniformly distributed in [$\mu_0$, $L_0$], with $\mu_0=1$ and $L_0=1000$, and $\mathbf{u}_{1}, \ldots, \mathbf{u}_d$ are obtained via the QR decomposition of a $d\times d$ random  matrix with standard Gaussian i.i.d elements. We set ${b}_i = {A}_i {x}^{\star} + {w}_i$, where ${x}^{\star}\sim \mathcal{N}(5\cdot {1}_d,  {I})$ is the ground truth   and  ${w}_i\sim\mathcal{N}({0}, 0.1\cdot{I})$ is the additive noise ($1_d$ is the $d$-dimensional vector of all ones).  
 
 \textbf{Algorithms:} We simulated two instances of  \texttt{ACC-} \texttt{SONATA}, corresponding to the   choices of   the surrogates \eqref{surrogates_f_i} and \eqref{surrogates_linearization}  in  \texttt{SONATA} (inner-loop); we termed them as \texttt{ACC-SONATA-F}  and \texttt{ACC-SONATA-L}, respectively (\texttt{F} stands for ``full local function'' while \texttt{L} for ``linearization''). The solution of the agents' subproblems solved in \texttt{ACC-SONATA-F}  and \texttt{ACC-SONATA-L} is computed in closed form. According to   Theorems~\ref{thm:cata_net} and~\ref{crl:cata_net_lin}, \texttt{ACC-SONATA-F} is expected to outperform first-order methods, including \texttt{ACC-SONATA-L}, whenever $1+\beta/\mu<\kappa$, while  \texttt{ACC-SONATA-L} should be competitive otherwise. 
  The free parameters of these two instances are tuned as suggested by the theory, with  $T=\lceil\log(\beta/\mu)\rceil$ for \texttt{ACC-SONATA-F} and $T=\lceil\log(\kappa)\rceil$ for \texttt{ACC-SONATA-L}, where $L$ and $\mu$ are estimated by the data (quadratic function) and so $\beta$ using Definition~\ref{assump:sim}. 
  The weight matrix $W$ according to the Metropolis-Hasting rule.  We compare our algorithms with the following, widely tested in the literature (Table~\ref{table}): $\texttt{APM-C}$,   $\texttt{Mudag}$,   and $\texttt{ACC-EXTRA}$. The tuning of these schemes follows the recommendations as in their respective papers. In particular, $\texttt{Mudag}$ and $\texttt{APM-C}$ require the gossip matrix to be positive definite, so we set   $(W+I)/2$, with $W$ being the matrix used  in the other algorithms.  
  
  In   Fig.~\ref{fig:real_data}(\textbf{left-panel}), we fix $\beta/\mu<\kappa$, and plot the optimality gap 
   $\frac{1}{m}\sum_{i=1}^m\|x_{i}^k-{x}_{\textnormal{rg}}\|^2$ versus the total number of communications, for each of the algorithms, where  ${x}_{\textnormal{rg}}$ is the solution of the ridge regression problem (computed in closed form). All the schemes achieve linear convergence. As predicted, \texttt{ACC-SONATA-F} outperforms the other accelerate methods that do not take advantage of function similarity. Quite interestingly, \texttt{ACC-SONATA-L} compares quite favorably also with  directed acceleration methods, such as \texttt{Mudag} and \texttt{AMP-C}, while sharing similar computational costs.

  To investigate the impact of $\kappa$ and $\beta/\mu$ on the convergence rate of the algorithms, in the    next experiment we consider the following two scenarios: 
  
  \textbf{(i) Changing $\beta/\mu$ with (almost) fixed $\kappa$}: We generate instances of  ridge regression   with decreasing $\beta$ and (almost) fixed $\kappa$, setting $\lambda=0$ and increasing the local sample size $n$ (20 values) within $[100, 40000]$; the empirical $\kappa$ remains approximately constant close to  the nominal value  $L_0/\mu_0=1000$.  Fig.~\ref{fig:real_data}({\bf mid-panel}) captures this scenario,  we plot   the number of communications to drive the optimality gap below  $10^{-4}$ versus $\beta/\mu$,  and  $\kappa\approx 930$.
  
 \textbf{(ii) Changing $\kappa$ with fixed $\beta/\mu$}: We generate a sequence of instances of   ridge regression with  varying $\kappa$ (acting on $\lambda$) while keeping $\beta/\mu$ constant by changing the local sample size $n$ to compensate for the variations of  $\mu$ (due to $\lambda$).       Fig.~\ref{fig:real_data}({\bf right-panel}) plots the   number of communications to drive the optimality gap below  $10^{-4}$ versus $\kappa$, for $\beta/\mu\approx 232$.

The following comments are in order. First, the \textbf{mid-panel} confirms what   predicted by  Theorem~\ref{thm:cata_net}:  the number of communications of  \texttt{ACC-SONATA-F} scales roughly with $\sqrt{\beta/\mu}$ while that  of first-order distributed schemes is fairly invariant with $\beta/\mu$. This is because methods using only gradient information (including \texttt{ACC-SONATA-L}) cannot benefit from statistical similarity.   On the other hand, the \textbf{right-panel} shows that  communication complexity    of   accelerated first-order methods (including \texttt{ACC-SONATA-L}, as stated in Theorem~\ref{crl:cata_net_lin}) deteriorates  when   $\kappa$ grows whereas that of \texttt{ACC-SONATA-F} is (almost) invariant. It is interesting to remark that \texttt{ACC-SONATA-F} remains competitive, outperforming the other schemes,  even when $\beta/\mu\approx \kappa$.

\textbf{2) Hinge loss minimization on real data:} We consider the minimization problem $\min_{x\in \mathbb{R}^d} \frac{1}{N}\sum_{i=1}^{N} l(y_{i}\langle x, z_i\rangle) + \frac{\lambda}{2}\norm{x}^2$ over the same network of agents considered in the previous example, where $\ell$ is the smooth hinge loss as in \cite{shamir2014communication}. We experimented on the data set  \texttt{COV1} (see, e.g., \cite{Shalev-Shwartz13}). 
The tuning of the algorithm is the same a described in the previous experiment, based upon estimation of the quantities  $\mu$, $L$, $L_\textnormal{mx}$, $\mu_\textnormal{mn}$, and $\beta$  from the data--see supplementary material for details. We notice that $\mu\approx \mu_{\textnormal{mn}}\approx \lambda$. 
The solution of the agents' subproblems in \texttt{ACC-SONATA-F}  is estimated up to the accuracy    $10^{-10}$ by running the  gradient algorithm with  step-size equal to 0.03.

Fig.~\ref{fig:real_data} (\textbf{left-panel}) plots the optimality gap  $\frac{1}{m}\sum_{i=1}^m\|x_{i}^k-{x}_{\textnormal{op}}\|^2$ versus the total number of communications, for each of the algorithms, where  ${x}_{\textnormal{op}}$ is an estimate of  the solution of the   problem (obtained running the centralized gradient algorithm). The results show that the proposed methods are competitive also on real data, with \texttt{ACC-SONATA-F} outperforming the others, when enough samples are present at the agents' sides.  
The \textbf{right-panel} plots the number of communications to reach an accuracy of $10^{-4}$ versus the total sample size (since $m$ is fixed, the local samples size varies). This corresponds to decrease $\beta$ while keeping $\mu$ (roughly) constant. 
As predicted, we observe that the communication saving experienced by  \texttt{ACC-SONATA-F} improves with the local sample size, as the method takes advantage of the   local function structure, a feature that first-order methods are lacking.

\subsubsection*{Acknowledgments}
We are   grateful to the five anonymous Referees for their thoughtful and constructive comments, which helped  improve the quality of the paper.

The work of Tian, Scutari, and Cao has been supported by the Army Research Office (ARO)  under the grant No. W911NF1810238, and the Office of Naval Research (ONR) under the grant No. N00014-21-1-2673.
The work of   Gasnikov was supported by a grant for research centers in the field of  artificial intelligence, provided by the Analytical Center for the  Government of the Russian Federation in accordance with the subsidy  agreement (agreement identifier 000000D730321P5Q0002) and the agreement  with the Ivannikov Institute for System Programming of the Russian  Academy of Sciences dated November 2, 2021, No. 70-2021-00142.

\bibliographystyle{apalike}
\bibliography{reference}


\clearpage
\appendix

\thispagestyle{empty}

\onecolumn \makesupplementtitle

This document serves as supplementary material of the paper entitled ``Acceleration in Distributed Optimization under Similarity''. It contains additional theoretical  and numerical results, along with  all the proofs of the theorems presented in the main  paper. Specifically \begin{itemize}
    \item[] \textbf{Sec.~\ref{App_LB}} contains the lower complexity bounds  for the class of problems and oracle algorithms of interest;
    
    \item[] \textbf{Sec.~\ref{app_num_res}} presents additional numerical results for different classes of problems on synthetic and real data;

    \item[] \textbf{Sec.~\ref{app_proof_Th4}} provides  the proof of Theorem~\ref{thm:cata_net} and Theorem~\ref{crl:cata_net_lin};
    
     
      \item[] \textbf{Sec.~\ref{app_SONATA_star}} customizes $\texttt{ACC-SONATA}$ to star-topologies (master/workers architectures); 
      
      \item[] \textbf{Sec.~\ref{app_inexact}} discusses \texttt{Inexact SONATA} and \texttt{Inexact Accelerated SONATA}.
\end{itemize}

 \section{LOWER COMPLEXITY BOUNDS OVER MESH NETWORKS UNDER SIMILARITY}\label{App_LB}
Given \eqref{eq:problem} over a mesh network $\mathcal{G}$, we consider the following general class of  distributed algorithms, which generalize the oracle model   \citep{arjevani2015communication} for centralized schemes (star architectures) and smooth ($r= 0$) instances of \eqref{eq:problem}.

 \begin{definition}[Distributed oracle]\label{app_def_oracle} Each agent $i$ has its own  local memory $\mathcal{M}_i=\{0\}$,   updated as follows: 
 
 $\bullet$ \texttt{Local computation:} Between communication rounds, each agent $i$ iteratively computes and adds to   $\mathcal{M}_i$ some finite number of points $x$, each satisfying  \begin{align*} 
    \tau_1 \, x + \tau_2 \nabla f_i(x) + \tau_3\, g \, \in \textnormal{span}\bigg\{x^\prime, \, \nabla f_i(x^\prime), (\nabla^2 f_i(x^\prime)+ D)x^{\prime\prime}, (\nabla^2 f_i(x^\prime)+ D)^{-1}x^{\prime\prime} \bigg\},  
\end{align*}
for given $x^{\prime}, x^{\prime\prime}\in \mathcal{M}_i$, $g\in \partial r(x)$,   $\tau_1, \tau_2, \tau_3\geq 0$ such that $\tau_1+ \tau_2+ \tau_3>0$; and $D$ is a diagonal matrix such that the inverse matrices above  exist;\vspace{-0.1cm}

 $\bullet$ \texttt{Communications:} After every communication round, each agent $i$ updates  $\mathcal{M}_i$ according to  
$$\mathcal{M}_i=\textnormal{span}\left\{\underset{(i,j)\in\mathcal{E}}{\cup} M_j\right\};$$

  $\bullet$ \texttt{Output}:  $x_i\in \mathcal{M}_i$, for all $i\in [m]$.\vspace{-0.1cm}
 \end{definition}

 The above  procedure models a fairly general class of   distributed algorithms over graphs.  Computations at each node are  based on linear operations involving current or past iterates, local gradients, and vector products with local Hessians and their inverses; exact optimization of local  subproblem involving  such quantities   or proximal solutions are also incorporated.   During communications, the agents  can share with  their neighbors   any of the vectors they have computed up until that time.

The following result provides   lower   complexity bounds  for solving Problem~\eqref{eq:problem} by any distributed algorithm in $\mathcal{A}$--this is  an extension of  \cite[Th. 1]{arjevani2015communication} to the decentralized setting.

\begin{theorem} \label{th_LB}
For any $\mu\in [0,1), \, \beta \in (0,1),$ and $\rho\in [0,1)$, there exist (i) an instance of \eqref{eq:problem}  (with sufficiently large $d$  and  solution $x^\star$) satisfying Assumption~\ref{assump:class}, with $r=0$, $f_i$'s being $\beta$-similar and   $f(x) = \frac{1}{m} \sum_{i=1}^m f_i(x)$  being  $1$-smooth and $\mu$-strongly convex; and (ii)  a gossip matrix   $W$ satisfying Assumption~\ref{assump:weight} over the graph  $\mathcal{G}$ with parameter  $\rho$  
such that for any distributed algorithm in $\mathcal{A}$ using the matrix $W$ in the communications, the number of communication rounds $N$ required for obtaining the solution $x_i$'s such that  $u({x}_i) - u(x^\star) \leq \epsilon$, for all $i\in [m]$, is  $$N=\Omega \left( \sqrt{\frac{\beta/\mu}{1-\rho}} \log \left( \frac{\mu\norm{x^\star}^2}{\varepsilon} \right) \right).$$
\end{theorem}
\begin{proof}
The case of fully connected networks ($\rho=0$) has been  already studied  in \cite[Th. 1]{arjevani2015communication}. Therefore, here we assumed $\rho>0$. 
Our proof is inspired by \cite[Th.~2]{scaman2017optimal}, with some key differences due to the different setting of our problem.  

We first introduce the local cost functions of agents numbered as $1,2,\ldots,m.$  With $\zeta = \frac{1}{32}$,   define the following two  subsets of agents:  $$\mathcal{A}_l = \left\{  i \, \big\vert \, 1\leq i\leq \lceil \zeta m\rceil \right\}\quad 
\text{and}\quad \mathcal{A}_r =  \left\{  i \, \big\vert \, \lfloor (1-\zeta ) m \rfloor + 1 \leq i\leq m \right\}.$$ 
We then define for each agent $i$ the cost function $f_i: \, \ell^2 \to \mathbb{R}$ as
\[
f_i(x) =
\begin{cases}
\frac{\beta(1-\mu)}{8} \frac{m}{\lceil \zeta m\rceil} x^\top A_1 x - \frac{\beta(1-\mu)}{4 } \frac{m}{\lceil \zeta m\rceil} e_1^\top x + \frac{\mu}{2}  \norm{x}^2, & 1\leq i \leq  \lceil \zeta m\rceil  \\
\frac{\mu}{2}  \norm{x}^2, & \lceil \zeta m\rceil  + 1\leq i \leq  \lfloor (1-\zeta ) m \rfloor   \\
\frac{\beta(1-\mu)}{8} \frac{m}{\lceil \zeta m\rceil} x^\top A_2 x  + \frac{\mu}{2}  \norm{x}^2, &  \lfloor (1-\zeta ) m \rfloor + 1 \leq i\leq m
\end{cases}
\]
with \begin{equation*}
A_1\triangleq  
\begin{bmatrix}
1 &0 &0 &0 &0 &\cdots\\
0 &1 &-1 &0 &0  &\cdots\\
0 &-1 &1 &0 &0  &\cdots\\
0 &0 &0 &1 &-1 &\cdots\\
0 &0 &0 &-1 &1  &\cdots\\
\vdots &\vdots &\vdots &\vdots &\vdots &\ddots
\end{bmatrix},
\qquad
A_2\triangleq  
\begin{bmatrix}
1 &-1 &0 &0 &0  &\cdots\\
-1 &1 &0 &0 &0  &\cdots\\
0 &0 &1 &-1 &0  &\cdots\\
0 &0 &-1 &1 &0  &\cdots\\
0 &0 &0 &0 &1 &\cdots\\
\vdots &\vdots &\vdots &\vdots &\vdots &\ddots
\end{bmatrix}.\label{eq:two-splitting-matrices}
\end{equation*}

We define the distance between the two subsets $\mathcal{A}_l$ and $\mathcal{A}_r$ as $d_c.$
It follows that, to have at least one non-zero element in the $k$th entry  of the local copies of agents in both of the above
two subsets, one must perform at least $k$ local computation steps and $(k-1)d_c$ communication steps.  The number of communication rounds required to obtain $f(\hat{x}) - f(x^\star) \leq \epsilon$ is thus $$\Omega \left(d_c \sqrt{\frac{\beta}{\mu }} \log \left( \frac{\mu\norm{x^\star}^2}{\epsilon} \right) \right).$$  

We then describe the communication graph for the given $\rho.$
For $m\geq 2$, we define $$\rho_m = \frac{\rho}{2+\rho} + \frac{2}{2+\rho} \cos{\frac{\pi}{m}}.$$  Since $\rho_2 = \frac{\rho}{2+\rho} < \rho$ and $\lim_{k\to \infty} \rho_k = 1,$ we know that there exist $m$ such that $\rho_m < \rho \leq \rho_{m+1}.$  We discuss the cases of $m\geq 3$ and $m=2$ separately:

\textbf{i) } $m\geq 3.$  We begin defining a Laplacian matrix $L_{m,a}$ for a line graph composed of $m$ nodes; specifically: $$L_{m,a} = L_{m,a}^\top, \quad L_{m,a}\,1 = 0, \quad \text{and}  \quad  L_{m,a}(i,i+1) = a \,\mathds{1}\{i = 1\} - 1,$$  with $a\in[0,1)$, where $\mathds{1}\{\bullet\}$ is the indicator function. We then define  $$W_{m,a} = I - \frac{1}{2+\rho} L_{m,a}.$$   It is not difficult to check that  $W_{m,a}$ satisfies Assumption~\ref{assump:weight}, with  $ \norm{W_{m,0} -1\,1^\top/m } = \rho_m.$  Furthermore, since $\norm{W_{m,1} -1\,1^\top/m } = 1,$ by continuity, we know that there exists an $a\in(0,1)$ such that $\norm{W_{m,a} -1\,1^\top/m } = \rho.$  In addition, we have   $$1-\rho \geq 1-\rho_{m+1} = \frac{2}{2+\rho} \mybrace{1-\cos{\frac{\pi}{m+1}}}\geq \frac{2}{3} \mybrace{1-\cos{\frac{\pi}{m+1}}} \overset{(*)}{\geq} \frac{8}{3} \frac{1}{(m+1)^2}.$$  Note that $(*)$ is due to that $\cos{\frac{\pi}{n}} \leq 1- \frac{4}{n^2}$ for $n\geq 4.$  Equivalently, $$m \geq \sqrt{\frac{8}{3(1-\rho)}}-1.$$

The distance between the two subsets is thus 
\[
d_c = \lfloor (1-\zeta ) m \rfloor + 1 - \lceil \zeta m\rceil \geq  \frac{15}{16} m-1 \geq \frac{15}{16} \mybrace{\sqrt{\frac{ 8}{3(1-\rho)}} -1 } -1 \overset{(*)}{\geq} \frac{4}{25} \sqrt{\frac{ 1}{1-\rho}} .
\]
Note that $(*)$ is due to that $\rho>\rho_3> \frac{1}{2}.$

\textbf{ii)} $m=2.$  In this case, we have $\rho \leq \rho_3 = \frac{1+\rho}{2+\rho}$ and equivalently, $ \rho \in (0,  \frac{\sqrt{5}-1}{2}].$  We define the Laplacian matrix for a complete graph of $3$ nodes as: $L_a = L_a^\top $, $L_{a}\,1 = 0$, and for $i< j$, $L_a(i,j) = a \,\mathds{1}\{i = 1,\,j=2\} - 1.$  We then set $W_a = I - \frac{1}{3}L_a.$  Due to $\norm{W_{1} -1\,1^\top/3 } = \frac{2}{3} > \frac{\sqrt{5}-1}{2}$ and $\norm{W_{0} -1\,1^\top/3 } = 0,$ there exists an $a\in(0,1)$ such that $\norm{W_{a} -1\,1^\top/3 } = \rho.$
In this case, we have 
$$d_c = 1 \geq \sqrt{\frac{3-\sqrt{5}}{2}} \frac{1}{\sqrt{1-\rho}}.$$
 
Therefore, for any $\rho\in(0,1)$, there exists a communication matrix $W$ satisfying  Assumption~\ref{assump:weight}, with $\norm{W - 1\, 1^\top/m} = \rho$, such that the number of communication rounds required to obtain $f(\hat{x}) - f(x^\star) \leq \epsilon$ is $$\Omega \left( \sqrt{\frac{\beta}{\mu \, (1-\rho)}} \log \left( \frac{\mu\norm{x^\star}^2}{\epsilon} \right) \right).$$

\end{proof}

 \section{ADDITIONAL NUMERICAL RESULTS}\label{app_num_res}
This section complements Sec.~\ref{sec:num_result} of the main paper, providing additional numerical results and details on the experiments presented therein. 
 Specifically, we consider the following problems and data sets: \begin{itemize}
    \item \textbf{Sec.~\ref{sec:Hloss-numerical-result}}  studies hinge-loss minimization on  \texttt{MNIST} \citep{deng2012mnist}  and     \texttt{HIGGS}    \citep[from \texttt{LIBSVM}]{chang2011libsvm} datasets; 
    \item    \textbf{Sec.~\ref{sec:logistic-numerical-result}} considers logistic regression problems  on the  \texttt{SUSY}  dataset   \citep[from \texttt{LIBSVM}]{chang2011libsvm}. 
\end{itemize}

\subsection{Setting and Algorithm Tuning}
 
The setting of the experiments is the same as the one described in Sec.~\ref{sec:num_result}, expect of course for the instance of \eqref{eq:problem}.  

The tuning of the  simulated algorithms follows the instructions as in the associated papers.  
To set the free parameters, an estimate of the  smoothness constants $L_i$'s and strong-convexity constants  $\mu_i$'s is needed. We use the following bounds for the hinge and logistic losses.  In both cases, the local losses have the following structure:
\begin{equation*}
\begin{aligned}
f_i(x) = \frac{1}{n}\sum_{j=1}^n \ell\left(b_{i}^j\cdot \inn{x}{a_{i}^j}\right) + \frac{\lambda}{2}\norm{x}^2,
\end{aligned}
\end{equation*}
where $a_{i}^j \in \mathbb{R}^{d}$ are the feature vectors  and $b_{i}^j \in \{-1, 1\}$ are the  associated  labels. The Hessian matrix of  $f_i$ is
\begin{equation*}
\nabla^2 f_i(x) = \frac{1}{n}\sum_{j=1}^{n}\ell^{\prime\prime}\left(b_i^{j}\inn{x}{a_i^{j}}\right)(b_i^{j})^{2}a_i^{j}(a_i^{j})^\top + \lambda I,
\end{equation*}
where $\ell^{\prime\prime}$ denotes the second derivative of the loss $\ell:\mathbb{R}\to \mathbb{R}$. Under the assumption that  $\ell^{\prime\prime} \leq C_{\ell}$,   $$\lambda I \preceq \nabla^2 f_i(x) \preceq \frac{1}{n}\sum_{j=1}^{n} C_\ell   (b_i^j)^2a_i^{j}(a_i^{j})^T + \lambda I \triangleq H_i.$$
In particular, $C_{\ell} = 1$ for the smooth hinge loss  and $C_{\ell} = 1/4$ for the logistic loss. 

Based on $H_i$ above, we use  $\lambda$  as estimate of $\mu_i$     and the largest  eigenvalue of $H_i$ as that of  $L_i$. Furthermore, for the smooth constant $L$ of the average loss $f$ we use the overestimate  $\hat{L} \triangleq   \frac{1}{m}\sum_{i=1}^m L_i\geq L$;  and for $\beta$ we use  $\hat{\beta} \triangleq \max_{i\in [m]} \norm{H_i - \frac{1}{m}\sum_{i=1}^m H_i}$.  

The other tuning of the algorithms is as follows: 
\begin{itemize}
    \item \texttt{ACC-SONATA-F}: The inner iterations $T$ of \texttt{SONATA-F} and  \texttt{SONATA-L} are set to  $\ceil{\frac{7}{5}\cdot\log\frac{\hat{L}}{\mu}}$ and $\ceil{\log\frac{\hat{\beta}}{\mu}}$, respectively;
    \item \texttt{Mudag}: the number of inner loops is set to $\ceil{\frac{1}{5\sqrt{1-\rho}}\log\frac{L_{\text{mx}}}{\mu}}$;
    \item \texttt{ACC-EXTRA}: the number inner loops is set to $\ceil{\frac{1}{5(1-\rho)}\log\frac{L_{\text{mx}}}{\mu(1-\rho)}}$; 
    \item \texttt{APM-C}:  the number of inner loops is set to  $T_k = \ceil{ \frac{k\sqrt{\mu/L_{\text{mx}}}}{100(\sqrt{1-\rho})}}$;
    
    \item \texttt{OPAPC}: This is a single-loop distributed algorithm and all parameters are set according to the instructions in  \cite{kovalev2020optimal}.
\end{itemize}
Whenever the subproblems of the agents do not have a closed form solution, the gradient algorithm is employed, and terminated when an accuracy of $10^{-10}$ is reached on the Euclidean distance between variables of two consecutive iterations.  


 We are now ready to describe our experiments.

\subsection{ Hinge Loss Minimization}\label{sec:Hloss-numerical-result}
 Consider the following instance of \eqref{eq:problem}:
\begin{equation*}\label{dist_hinge}
\min_{x\in\mathbb{R}^d} \frac{1}{m}\sum_{i=1}^{m}\frac{1}{n}\sum_{j=1}^{n} \ell_{s}(b_i^{j}\cdot\inn{x}{a_i^{j}}) + \frac{\lambda}{2}\norm{x}^2,
\end{equation*}
where $\ell_{s}$ is the smooth hinge loss, defined as:
 \begin{equation*}\label{eq:hinge}
\ell_{s}(t)=\left\{
\begin{aligned}
&0 &t > 1, \\
&\frac{1}{2}(t-1)^2  & t \in [0,1], \\
&\frac{1}{2}-t  & t < 0.
\end{aligned}
\right.
\end{equation*}
  We consider two datasets for the above problem, namely  the \texttt{MNIST} and the \texttt{HIGGS}. We use the label $1$ for images of digit $4$ and $-1$ for the others.  
 Results are summarized in Fig.~\ref{fig:MNIST} (\texttt{MNIST}) and Fig.~\ref{fig:HIGGS} (\texttt{HIGGS}).

 
 Specifically, Fig.~\ref{fig:MNIST} (resp. Fig.~\ref{fig:HIGGS})-\textbf{left-panel}   
 plots the  optimality gap  $\frac{1}{m}\sum_{i=1}^m\|x_{i}^k-{x}_{\textnormal{op}}\|^2$ versus the  total number of communications, achieved by the different algorithms, where ${x}_{\textnormal{op}}$ is the optimal solution of the problem (estimated running the gradient algorithm up to a precision of $10^{-8}$ on the norm gradient).   In the \textbf{mid-panel}, we plot the number of communications to drive the optimality gap below $10^{-4}$ versus the   total sample size; we consider four sizes, namely: 
 $1.8\times10^{4}$, $3\times 10^{4}$, $4.8\times 10^{4}$ and $6\times 10^{4}$ 
 for the \texttt{MNIST} dataset and 
 $1.2\times10^{5}$, $2.4\times 10^{5}$, $4.8\times 10^{5}$ and $9\times 10^{5}$ 
 for the \texttt{HIGGS} dataset. The \textbf{righ-panel} shows the same results for \texttt{ACC-SONATA} and \texttt{OPAPC} on a re-scaled y-axes, to highlight the decreasing number of communications with the local sample size.    These results confirm that \texttt{ACC-SONATA-F} and \texttt{OPAPC}  consistently outperform the other methods. Notice that, differently \texttt{ACC-SONATA},  \texttt{OPAPC} is applicable only to smooth, unconstrained  instances of (P) (i.e., with $r\equiv 0$).  
 
 

\begin{figure*}[!h] \medskip 
\centering{\includegraphics[width=5.4cm, height=4.5cm]{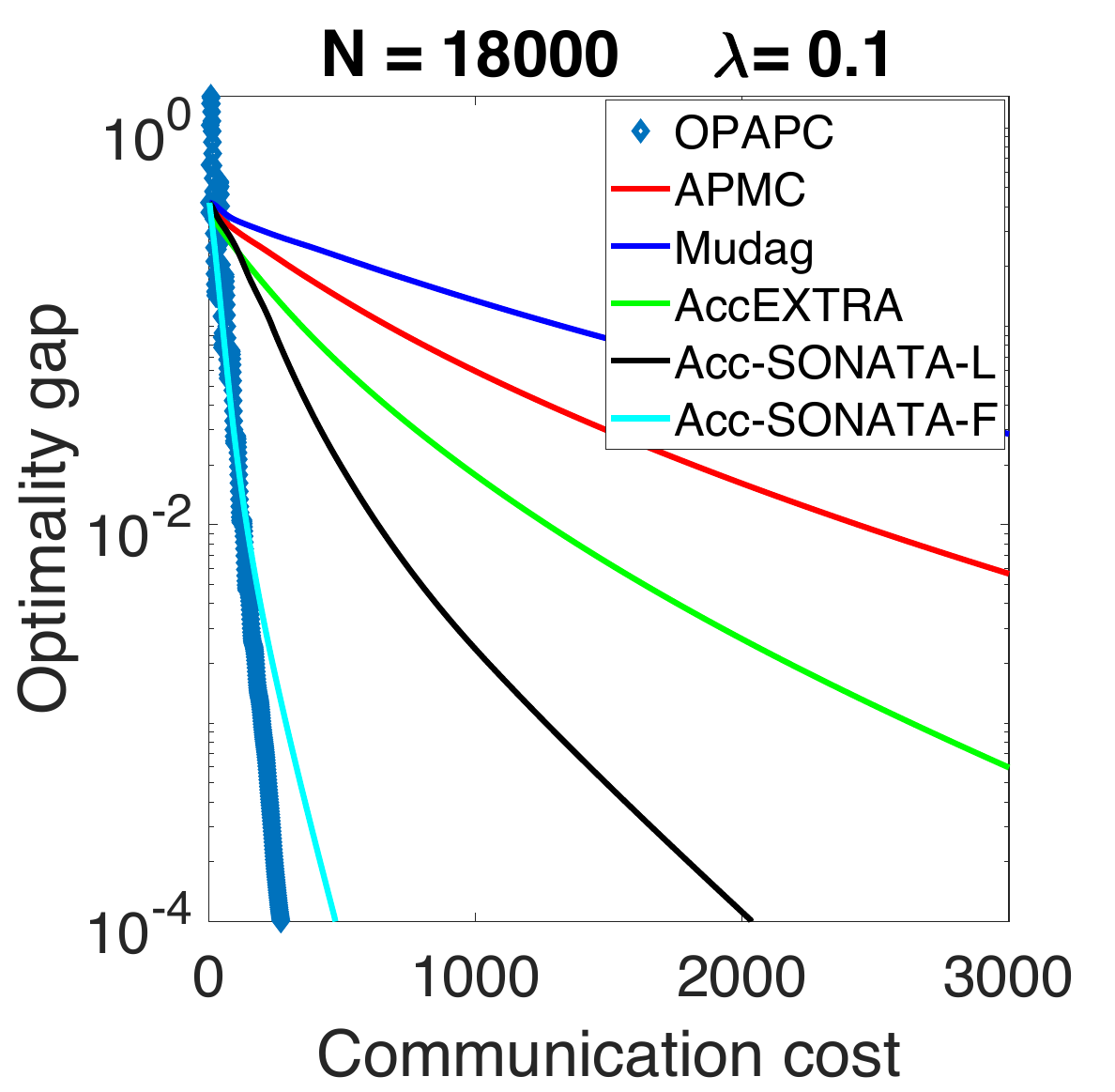}}\hspace{-.1cm}
\centering{\includegraphics[width=5.4cm, height=4.5cm]{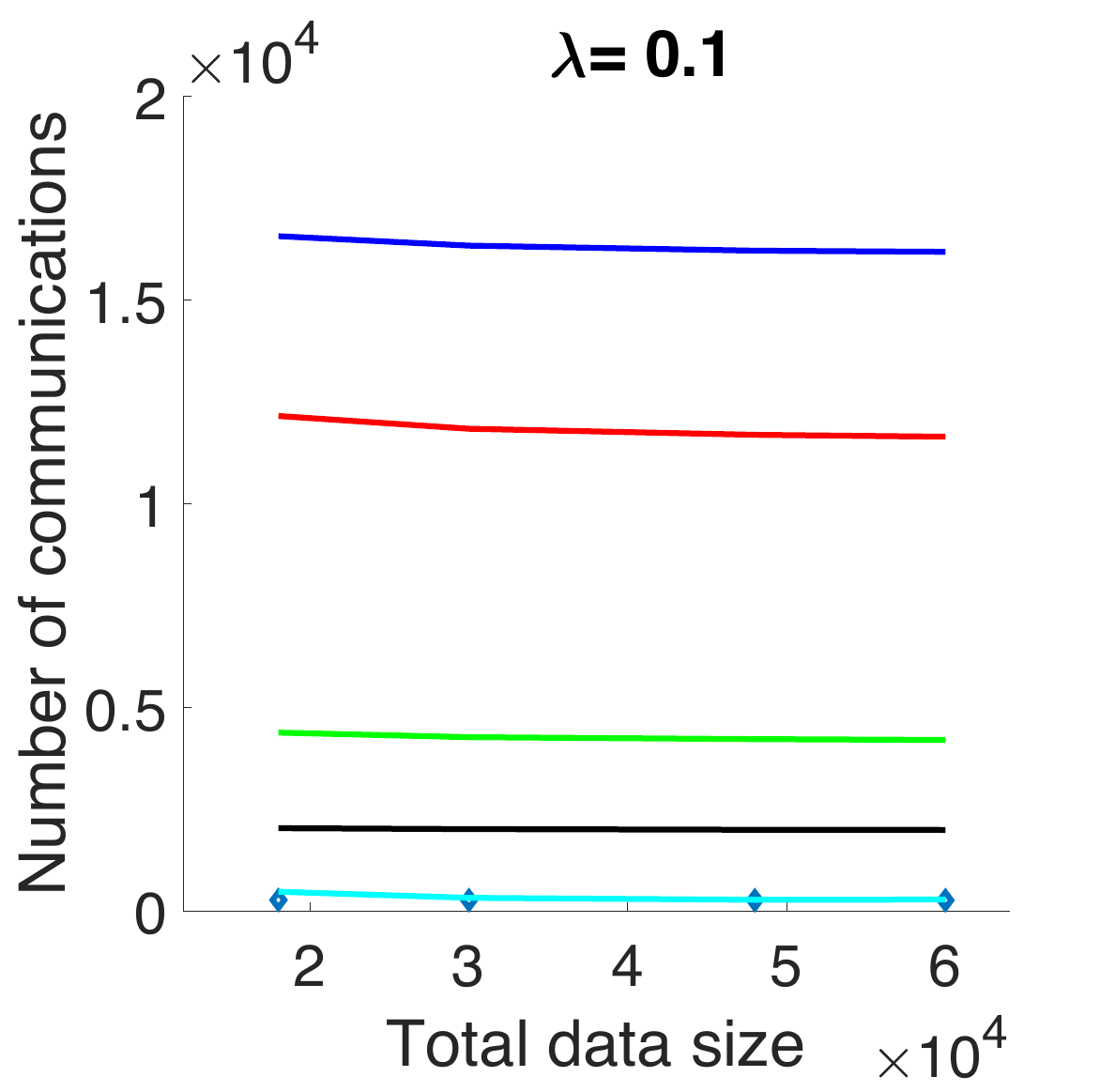}}
 \centering{\includegraphics[width=5.4cm, height=4.5cm]{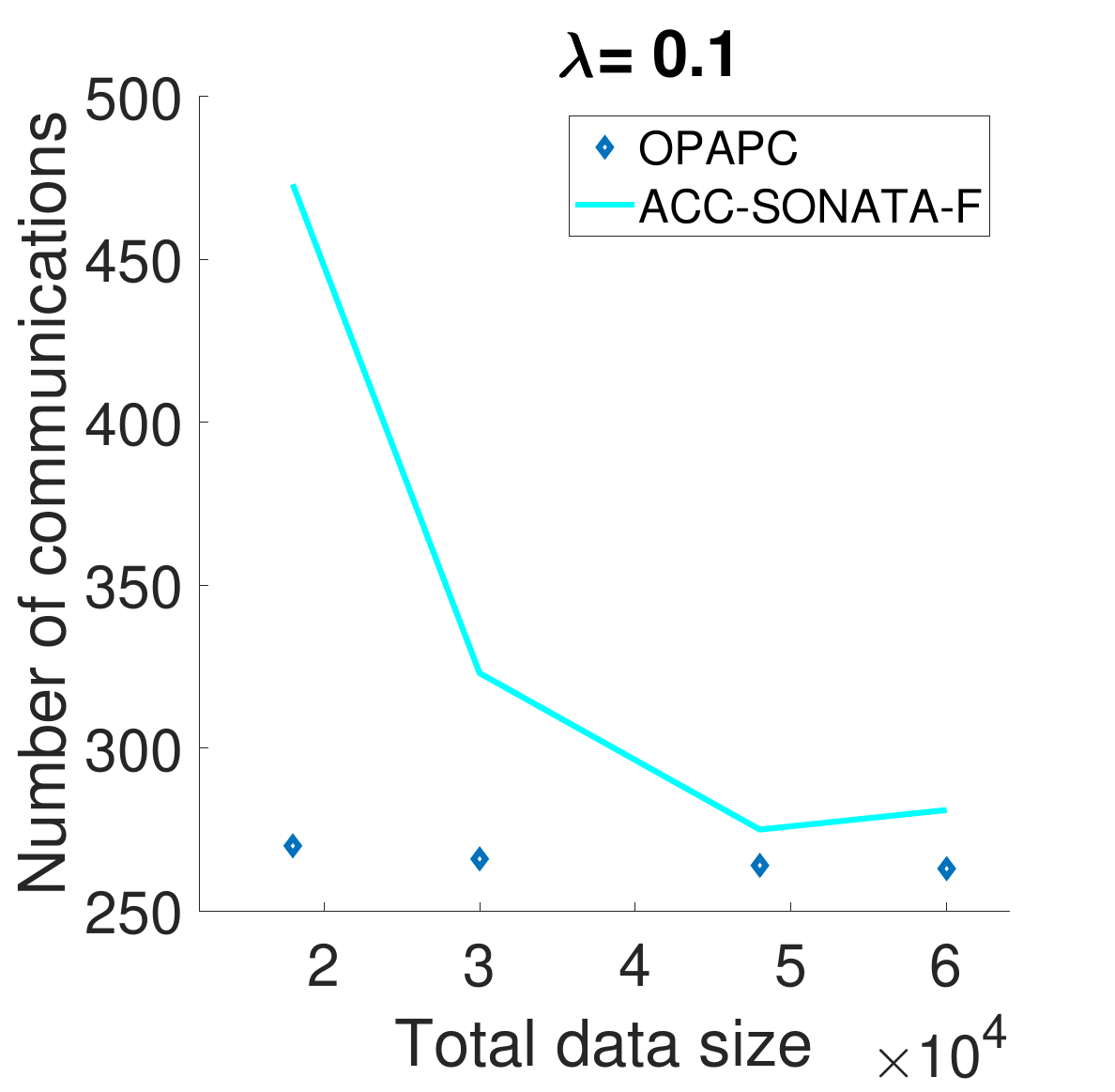}}

\caption{\small Hinge loss minimization, \texttt{MNIST} dataset.     \textbf{(left panel)}:  optimality gap  versus     total number of communications; \textbf{(mid   panel)}:     number of communications  to reach a precision of $10^{-4}$ versus (total) sample; \textbf{(right   panel)}: the mid panel on a different scale of the y-axes.}
\label{fig:MNIST} 
\end{figure*}

\begin{figure*}[!h] 
\centering{\includegraphics[width=5.4cm, height=4.5cm]{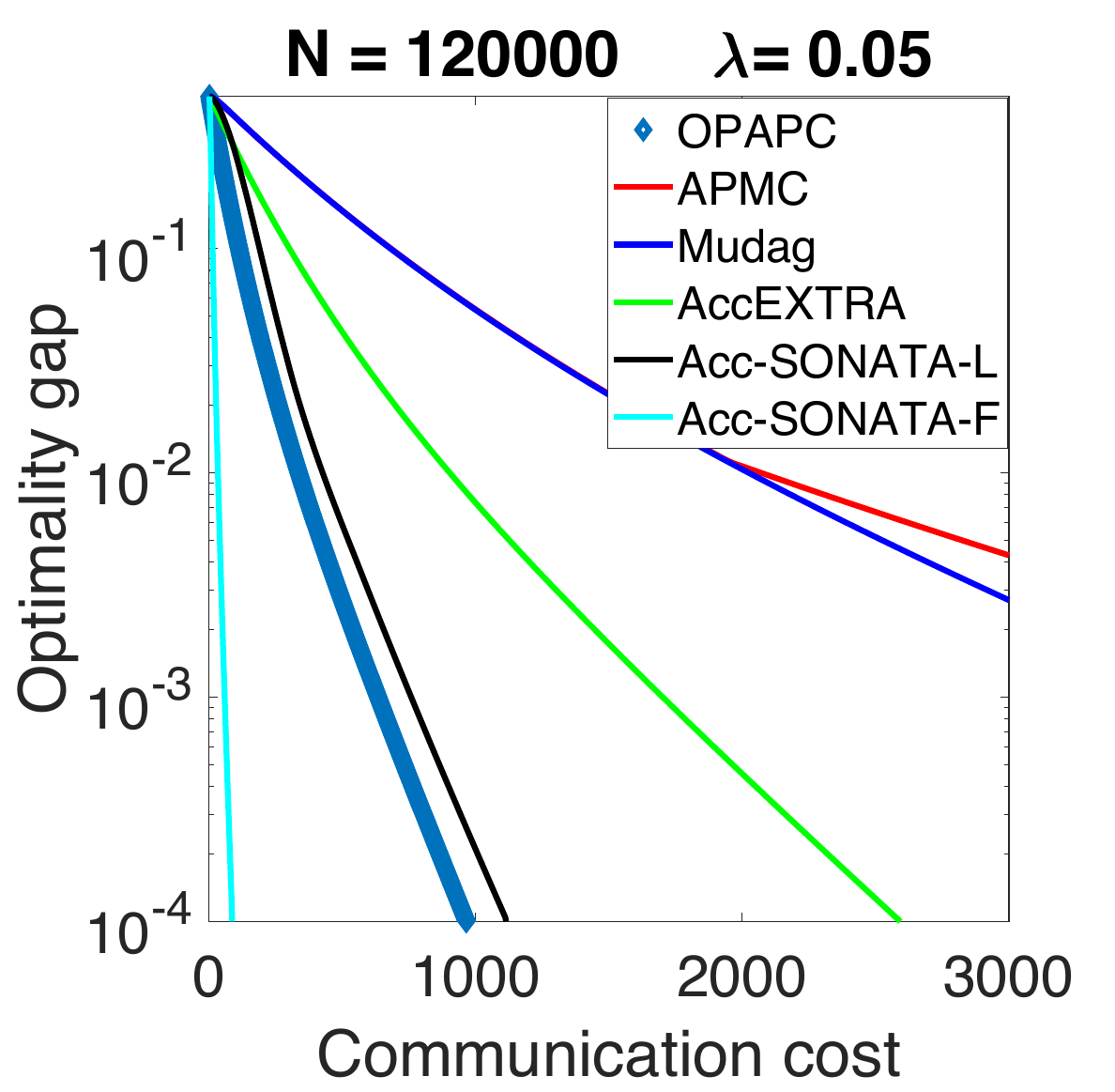}}\hspace{-.1cm}
\centering{\includegraphics[width=5.4cm, height=4.5cm]{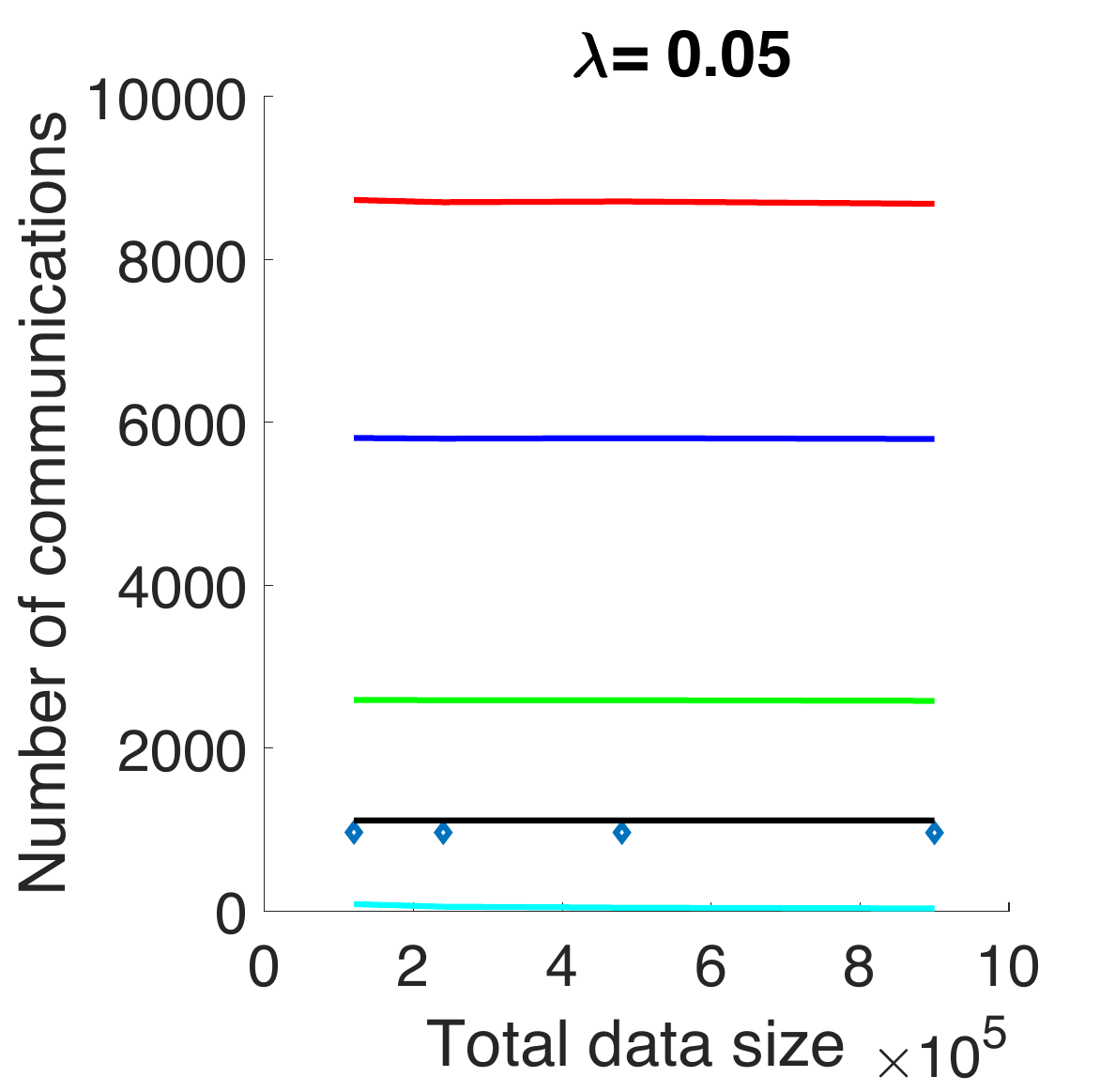}}
\centering{\includegraphics[width=5.4cm, height=4.5cm]{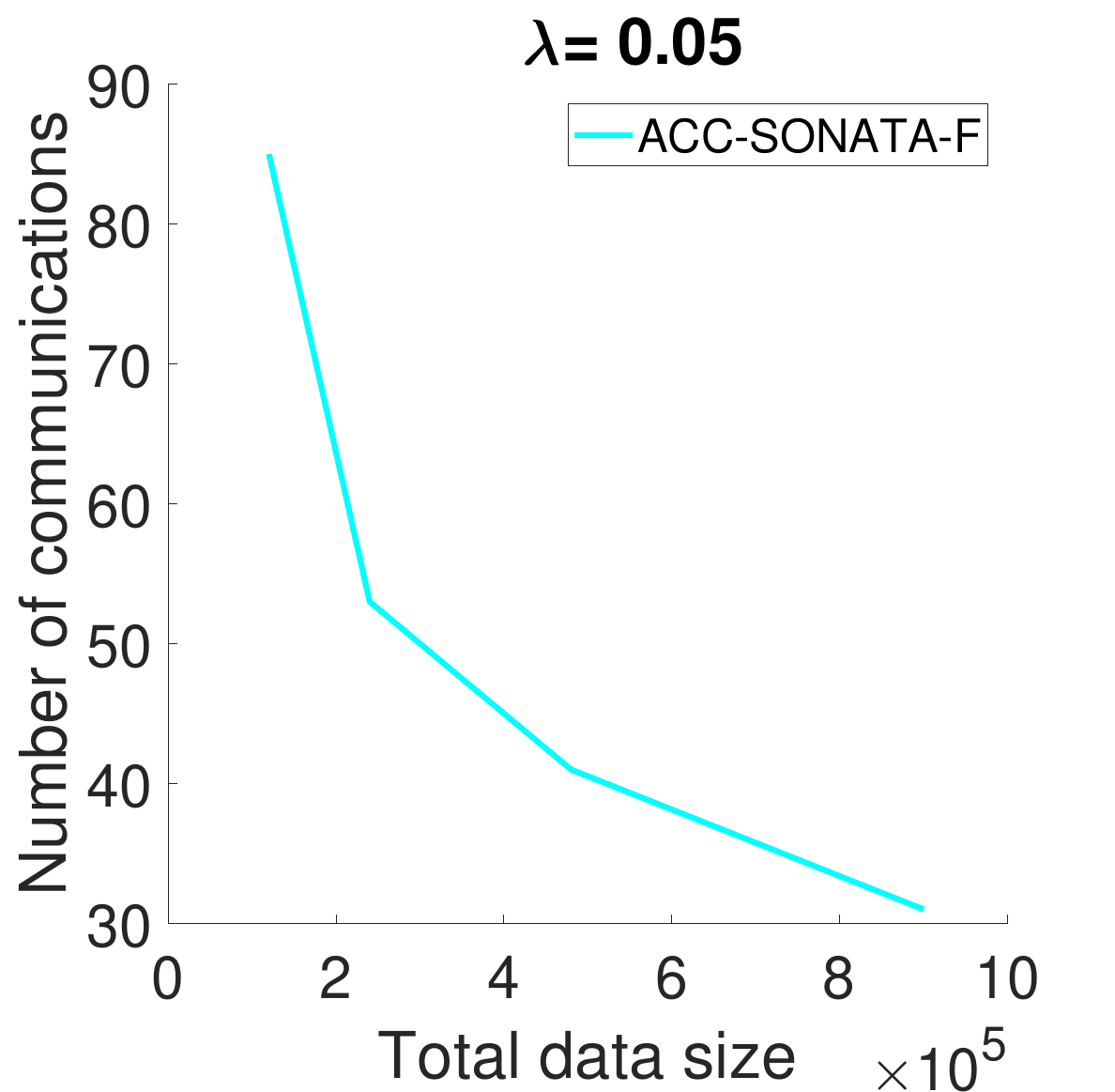}}

\caption{\small Hinge loss minimization, \texttt{HIGGS} dataset.     \textbf{(left panel)}:  optimality gap  versus     total number of communications; \textbf{(mid   panel)}:     number of communications  to reach a precision of $10^{-4}$ versus (total) sample; \textbf{(right   panel)}: the mid panel on a different scale of the y-axes. }
\label{fig:HIGGS} 
\end{figure*}

\begin{figure*}[!h] 
\centering{\includegraphics[width=5.4cm, height=4.5cm]{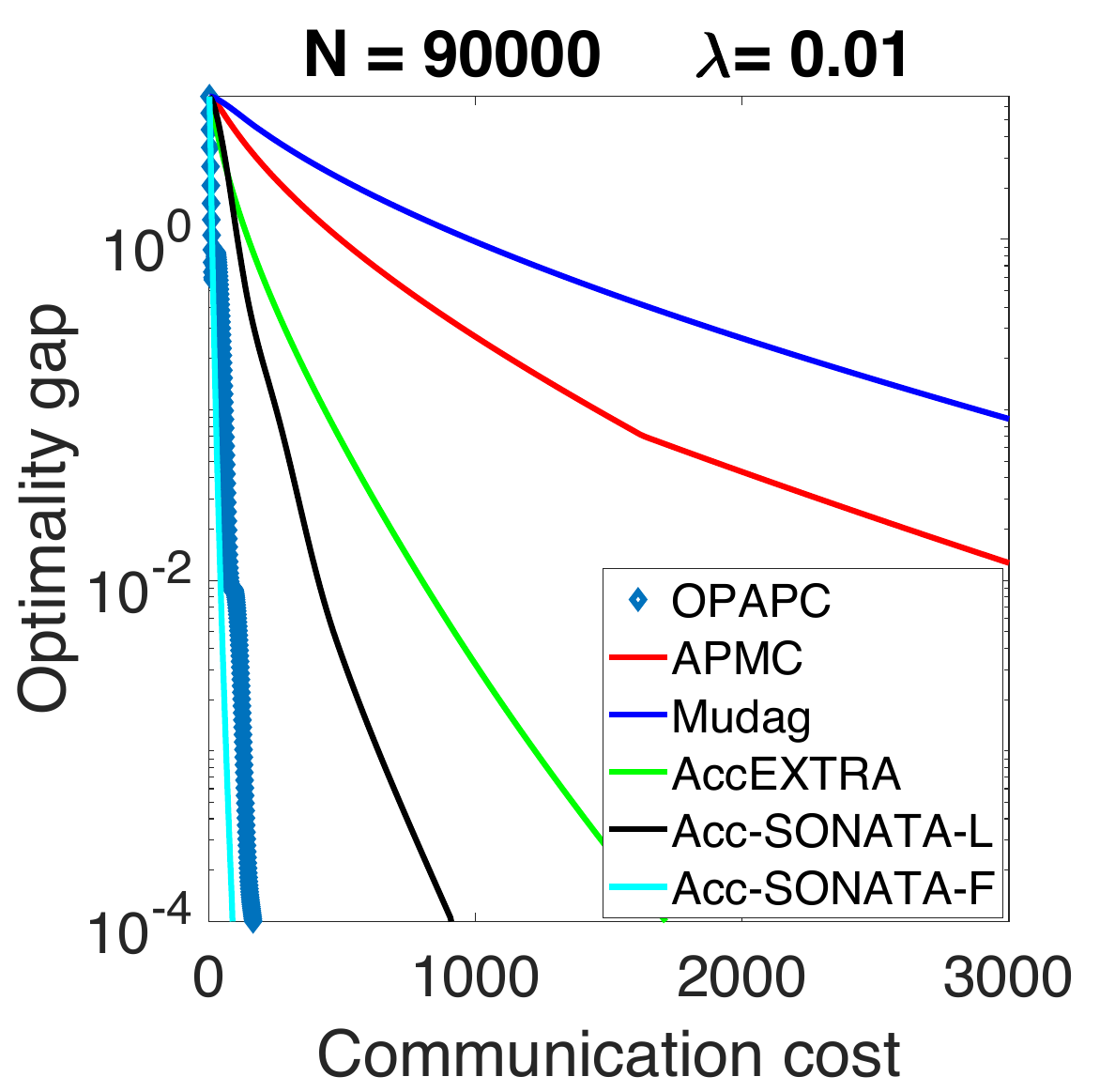}}\hspace{-.1cm}
\centering{\includegraphics[width=5.4cm, height=4.5cm]{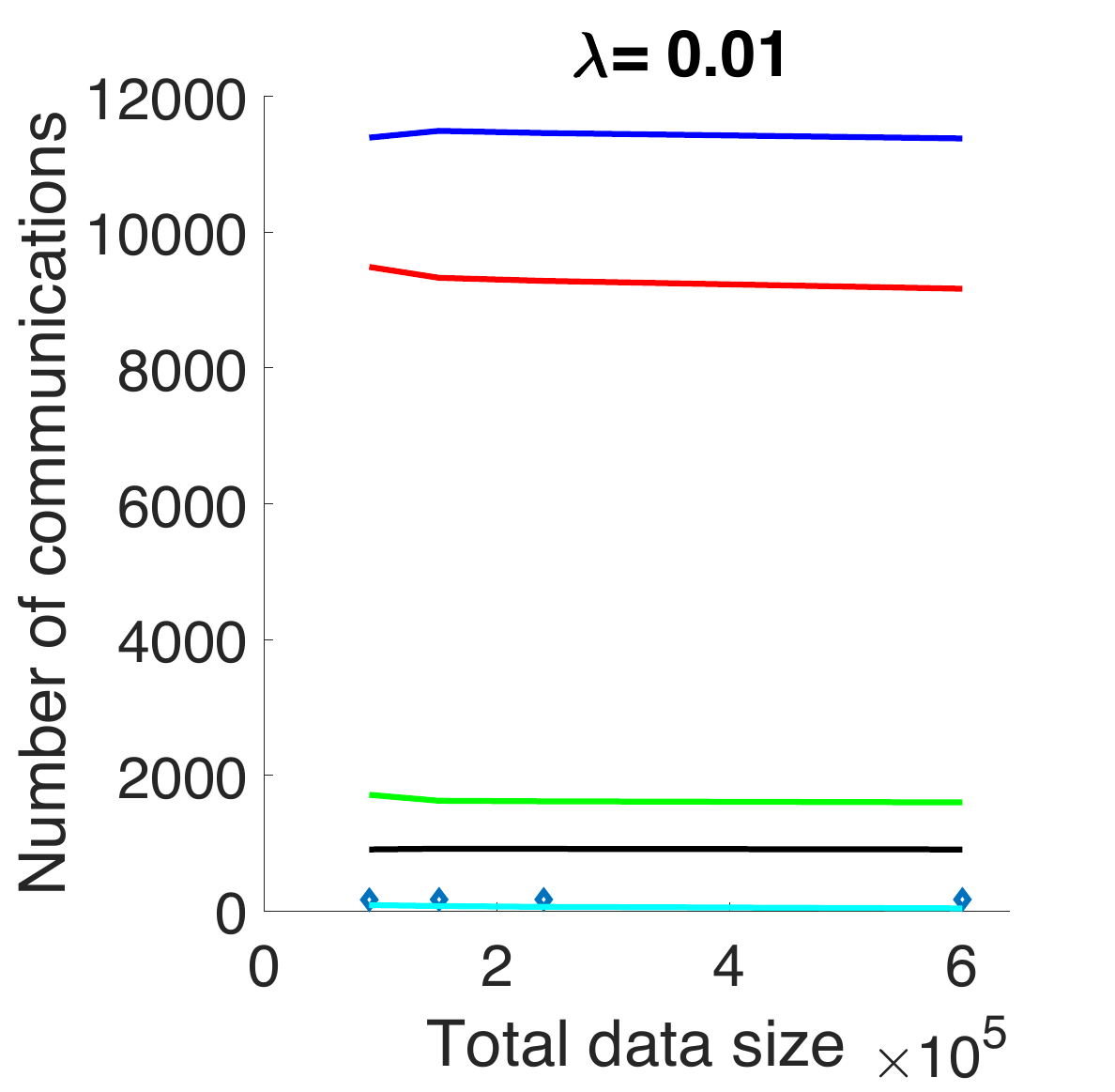}}
 \centering{\includegraphics[width=5.4cm, height=4.5cm]{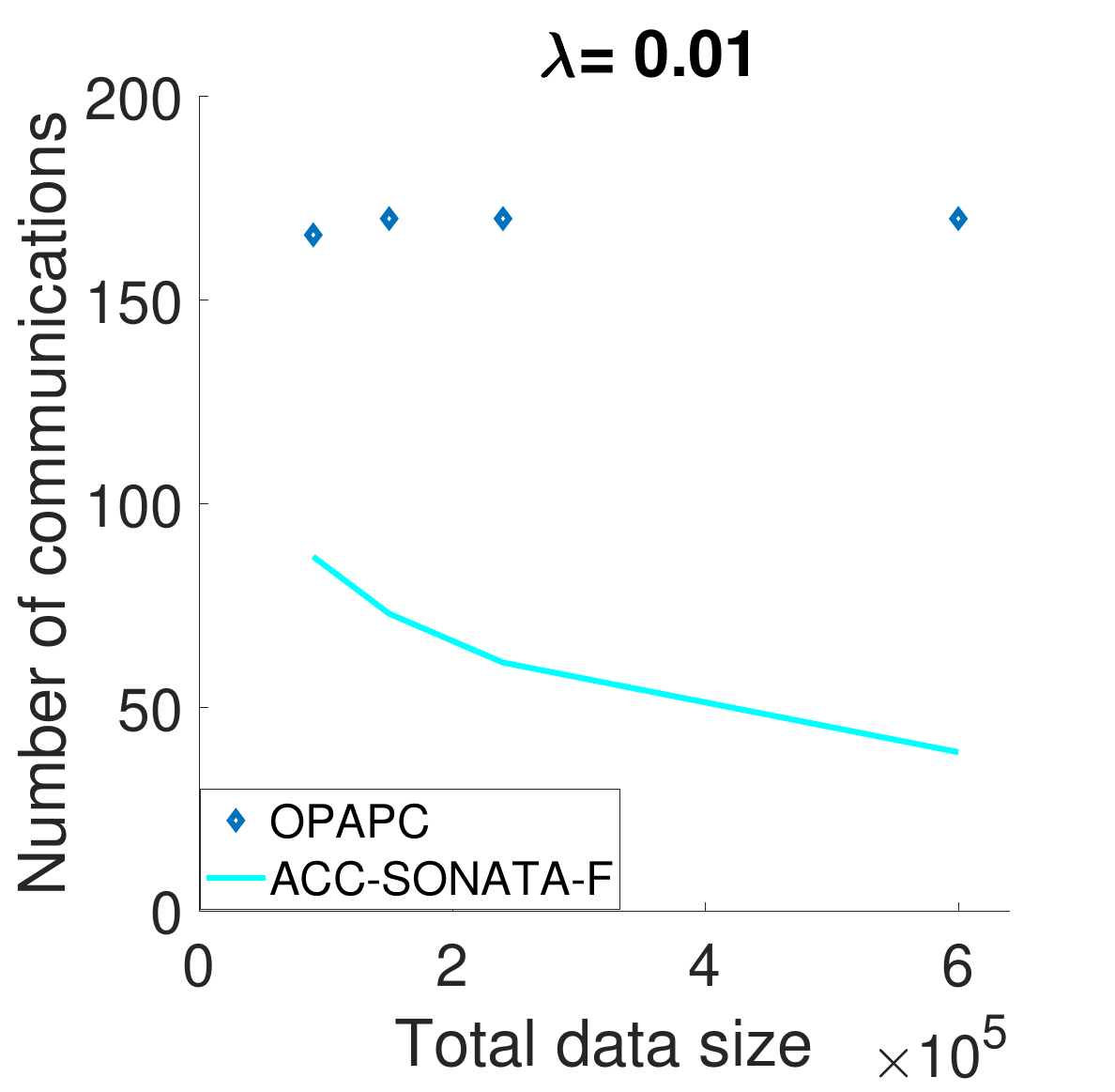}}

\caption{\small Logistic regression,  \texttt{SUSY} dataset.     \textbf{(left panel)}:  optimality gap  versus     total number of communications; \textbf{(mid  panel)}:     number of communications  to reach a precision of $10^{-4}$ versus (total) sample; \textbf{(right   panel)}: the mid panel on a different scale of the y-axes.}
\label{fig:susy} 
\end{figure*}

 \subsection{Logistic Regression}\label {sec:logistic-numerical-result} We consider here a logistic regression problem on \texttt{SUSY} dataset:
\begin{equation}\label{dist_log}
\min_{x\in\mathbb{R}^d} \frac{1}{m}\sum_{i=1}^{m}\frac{1}{n}\sum_{j=1}^{n} \ell_{r}\left(b_i^{j}\cdot\inn{x}{a_i^{j}}\right) + \frac{\lambda}{2}\norm{x}^2,
\end{equation}
where $\ell_{r} = \log(1+e^{-t})$.

 
  Fig. ~\ref{fig:susy} (\textbf{left-panel}) is a plot of optimality gap  $\frac{1}{m}\sum_{i=1}^m\|x_{i}^k-{x}_{\textnormal{op}}\|^2$ versus total number of communications. 
  In the \textbf{mid-panel}, we plot the number of communications to obtain optimality gap $10^{-4}$ versus the size of total data samples, namely: 
  $9\times 10^{4}$, $1.5\times 10^{5}$, $2.4\times 10^{5}$ and $6\times 10^{5}$. The \textbf{righ-panel} shows the same results for \texttt{ACC-SONATA} and \texttt{OPAPC} on a re-scaled y-axes.
  Consistently with the other results, also for this class of (nonquadratic) problems, \texttt{ACC-SONATA} and \texttt{OPAPC} exhibit favorable performance.  

\section{PROOF OF THEOREM~\ref{thm:cata_net} AND THEOREM~\ref{crl:cata_net_lin}}\label{app_proof_Th4}

 \subsection{Sketch of the Proof}\label{sec:app_sketch}
 As discussed in Sec.~\ref{sec:alg-design} [see (\texttt{S.1})$^\prime$ and (\texttt{S.2})$^\prime$], \texttt{ACC-SONATA} can be interpreted as an accelerated version of the inexact proximal point algorithm. The challenge here is finding a suitable notion of inexactness for proximal operations that captures all consensus  errors (on $x, y, z$-variables) while being implementable in the distributed setting and, at the same time, retains (for the outer loop) the   convergence rate, $\alpha$,   of the exact accelerated proximal method. Our path towards this goal consists in the following two steps: \begin{itemize}
     \item \textbf{Step 1 (inexactness and outer-loop convergence):} 
 We introduce an inner termination condition for the \texttt{SONATA} algorithm [see \eqref{eq:termination}, Sec.~\ref{sec_Step1}], serving as inexact notion of approximate proximal solution. This criterion hinges  on a proper potential function controlling the decrease of the optimality gap up to consensus/tracking errors. Roughly speaking, this quantifies the amount of errors in the minimization of   $u_k$ [see \eqref{eq:u_k}] that  can be tolerated to preserve the convergence of the outer loop   at the desired rate $\alpha=\sqrt{\frac{\mu}{\mu+\delta}}$. Convergence of the outer loop at such a rate is established  by introducing a proper  potential function [see \eqref{eq:LYAPU_0}, Sec.~\ref{sec_Step1}]. Such a potential function certifies  linear convergence of the optimality gap     $$\Delta(x^k)= \max\mybrace{\frac{1}{m}\sum_{i=1}^m u(x_i^k) - u^\star,\frac{1}{m}\sum_{i=1}^m \norm{x_i^k - \bar{x}^k}^2 }$$  at rate  $\alpha$. In the  setting of Theorem~\ref{thm:cata_net}   (i.e., $\delta=\beta-\mu$) and Theorem~\ref{crl:cata_net_lin}  (i.e., $\delta=L-\mu$), $\alpha$ reads  \begin{equation}\label{eq:rate_outer_loop}
     \alpha=\sqrt{\frac{\mu}{\beta}}\quad \text{and} \quad \alpha=\sqrt{\frac{1}{\kappa}},
 \end{equation}   respectively. 
 
   \item \textbf{Step 2 (inner-loop convergence):}  We introduce a refined analysis of the \texttt{SONATA} algorithm based on a new potential function  certifying that the inner termination criterion defined in \textbf{Step 1} is met in $T=\tilde{\mathcal{O}}(1)$ number of iterations. 
 \end{itemize}
  Combining \textbf{Step 1} and \textbf{Step 2}, we can then conclude that \texttt{ACC-SONATA} achieves an $\varepsilon$-solution of Problem~\eqref{eq:problem} in $${\mathcal O}\left(T \frac{1}{\alpha}\,\log\frac{1}{\varepsilon} \right)=\widetilde{\mathcal O}\left(\frac{1}{\alpha}\,\log\frac{1}{\varepsilon} \right)=\left\{\begin{matrix}
 \widetilde{\mathcal O}\left(\sqrt{\dfrac{\beta}{\mu}}\,\log\dfrac{1}{\varepsilon} \right),&  \text{if }\delta=\beta-\mu\,\, (\text{Theorem~\ref{thm:cata_net}}), &\medskip \\
 \widetilde{\mathcal O}\left(\sqrt{\kappa}\,\log\dfrac{1}{\varepsilon} \right),&  \text{if  }\delta=L-\mu\,\, (\text{Theorem~\ref{crl:cata_net_lin}}),& \end{matrix}\right.$$    
 total number of communications. 
 The  formal proof of \textbf{Step 1} and \textbf{Step 2} is given in Sec.~\ref{sec_Step1} and Sec.~\ref{sec_Step2}, respectively.

\textbf{Notation:} Before detailing  the two steps, we introduce some notation used throughout the proofs.     For each $i\in [m]$, we denote by 
$x_i^{k,t}$ and $y_i^{k,t}$   the decision variable and tracking variable of agent $i$ after  $t=0,\ldots, T-1$  inner iterations   in the $k$-th outer iteration, respectively. Clearly, it must be   
$$x_i^{k,0} = x_i^k,\quad  y^{k,0} = y_i^k+ \delta  \mybrace{z_i^{k-1} - z_i^k},\quad  x_i^{k,T} = x_i^{k+1},\quad\text{and}\quad y^{k,T} = y_i^{k+1}.$$  
Furthermore, we   define $$x^\star \triangleq \argmin_{x\in \mathbb{R}^d} u(x),\quad x^{k+1,\star} = \argmin_{x \in \mathbb{R}^d} u_k(x),\quad u_k^\star = \min_{x \in \mathbb{R}^d} u_k(x).$$ 
For any vector $x=[x_1^\top,\ldots ,x_m^\top]^\top$, we use $\bar{x}$ to denote the average of its $d$-dimensional blocks $x_i$'s, that is, $\bar{x}=\frac{1}{m}\sum_{i=1}^m x_i$. 

Consensus and tracking errors   associated with the iterates $x_i^{k,t}$ and $y_i^{k,t}$ are defined as
\begin{equation}\label{eq:cons_trak_errors}
    \norm{x_\perp^{k,t}}^2 = \frac{1}{m}\sum_{i\in[m]} \norm{x_i^{k,t}-\bar{x}^{k,t}}^2\quad \text{and}\quad \norm{y_\perp^{k,t}}^2 = \frac{1}{m}\sum_{i\in[m]} \norm{y_i^{k,t}-\bar{x}^{y,t}}^2,
\end{equation}
respectively.

\subsection{Step 1: Inexactness and Outer Loop Convergence}\label{sec_Step1}

Our inexact notion of approximate proximal solution of the minimization of $u_k$ [see (\ref{eq:u_k})] is defined in terms of   the (average) optimality gap,   
\begin{equation}\label{eq:_opt_gap}
    g^{k,t} \triangleq \frac{1}{m}\sum_{i=1}^m \mybrace{u_k(x_i^{k,t}) - u_k^\star},\quad k=0,1,\ldots, \,\, t=0,\ldots, T,
\end{equation}
and the consensus and tracking errors (\ref{eq:cons_trak_errors}), captured by  
\begin{equation}\label{eq:def_ekt}
    e^{k,t} \triangleq  c_x \norm{x_\perp^{k,t}}^2 + c_y \norm{y_\perp^{k,t}}^2,\quad k=0,1,\ldots, \,\, t=0,\ldots, T,
\end{equation}
where $c_x,c_y>0$ are suitably defined universal constants. Specifically, under the following event  
\begin{equation}\label{eq:termination}
    g^{k,T} + e^{k,T} \leq \epsilon^{k+1},\quad   k=0,1, \ldots,   
\end{equation} with  $\{\epsilon^k\}$ being a suitably defined  geometrically-vanishing positive sequence, we establish linear decay    of the following potential function  along the outer-loop iterates of \texttt{ACC-SONATA}:  
\begin{align}\label{eq:LYAPU_0}
    P^{k} \triangleq \frac{1}{m}\sum_{j=1}^m \mybrace{u(x_j^k) -u^\star} + \frac{1}{m}\sum_{j=1}^m \frac{\mu}{2} \norm{x_j^{k-1}+ \frac{1}{\alpha} (x_j^k - x_j^{k-1}) - x^\star}^2 + e^{k-1,T},\quad k=0,1,\ldots ,
\end{align}
where we set  $e^{-1,T} = c_y \norm{y_\perp^0}^2$ and $x_i^{-1} = 0, \,\forall i$.  

In \textbf{Step 2} we will show that (\ref{eq:termination}) can be met by running \texttt{SONATA} for  $T=\tilde{\mathcal O}(1)$ iterations.

\begin{proposition}\label{prop_step1}
Consider problem \eqref{eq:problem} under Assumption~\ref{assump:class}, with optimal objective value $u^\star$. 
Let $\{(x_i^k,z_i^k)_{i\in [m]}\}$ be the sequence generated by \textnormal{\texttt{ACC-SONATA}} under (\ref{eq:termination}), with error sequence $\epsilon^k=P^0 \,(1-c \alpha)^k$, $k=1,\ldots $, where $c$ is any given constant in $(0,1)$. Then, the potential function $P^k$ in (\ref{eq:LYAPU_0}) satisfies: \begin{align}\label{eq:RATE}
    P^k \leq c_2 \,P^0\, (1-c\cdot \alpha)^{k}, 
\end{align}
with \begin{equation}\label{eq:c2}
    c_2 = \frac{(2+\sqrt{c_1})^2 }{\mybrace{\sqrt{\frac{1-c\cdot \alpha}{1-\alpha}} -1 }^2 (1-\alpha)}
    \quad\text{and}\quad 
    c_1 = 1+ \frac{\delta}{c_x } \frac{\frac{3}{2} (1-c\alpha)^2 + 5- 4c\,\alpha}{(1-c\alpha)^2}
.\end{equation} Therefore, $$\max\mybrace{\frac{1}{m}\sum_{i=1}^m u(x_i^k) - u^\star,\frac{1}{m}\sum_{i=1}^m \norm{x_i^k - \bar{x}^k}^2 }=\mathcal{O}\left((1-c\cdot \alpha)^{k}\right).$$ 
\end{proposition}

\begin{proof}
We provide  a constructive proof for the choice of  the error sequence $\epsilon^k$ in (\ref{eq:termination}) and potential  function \eqref{eq:LYAPU_0}, yielding \eqref{eq:RATE}.  

  Using the definition of $u_k$, we have: for  all $x_j\in \text{dom}\,r$,  
\begin{equation*}
\begin{split}
    & \frac{1}{m} \sum_{j=1}^m u(x_j^{k+1}) = \frac{1}{m} \sum_{j=1}^m u_k(x_j^{k+1}) - \frac{\delta}{2m^2} \sum_{i,j} \norm{x_j^{k+1} - z_i^k}^2  \overset{\eqref{eq:termination}}{\leq} u_k^\star + \epsilon^{k+1} -e^{k,T} - \frac{\delta}{2m^2} \sum_{i,j} \norm{x_j^{k+1} - z_i^k}^2 \\
    & \leq \frac{1}{m} \sum_{j=1}^m \mybrace{u_k(x_j) - \frac{\mu+\delta}{2} \norm{x_j-x^{k+1, \star}}^2} + \epsilon^{k+1} -e^{k,T} - \frac{\delta}{2m^2} \sum_{i,j} \norm{x_j^{k+1} - z_i^k}^2 \\
    & = \frac{1}{m} \sum_{j=1}^m u_k(x_j) - \frac{1}{m} \sum_{j=1}^m \frac{\mu+\delta}{2} \norm{x_j-x_j^{k+1}}^2 + \frac{1}{m} \sum_{j=1}^m \frac{\mu+\delta}{2} \norm{x_j-x_j^{k+1}}^2 - \frac{1}{m} \sum_{j=1}^m \frac{\mu+\delta}{2} \norm{x_j-x^{k+1, \star}}^2  \\
    & \qquad +  \epsilon^{k+1} -e^{k,T} - \frac{\delta}{2m^2} \sum_{i,j} \norm{x_j^{k+1} - z_i^k}^2 \\
    & = \frac{1}{m} \,\sum_{j=1}^m u(x_j) + \frac{1}{m^2} \sum_{i,j} \frac{\delta}{2} \norm{x_j-z_i^k}^2 - \frac{1}{m} \sum_{j=1}^m \frac{\mu+\delta}{2} \norm{x_j-x_j^{k+1}}^2 + \frac{1}{m} \sum_{j=1}^m \frac{\mu+\delta}{2} \norm{x_j-x_j^{k+1}}^2 \\
    & \qquad - \frac{1}{m} \sum_{j=1}^m \frac{\mu+\delta}{2} \norm{x_j-x^{k+1, \star}}^2  +  \epsilon^{k+1} -e^{k,T} - \frac{\delta}{2m^2} \sum_{i,j} \norm{x_j^{k+1} - z_i^k}^2 \\
    & = \frac{1}{m} \,\sum_{j=1}^m u(x_j) - \frac{\mu}{2m} \sum_{j=1}^m \norm{x_j-x_j^{k+1}}^2 + \frac{\delta}{2m} \sum_{j=1}^m \mybrace{\frac{1}{m} \sum_{i=1}^m \norm{x_j-z_i^k}^2 - \norm{x_j-x_j^{k+1}}^2 - \frac{1}{m} \sum_{i=1}^m \norm{x_j^{k+1} - z_i^k}^2} \\
    & \qquad + \frac{1}{m} \sum_{j=1}^m \frac{\mu+\delta}{2} \norm{x_j-x_j^{k+1}}^2 - \frac{1}{m} \sum_{j=1}^m \frac{\mu+\delta}{2} \norm{x_j-x^{k+1, \star}}^2  + \epsilon^{k+1} - e^{k,T}\\
    & \leq \frac{1}{m} \,\sum_{j=1}^m u(x_j) - \frac{\mu}{2m} \sum_{j=1}^m \norm{x_j-x_j^{k+1}}^2 + \frac{\delta}{m} \sum_{j=1}^m \inn{x_j-x_j^{k+1}}{x_j^{k+1} - \bar{z}^k} \\
    & \qquad - \frac{1}{m} \sum_{j=1}^m (\mu+\delta) \inn{x_j-x_j^{k+1}}{x_j^{k+1} - x^{k+1, \star}} + \epsilon^{k+1} - e^{k,T}.
\end{split}
\end{equation*}

Setting $x_j = x^\star $,  $j\in[m]$, leads to
\begin{equation}\label{eq:xstar}
\begin{aligned}
& \frac{1}{m} \sum_{j=1}^m u(x_j^{k+1}) \leq u^\star - \frac{\mu}{2m} \sum_{j=1}^m \norm{x^\star-x_j^{k+1}}^2 + \frac{\delta}{m} \sum_{j=1}^m \inn{x^\star - x_j^{k+1}}{x_j^{k+1} - \bar{z}^k} \\
& \qquad - \frac{1}{m} \sum_{j=1}^m (\mu+\delta) \inn{x^\star - x_j^{k+1}}{x_j^{k+1} - x^{k+1, \star}} + \epsilon^{k+1} - e^{k,T}. 
\end{aligned}  
\end{equation}
Similarly, setting $x_j = x_j^k$, we have
\begin{equation}\label{eq:xk}
\begin{aligned}
& \frac{1}{m} \sum_{j=1}^m u(x_j^{k+1}) \leq \frac{1}{m} \sum_{j=1}^m u(x_j^k) - \frac{\mu}{2m} \sum_{j=1}^m \norm{x_j^k-x_j^{k+1}}^2 + \frac{\delta}{m} \sum_{j=1}^m \inn{x_j^k - x_j^{k+1}}{x_j^{k+1} - \bar{z}^k} \\
& \qquad - \frac{1}{m} \sum_{j=1}^m (\mu+\delta) \inn{x_j^k - x_j^{k+1}}{x_j^{k+1} - x^{k+1, \star}} + \epsilon^{k+1} - e^{k,T}.
\end{aligned}  
\end{equation}
Define   the optimality gap at the beginning of the $k$-th outer iteration, pertaining to the minimization of 
 the original objective function $u(x)$ as: $$p^k \triangleq \frac{1}{m} \sum_{j=1}^m  u(x_j^k) -u^\star.$$  Then, multiplying \eqref{eq:xstar} by $\alpha$ and \eqref{eq:xk} by $(1-\alpha)$, and suming the obtained equations, yields
\begin{equation}\label{eq:dist_lyap_part1}
\begin{aligned}
    & p^{k+1} \leq   (1-\alpha) p^k - \frac{\alpha \mu}{2m}  \sum_{j=1}^m \norm{x^\star-x_j^{k+1}}^2 -  \frac{(1-\alpha) \mu}{2m}  \sum_{j=1}^m \norm{x_j^k-x_j^{k+1}}^2  \\
    & \quad + \frac{\delta}{m} \sum_{j=1}^m \inn{\alpha x^\star + (1-\alpha) x_j^k-x_j^{k+1}}{x_j^{k+1} - \bar{z}^k} \\
    & \quad - \frac{\mu+\delta}{m} \sum_{j=1}^m \inn{\alpha x^\star + (1-\alpha) x_j^k-x_j^{k+1} }{x_j^{k+1} - x^{k+1, \star}} + \epsilon^{k+1} - e^{k,T} \\
    & = (1-\alpha) p^k  - \frac{\alpha \mu}{2m}  \sum_{j=1}^m \norm{x^\star-x_j^{k+1}}^2 -  \frac{(1-\alpha) \mu}{2m}  \sum_{j=1}^m \norm{x_j^k-x_j^{k+1}}^2 - \frac{\delta}{m} \sum_{j=1}^m \norm{x_j^{k+1} - \bar{z}^k}^2 \\
    & \quad + \frac{\delta}{m} \sum_{j=1}^m \inn{\alpha x^\star + (1-\alpha) x_j^k-\bar{z}^k }{x_j^{k+1} - \bar{z}^k} \\
    & \quad - \frac{\mu+\delta}{m} \sum_{j=1}^m \inn{\alpha x^\star + (1-\alpha) x_j^k-x_j^{k+1} }{x_j^{k+1} - x^{k+1, \star}} + \epsilon^{k+1} - e^{k,T}.
\end{aligned}
\end{equation}
The above is an approximate descent on $p^k,$ up to the error term $\epsilon^{k+1} - e^{k,T}$ and the two  inner-product terms.  
We proceed  working first  on  the term $\inn{\alpha x^\star + (1-\alpha) x_j^k-\bar{z}^k }{x_j^{k+1} - \bar{z}^k}$. Define $$v_j^k \triangleq x_j^{k-1}+ \frac{1}{\alpha} (x_j^k - x_j^{k-1}).$$   
We have
\begin{align*}
     v_j^{k+1} & = x_j^k + \frac{1}{\alpha} (x_j^{k+1} - x_j^k)
    = \mybrace{1-\frac{1}{\alpha}} x_j^k + \frac{1}{\alpha}z_j^k  + \frac{1}{\alpha} (x_j^{k+1}-z_j^k) \\
    & = \alpha z_j^k + \frac{1+\alpha}{\alpha} (z_j^k - x_j^k) - (1+ \alpha) z_j^k + 2x_j^k  + \frac{1}{\alpha} (x_j^{k+1}-z_j^k)    \\
    & = \alpha z_j^k + \frac{1+\alpha}{\alpha} (z_j^k - x_j^k)  + (1-\alpha) x_j^{k-1}  + \frac{1}{\alpha} (x_j^{k+1}-z_j^k)
    \\
    & = \alpha z_j^k + (1-\alpha) v_j^k + \frac{1}{\alpha} (x_j^{k+1}-z_j^k)  = \alpha x_j^{k+1} + (1-\alpha) v_j^k + \mybrace{\frac{1}{\alpha} - \alpha} \mybrace{x_j^{k+1} - z_j^k}.
\end{align*}
As   byproduct of the above derivation, we also have 
$$
    \alpha z_j^k + (1-\alpha) v_j^k = \frac{1}{\alpha} \mybrace{ z_j^k - (1-\alpha) x_j^k}.
$$
Therefore,
\vspace{-0.5cm}
\begin{equation}\label{eq:dist_lyap_part2}
\begin{aligned}
    & \frac{\mu}{2} \norm{v_j^{k+1} - x^\star}^2 = \frac{\mu}{2} \norm{\alpha x_j^{k+1} + (1-\alpha) v_j^k- x^\star}^2 + \frac{\mu}{2} \mybrace{\frac{1}{\alpha} - \alpha}^2 \norm{x_j^{k+1}-z_j^k}^2 \\
    & \qquad + \mu \mybrace{\frac{1}{\alpha} - \alpha} \inn{\alpha x_j^{k+1} + (1-\alpha) v_j^k- x^\star}{x_j^{k+1} - z_j^k}\\
    & = \frac{\mu}{2} \norm{\alpha x_j^{k+1} + (1-\alpha) v_j^k- x^\star}^2 + \frac{\mu}{2} \mybrace{\frac{1}{\alpha} - \alpha}^2 \norm{x_j^{k+1}-z_j^k}^2 + \mu \mybrace{\frac{1}{\alpha} - \alpha} \alpha \norm{x_j^{k+1}-z_j^k}^2\\
    & \qquad + \mu \mybrace{\frac{1}{\alpha} - \alpha} \inn{\alpha z_j^k + (1-\alpha) v_j^k- x^\star}{x_j^{k+1} - z_j^k} \\
    & \leq \frac{(1-\alpha) \mu}{2} \norm{v_j^k-x^\star}^2 + \frac{\alpha \mu}{2} \norm{x_j^{k+1}-x^\star}^2 + \frac{\delta^2 + 2\mu \delta}{2(\mu+\delta)}  \norm{x_j^{k+1}-z_j^k}^2 \\
    & \qquad + \frac{\mu}{\alpha} \mybrace{\frac{1}{\alpha} - \alpha} \inn{z_j^k - (1-\alpha) x_j^k - \alpha x^\star}{x_j^{k+1} - z_j^k}.
\end{aligned}
\end{equation}
  Combining \eqref{eq:dist_lyap_part1} and \eqref{eq:dist_lyap_part2}, we obtain the decay of the potential function $P^k$ defined in (\ref{eq:LYAPU_0}):
\begin{align*}
 P^{k+1}&=p^{k+1} + \frac{1}{m}\sum_{j=1}^m \frac{\mu}{2} \norm{v_j^{k+1} - x^\star}^2 + e^{k,T}\\
& \leq (1-\alpha) \mybrace{p^k + \frac{1}{m}\sum_{j=1}^m \frac{\mu}{2} \norm{v_j^k - x^\star}^2} +   \frac{\delta}{m} \sum_{j=1}^m \mybrace{\norm{z_j^k - \bar{z}^k}^2 + 2 \inn{x_j^{k+1} - \bar{x}^{k+1}}{\bar{z}^k - z_j^k}} \\
& \quad + \underbrace{\mybrace{\mu - \frac{\mu+\delta}{2} - \frac{\mu^2}{2(\mu+\delta)} }}_{\leq 0} \frac{1}{m}\sum_{j=1}^m \norm{x_j^{k+1}-z_j^k}^2\\
& \quad + \frac{\delta}{m} \sum_{j=1}^m \mybrace{ \inn{z_j^k - \bar{z}^k}{x_j^{k+1} - \bar{x}^{k+1}} + (1-\alpha) \inn{z_j^k - \bar{z}^k}{x_j^k - \bar{x}^k} - \norm{z_j^k - \bar{z}^k}^2}\\
& \quad - \frac{\mu+\delta}{m} \sum_{j=1}^m \inn{\alpha x^\star + (1-\alpha) x_j^k-x_j^{k+1} }{x_j^{k+1} - x^{k+1, \star}} + \epsilon^{k+1}  \\
& {\leq} (1-\alpha) \mybrace{p^k + \frac{1}{m}\sum_{j=1}^m \frac{\mu}{2} \norm{v_j^k - x^\star}^2}  + \frac{\delta}{m} \sum_{j=1}^m \mybrace{  (1-\alpha) \inn{z_j^k - \bar{z}^k}{x_j^k - \bar{x}^k} - \inn{z_j^k - \bar{z}^k}{x_j^{k+1} - \bar{x}^{k+1}}} \\
    & \quad - \frac{\mu+\delta}{m} \sum_{j=1}^m \inn{\alpha x^\star - \alpha v_j^{k+1} }{x_j^{k+1} - x^{k+1, \star}} + \epsilon^{k+1}  \\
& \leq (1-\alpha) \underbrace{\mybrace{p^k + \frac{1}{m}\sum_{j=1}^m \frac{\mu}{2} \norm{v_j^k - x^\star}^2 + e^{k-1, T}}}_{=P^k} +  \frac{\delta \alpha}{m} \norm{v_\perp^k} \norm{z_\perp^k }   \\
& \quad + \alpha \frac{\mu+\delta}{m} \sum_{j=1}^m \norm{ x^\star -v_j^{k+1} } \norm{x_j^{k+1} - x^{k+1, \star}} + \epsilon^{k+1} \\
& \leq (1-\alpha)\, P^k + \frac{\delta}{1+\alpha}  \, \frac{1}{m}  \mybrace{\norm{x_\perp^{k+1}} + (1-\alpha) \norm{x_\perp^k}} \mybrace{2 \norm{x_\perp^k} + (1-\alpha) \norm{x_\perp^{k-1}}} \medskip  \\
& \quad + \alpha (\mu+\delta)  \sqrt{\frac{1}{m} \sum_{j=1}^m \norm{ x^\star -v_j^{k+1} }^2} \sqrt{\frac{1}{m} \sum_{j=1}^m \norm{x_j^{k+1} - x^{k+1, \star}}^2}   + \epsilon^{k+1}   \end{align*}
\begin{align*}
& \leq (1-\alpha) \,P^k  + \frac{\delta}{1+\alpha}  \, \frac{1}{m}  \mybrace{ \frac{3}{2} \norm{x_\perp^{k+1}}^2 + \mybrace{\frac{7}{2}-3\alpha  +\frac{\alpha^2}{2}} \norm{x_\perp^k}^2 + (1-\alpha)^2  \norm{x_\perp^{k-1}}^2}\\
& \quad + \sqrt{2} \alpha \sqrt{\mu+\delta } \sqrt{\frac{1}{m} \sum_{j=1}^m \norm{ x^\star -v_j^{k+1} }^2} \cdot \sqrt{g^{k,T}}  + \epsilon^{k+1}.
\end{align*}

We proceed now to bound $g^{k,T}$ and the   consensus error terms   by $\epsilon^{k+1}$. For the former we readily have $g^{k,T}\leq \epsilon^{k+1}$, due to (\ref{eq:termination}). For the latter, 
  we choose    $\{\epsilon^{k}\}$ as \begin{equation}\label{eq:epsilon_def}
    \epsilon^{k+1} \triangleq  \epsilon^{k} \mybrace{1-c\alpha},\quad k=0,1,\ldots,
\end{equation} where $c$ is any constant in $(0,1)$ and $\epsilon^0$ is to be determined; and  use     $e^{k,T}\leq \epsilon^{k+1}$, for any $k= 0,1\ldots $ (still due to (\ref{eq:termination})).  
We can   write 
\begin{align*}
   P^{k+1} \leq (1-\alpha) P^k + \sqrt{2 \mu}  \sqrt{\frac{1}{m} \sum_{j=1}^m \norm{ x^\star -v_j^{k+1} }^2} \cdot \sqrt{\epsilon^{k+1}}   + c_1 \epsilon^{k+1},
\end{align*}
where 
$$
c_1 = 1+ \frac{\delta}{c_x } \frac{\frac{3}{2} (1-c\alpha)^2 + 5- 4c\,\alpha}{(1-c\alpha)^2}.$$
  Define $\lambda^k \triangleq (1-\alpha)^k;$  we have
\begin{align}\label{eq:Lyap_dec}
  \frac{P^{k+1}}{\lambda^{k+1}} & \leq \frac{P^k }{\lambda^k} + \sqrt{2 \mu}  \sqrt{\frac{1}{m} \sum_{j=1}^m \norm{ x^\star -v_j^{k+1} }^2} \cdot \frac{\sqrt{\epsilon^{k+1}}}{\lambda^{k+1}} + c_1 \frac{\epsilon^{k+1}}{\lambda^{k+1}} \nonumber  \\
& \leq P^0 + c_1 \sum_{t=1}^{k+1} \frac{\epsilon^t}{\lambda^t} + \sum_{t=1}^{k+1} \frac{\sqrt{2\mu \epsilon^t}}{\lambda^t} \sqrt{\frac{1}{m} \sum_{j=1}^m \norm{ x^\star -v_j^{k+1} }^2}.
\end{align} Introducing the following quantities:  
\begin{equation}\label{eq:def_Lyap_terms}
    \hat{v}^k \triangleq  \sqrt{\frac{\mu}{2\lambda^k}} \cdot \sqrt{\frac{1}{m} \sum_{j=1}^m \norm{v_j^k-x^\star}^2}, \quad a^k \triangleq 2 \sqrt{ \frac{\epsilon^k}{\lambda^k}},\quad \text{and} \quad S^k \triangleq P^0 + c_1 \sum_{t=1}^{k} \frac{\epsilon^t }{\lambda^t},\end{equation}
(\ref{eq:Lyap_dec}) can be rewritten as 
\begin{align}\label{pre_lemma_A10}
    (\hat{v}^k)^2 \leq S^k+ \sum_{t=1}^k a^t \hat{v}^t.
\end{align}
Using \eqref{pre_lemma_A10} we can invoke  \cite[Lemma~A.10]{lin2015Catalyst} and conclude     
\begin{align}\label{eq:PRE_RATE_0}
  S^k+  \sum_{t=1}^k a^t \hat{v}^t\leq  \mybrace{\sqrt{S^k} + \sum_{t=1}^k a^t}^2.   
\end{align}
This,  together with (\ref{eq:Lyap_dec}), yields 
\begin{align*}
    P^k \leq \lambda^k \mybrace{\sqrt{S^k} + \sum_{t=1}^k a^t}^2  \overset{\eqref{eq:def_Lyap_terms}}{\leq} \lambda^k \mybrace{\sqrt{L^0} + (2+\sqrt{c_1}) \sum_{t=1}^{k} \sqrt{\frac{\epsilon^t }{\lambda^t}}}^2. 
\end{align*}
Choosing  $\epsilon^0=P^0$, resulting in 
    $\epsilon^k = P^0\cdot (1-c\cdot \alpha)^{k}
$ [cf.~\eqref{eq:epsilon_def}], leads to the desired result:
\begin{align*}
    P^k \leq c_2 \,P^0\, (1-c\cdot \alpha)^{k+1}, 
\end{align*}
with $$c_2 = \frac{(2+\sqrt{c_1})^2 }{\mybrace{\sqrt{\frac{1-c\cdot \alpha}{1-\alpha}} -1 }^2 (1-\alpha)}.$$
\end{proof}
\subsection{Step 2: \texttt{SONATA} and Inner-Loop Convergence}\label{sec_Step2}
 Given Proposition~\ref{prop_step1}, to complete the proof of Theorem~\ref{thm:cata_net} and Theorem~\ref{crl:cata_net_lin},    we need to show that the error condition \eqref{eq:termination} is satisfied if \texttt{SONATA} runs for $T=\tilde{\mathcal O}(1)$ iterations; the  explicit expression of $T$   depends on  the specific setting considered for the algorithm, and it is different for the one  specified in  Theorem~\ref{thm:cata_net} and  in  Theorem~\ref{crl:cata_net_lin}--the two cases are studied separately in Sec.~\ref{sec_proof_th_4} and Sec.~\ref{sec_proof_th_5}, respectively. 
 
 To control the number of calls of \texttt{SONATA} we  need to bound $g^{k,0} + e^{k,0}$, which is done in the following lemma.

\begin{lemma}\label{lm:good_init}
Instate the setting of Proposition~\ref{prop_step1}. 
For $k\geq 1$, if $g^{k-1,T}+e^{k-1,T} \leq \epsilon^k$, then the following holds: 
   \begin{equation} \label{termination_zero_bound}
   g^{k,0} + e^{k,0} \leq P^0 \mybrace{2 (1-c\cdot \alpha)^k + \max \left( \frac{1}{\mu+\delta}, \, 2\,c_y\right) \dfrac{72\, \delta^2}{\mu}\, c_2 \, (1-c\cdot \alpha)^{k-1}}.
\end{equation}
\end{lemma}

\begin{proof}
Since $$g^{k,0} + e^{k,0}   = \frac{1}{m} \sum_{i=1}^m \mybrace{u_k(x_i^k)- u_k^\star} +  c_x\norm{x_\perp^{k,0}}^2 + c_y \norm{y_\perp^{k,0}}^2,$$
we begin  bounding $\frac{1}{m} \sum_{i=1}^m \mybrace{u_k(x_i^k)- u_k^\star}$. Postponing the proof to the end of this section,  we show that  the following   holds 
\begin{equation}\label{eq:bound_i}
     \dfrac{1}{m} \displaystyle{\sum}_{i=1}^m  u_k(x_i^k)- u_k^\star \leq  2\epsilon^k - 2e^{k-1,T} + \dfrac{\delta^2}{\mu+\delta} \norm{\bar{z}^k - \bar{z}^{k-1}}^2.
\end{equation}
Therefore, we can write 
\begin{align*}
     g^{k,0} + e^{k,0}  
    &\overset{\eqref{eq:bound_i}}{ \leq}  2\epsilon^k - 2e^{k-1,T} + \frac{\delta^2}{\mu+\delta} \norm{\bar{z}^k - \bar{z}^{k-1}}^2 +   c_x\norm{x_\perp^k}^2 + 2\,c_y \norm{y_\perp^k}^2 + 2\,c_y\, \delta^2 \norm{(z^{k-1}-z^k)_\perp}^2 \\
    & \leq 2\epsilon^k + \max \left( \frac{1}{\mu+\delta}, \, 2\,c_y\right) \delta^2  \frac{1}{m} \sum_{i\in[m]}   \norm{z_i^k - z_i^{k-1}}^2\\
    & \leq 2\epsilon^k + \max \left( \frac{1}{\mu+\delta}, \, 2\,c_y\right) \frac{72\, \delta^2}{\mu}  \frac{P^{k} }{(1-c\cdot \alpha)^2} \\
    & = P^0 \mybrace{2 (1-c\cdot \alpha)^k + \max \left( \frac{1}{\mu+\delta}, \, 2\,c_y\right) \frac{72\, \delta^2}{\mu}\, c_2 \, (1-c\cdot \alpha)^{k-1}}.
\end{align*}

 It remains to show that (\ref{eq:bound_i}) holds. We prove it following similar ideas as in   \cite[ Lemma~B.1]{lin2015Catalyst}, with differences due to  the distributed setting.
Writing 
\begin{align*}
    u_k(x) - u_{k-1}(x) - \mybrace{u_k(z) - u_{k-1}(z)} 
      &= \frac{\delta}{2m} \sum_{i=1}^m \mybrace{\norm{x-z_i^k}^2 - \norm{x-z_i^{k-1}}^2 - \norm{z-z_i^k}^2 + \norm{z-z_i^{k-1}}^2} \\
    & = \delta \inn{\bar{z}^{k-1} - \bar{z}^k}{x-z},
\end{align*}
  we have
\begin{align*}
    & u_k(x_i^k) - u_k(x^{k+1,\star})\\
    & = u_{k-1}(x_i^k) - u_{k-1}(x^{k+1,\star}) + \delta \inn{\bar{z}^{k-1}- \bar{z}^k}{x_i^k - x^{k+1,\star}} \\
    & = u_{k-1}(x_i^k) - u_{k-1}(x^{k,\star}) + u_{k-1}(x^{k,\star}) -  u_{k-1}(x^{k+1,\star}) + \delta \inn{\bar{z}^{k-1}- \bar{z}^k}{x_i^k - x^{k+1,\star}} \\
    & \leq u_{k-1}(x_i^k) - u_{k-1}(x^{k,\star}) - \frac{\mu+\delta}{2} \norm{x^{k,\star} - x^{k+1,\star}}^2 + \delta \inn{\bar{z}^{k-1}- \bar{z}^k}{x_i^k - x^{k+1,\star}}.
\end{align*}
Therefore, 
\begin{equation}\label{eq:B.1-1}
\begin{aligned}
    & \frac{1}{m} \sum_{i=1}^m \mybrace{ u_k(x_i^k) - u_k(x^{k+1,\star})} \\
    & \leq \epsilon^k - e^{k-1,T} - \frac{\mu+\delta}{2} \norm{x^{k,\star} - x^{k+1,\star}}^2 + \delta \inn{\bar{z}^{k-1}- \bar{z}^k}{\bar{x}^k - x^{k+1,\star}}.
\end{aligned}
\end{equation}
At the same time, it holds 
\begin{equation}\label{eq:B.1-2}
\begin{aligned}
    & \delta \inn{\bar{z}^{k-1}- \bar{z}^k}{ x^{k,\star}- x^{k+1,\star}} \leq \frac{\mu+\delta}{2} \norm{x^{k,\star}- x^{k+1,\star}}^2 + \frac{\delta^2}{2(\mu+\delta)} \norm{\bar{z}^{k-1}- \bar{z}^k}^2, \\
    & \delta \inn{\bar{z}^{k-1}- \bar{z}^k}{\bar{x}^k - x^{k,\star} } \leq \frac{\mu+\delta}{2} \norm{\bar{x}^k - x^{k,\star}}^2 + \frac{\delta^2}{2(\mu+\delta)} \norm{\bar{z}^{k-1}- \bar{z}^k}^2.
\end{aligned}
\end{equation}
Combining \eqref{eq:B.1-1} and \eqref{eq:B.1-2}, leads to
\begin{align*}
    & \frac{1}{m} \sum_{i=1}^m \mybrace{ u_k(x_i^k) - u_k(x^{k+1,\star})} \leq \epsilon^k - e^{k-1,T}  + \frac{\delta^2}{\mu+\delta} \norm{\bar{z}^{k-1}- \bar{z}^k}^2 + \frac{\mu+\delta}{2} \norm{\bar{x}^k - x^{k,\star}}^2 \\
    & \leq \epsilon^k - e^{k-1,T}  + \frac{\delta^2}{\mu+\delta} \norm{\bar{z}^{k-1}- \bar{z}^k}^2 +  \frac{1}{m} \sum_{i=1}^m \mybrace{ u_{k-1}(x_i^k) - u_{k-1}(x^{k,\star})} \\
    & \leq 2\epsilon^k - 2 e^{k-1,T}  + \frac{\delta^2}{\mu+\delta} \norm{\bar{z}^{k-1}- \bar{z}^k}^2,\qquad \text{for } k\geq 1. 
\end{align*}
\end{proof}

\subsubsection{Proof of Theorem~\ref{thm:cata_net}}\label{sec_proof_th_4}

We begin proving convergence of \texttt{SONATA} in the setting of Theorem~\ref{thm:cata_net}--we refer to such an instance of  \texttt{SONATA} as  \texttt{SONATA-F}. Note that   convergence established in \cite{sun2019distributed} is not directly applicable here. First, there is a mismatch between the gradient tracking initialization therein and the one used in our setting. Second,  $R$-linear convergence of the optimality gap as in \cite{sun2019distributed}  is not enough to certify that the termination criterion \eqref{eq:termination} is  satisfied   after a finite number of iterations $T=\tilde{\mathcal O}(1)$.  
Our refined convergence analysis of   \texttt{SONATA-F}     is stated in Lemma~\ref{lm:rev_sonata} below.  
 
 Notice that,  with $\delta = \beta -\mu,$ we have  that: (i) the objective function $u_k(x)$ is $\mu_{\textnormal{u}_k} = \beta$-similar and $\beta$-strongly convex; and (ii) every $f^k_i(x)$ is $L_{\text{mx}} = L+2\beta -\mu$ smooth. By using the surrogates (\ref{surrogates_f_i}), \texttt{SONATA-F}  takes advantage of  similarity and achieves  linear convergence rate, scaling with  $\mathcal{O}({\beta}/{\mu_{\textnormal{u}_k}})=\mathcal{O}(1)$.   

\begin{lemma}\label{lm:rev_sonata}
Consider the minimization of    $u_k(x)$ wherein   $\delta = \beta - \mu$, running \texttt{SONATA-F} (initialized as in \texttt{ACC-SONATA}). 
With   $$c_x = \frac{8 (L+2\beta - \mu)^2}{\beta},\quad  c_y = \frac{4}{\beta},$$ and the network connectivity $\rho$ satisfying 
\begin{align}\label{eq:cond_rho_SONATA_F}
\rho \leq   \frac{1}{4\sqrt{1785}}\frac{\beta (2\beta - \mu)}{(L+2\beta - \mu) (L+4\beta - \mu)},
\end{align}
\texttt{SONATA-F} converges  Q-linearly, that is,  
\begin{align}\label{eq:contraction_inner}
    g^{k,t+1}+e^{k,t+1} \leq \frac{33}{34} \mybrace{g^{k,t}+e^{k,t}}.
\end{align}

\end{lemma}

\begin{proof} The proof  builds on some intermediate results in  \cite{sun2019distributed}; when recalled here, we use the same notation as defined  therein. 

 Consider \cite[ Proposition~3.4]{sun2019distributed}, and    set therein  $\epsilon_{opt} = \frac{1}{2}(2\beta - \mu)$; we    get $$\sigma(1) \leq \frac{16}{17}\quad \text{and} \quad \eta(1) \leq \frac{18}{17\beta}. $$  Therefore, 
 \begin{align}\label{3.4}
     g^{k,t+1} \leq  \frac{16}{17}\, g^{k,t} + \frac{9}{17} e^{k,t}.
 \end{align}
According to \cite[ Proposition~3.5]{sun2019distributed}, we have 
\begin{align*}
    & \norm{x_\perp^{k,t+1}}^2 \leq 2 \rho^2 \norm{x_\perp^{k,t}}^2 + 2 \rho^2 \frac{1}{m} \norm{d^{k,t}}^2, \\
    & \norm{y_\perp^{k,t+1}}^2 \leq 3 \rho^2 \norm{y_\perp^{k,t}}^2 + 12 L_{\text{mx}}^2 \rho^2 \norm{x_\perp^{k,t}}^2 + 3L_{\text{mx}}^2 \rho^2 \frac{1}{m} \norm{d^{k,t}}^2;
\end{align*}
which leads to
\begin{equation}\label{3.5}
\begin{aligned}
 4 L_{\text{mx}}^2 \norm{x_\perp^{k,t+1}}^2 + 2 \norm{y_\perp^{k,t+1}}^2 \leq \rho^2\,  \mybrace{32 L_{\text{mx}}^2 \norm{x_\perp^{k,t}}^2 + 6 \norm{y_\perp^{k,t}}^2 + 14 L_{\text{mx}}^2 \frac{1}{m} \norm{d^{k,t}}^2 }.
\end{aligned}
\end{equation}
 \cite[ Proposition~3.6]{sun2019distributed} becomes
\begin{align}\label{3.6}
    \frac{1}{m} \norm{d^{k,t}}^2 \leq \frac{6}{\beta} \frac{(4 \beta - \mu)^2 + 4L_{\text{mx}}^2}{(2\beta - \mu)^2}  \,g^{k,t} + \frac{3}{(2\beta - \mu)^2} \norm{y_\perp^{k,t}}^2.
\end{align}
Combining \eqref{3.4}, \eqref{3.5} and \eqref{3.6} leads to
\begin{align*}
    & g^{k,t+1}+e^{k,t+1} \leq \mybrace{\frac{16}{17} + \rho^2   \frac{168 L_{\text{mx}}^2 }{\beta^2}  \frac{(4\beta -\mu)^2 + 4 L_{\text{mx}}^2 }{(2\beta - \mu)^2}} g^{k,t} \\
    & \qquad\qquad\qquad \quad+ \mybrace{\frac{9}{17} + \rho^2 \max \mybrace{8,\,\,  3+21\frac{L_{\text{mx}}^2}{(2\beta-\mu)^2}}} e^{k,t}.
\end{align*}
It is not difficult to  check  that   contraction of $g^{k,t}+e^{k,t}$ as in (\ref{eq:contraction_inner}) holds when $\rho$ satisfies
\begin{align*}
    & \rho \leq \min \left( \sqrt{\frac{15}{272}}, \,\,\sqrt{\frac{15}{34 \left( 3+\frac{21 L_{\text{mx}}^2}{(2\beta - \mu)^2} \right)}}, \,\, \sqrt{\frac{\beta^2 }{5712 L_{\text{mx}}^2 }  \frac{(2\beta - \mu)^2}{(4\beta -\mu)^2 + 4 L_{\text{mx}}^2 }  } \right) .
\end{align*}
A sufficient condition for that is  \eqref{eq:cond_rho_SONATA_F}.
\end{proof}


Invoking  Lemma~\ref{lm:good_init}, the number of inner iterations needed for (\ref{eq:termination}) to hold, for any $k=1,2,\ldots,$  can be bounded as
\begin{equation}\label{eq:F_T}
\begin{aligned}
T & \leq  \ceil{34 \, \log \frac{g^{k,0} + e^{k,0}}{\epsilon^{k+1}}}  \overset{\eqref{termination_zero_bound}}{\leq}  \ceil{34 \, \log \frac{P^0 \mybrace{2 (1-c\cdot \alpha)^k + \frac{576(\beta-\mu)^2}{\mu\beta} c_2\, (1-c\cdot \alpha)^{k-1}}}{P^0\, (1-c\cdot \alpha)^{k+1}}} \\
    & = \ceil{34 \, \log \frac{2 (1-c\cdot \alpha) + \frac{576(\beta-\mu)^2}{\mu\beta} c_2 }{(1-c\cdot \alpha)^2}} \\
    & = \ceil{34 \, \log \mybrace{ \frac{2}{1-c\cdot \sqrt{ \mu / \beta}}   + \frac{ 576(\beta-\mu)^2 }{\mu\beta (1-c\cdot \sqrt{ \mu / \beta})^2}  \frac{\mybrace{2+\sqrt{1+ \frac{(\beta - \mu)\beta}{8 (L+2\beta - \mu)^2 } \frac{\frac{3}{2} (1-c\sqrt{ \mu / \beta})^2 + 5- 4c\,\sqrt{ \mu / \beta}}{(1-c\sqrt{ \mu / \beta})^2}}}^2 }{\mybrace{\sqrt{\frac{1-c\cdot \sqrt{ \mu / \beta}}{1-\sqrt{ \mu / \beta}}} -1 }^2 (1-\sqrt{ \mu / \beta})} }} \\
    & = \mathcal{O} \mybrace{\log\frac{\beta}{\mu}}.
\end{aligned}
\end{equation}

For $k=0$,
due to $g^{0,0}+e^{0,0}\leq P^0,$ we have
$T = \ceil{34 \, \log \frac{1}{1-c\cdot \alpha}}$, which is smaller than the RHS in  \eqref{eq:F_T}. 

This, together with   Proposition~\ref{prop_step1} completes the proof of Theorem~\ref{thm:cata_net}.$\hfill \square$
 
\subsubsection{Proof of Theorem~\ref{crl:cata_net_lin}}  \label{sec_proof_th_5}
We study now convergence of \texttt{SONATA} in the setting of Theorem~\ref{crl:cata_net_lin}--we refer to such an instance as  \texttt{SONATA-L}. 

We begin noticing that, with  $\delta = L -\mu,$ we have the following properties for $u_k(x)$ and $f^k_i(x)$'s: (i)   $u_k(x)$ is $(2L-\mu)$-smooth, $\beta$-similar and $L$-strongly convex; and (ii) every $f^k_i(x)$ is $L_{\text{mx}} = 2L+\beta -\mu$ smooth. Hence, when using the linearization surrogates (\ref{surrogates_linearization}),  \texttt{SONATA-L} achieves linear rate, scaling as $\mathcal{O}((2L-\mu)/L)=\mathcal{O}(1)$. 
We  establish  such a result below, following the same path as   in the proof of Lemma~\ref{lm:rev_sonata}.  

\begin{lemma}\label{lm:rev_sonata-L}
Consider the minimization of    $u_k(x)$ wherein   $\delta = L- \mu$, running \texttt{SONATA-L} (initialized as in \texttt{ACC-SONATA}). 
With $$c_x = \frac{56(2L+\beta -\mu)^2}{L}, \quad  c_y = \frac{28}{L},$$ and the network connectivity $\rho$ satisfying 
\begin{align}\label{eq:cond_on_rho_SONATA_L}
\rho \leq   \frac{1}{70\sqrt{15}} \frac{L^2}{\mybrace{2L-\mu+\beta}^2 },
\end{align}
\texttt{SONATA-L} converges   Q-linearly, that is,  
\begin{align*}
    g^{k,t+1}+e^{k,t+1} \leq \frac{9}{10} \mybrace{g^{k,t}+e^{k,t}}.
\end{align*}
\end{lemma}

\begin{proof}  
 We set  $\epsilon_{opt}$ in  \cite[ Proposition~3.4 ]{sun2019distributed} as  $\epsilon_{opt} = \frac{L}{2}$; one gets therein $\sigma(1) \leq \frac{4}{5}$ and $\eta(1) \leq \frac{7}{L} $.  Therefore, 
 \begin{align}\label{3.4_L}
     g^{k,t+1} \leq  \frac{4}{5} g^{k,t} + \frac{1}{2} e^{k,t}.
 \end{align}
In addition,  \cite[Proposition~3.6]{sun2019distributed} becomes
\begin{align}\label{3.6_L}
    \frac{1}{m} \norm{d^{k,t}}^2 \leq \frac{6}{L} \mybrace{\frac{9}{4} + \frac{4 L_{\text{mx}}^2}{(2L-\mu)^2}}g^{k,t} + \frac{3}{(2L-\mu)^2} \norm{y_\perp^{k,t}}^2.
\end{align}
Combining \eqref{3.5}, \eqref{3.4_L} and \eqref{3.6_L}, leads to
\begin{align*}
    & g^{k,t+1}+e^{k,t+1} \leq \mybrace{\frac{4}{5} + 1176\, \rho^2  \frac{L_{\text{mx}}^2}{L^2}  \mybrace{\frac{9}{4} + \frac{4 L_{\text{mx}}^2}{(2L-\mu)^2}}} g^{k,t}  + \mybrace{\frac{1}{2} + \rho^2 \max \mybrace{8,\,\,  3+ 21 \frac{L_{\text{mx}}^2}{(2L-\mu)^2}}} e^{k,t}
\end{align*}
A contraction on $g^{k,t}+e^{k,t}$ is ensured choosing 
\begin{align*}
    & \rho \leq \min \left( \frac{1}{2\sqrt{5}}, \,\,\sqrt{\frac{2}{5\mybrace{3+21 \frac{L_{\text{mx}}^2}{(2L-\mu)^2}}}}, \,\, \frac{1}{28} \sqrt{\frac{1}{15 \frac{L_{\text{mx}}^2}{L^2} \mybrace{\frac{9}{4} + \frac{4L_{\text{mx}}^2}{(2L-\mu)^2}} }  } \right).
\end{align*}
The above is satisfied if $\rho$ satisfies \eqref{eq:cond_on_rho_SONATA_L}. 
\end{proof}

We can conclude the proof of Theorem~\ref{thm:cata_net}, using Lemma~\ref{lm:good_init} to determine the number of inner iterations needed for (\ref{eq:termination}) to hold:  for any $k=1,2,\ldots,$  we have 
\begin{equation}\label{eq:L_T}
\begin{aligned}
T & \leq  \ceil{10 \, \log \frac{g^{k,0} + e^{k,0}}{\epsilon^{k+1}}}  \leq \ceil{ 10 \log\frac{2(1-c\cdot\alpha)+\frac{4032 (L-\mu)^2}{L \mu }c_2}{(1-c\cdot\alpha)^2}}  = \mathcal{O} \mybrace{\log \kappa}.
\end{aligned}
\end{equation}
 For $k=0$,
due to $g^{0,0}+e^{0,0}\leq P^0,$ we have
$T = \ceil{10 \, \log \frac{1}{1-c\cdot \alpha}}$, which is smaller than the value of $T$ in \eqref{eq:L_T} for $k\geq 1$.

 \section{\texttt{ACC-SONATA} OVER STAR-NETWORKS}\label{app_SONATA_star}

In this section we customize \texttt{ACC-SSONATA} to the special setting of master/workers architectures. 
Specifically, consider Problem \eqref{eq:problem} over a star (unidirected)  graph with $m$ nodes, where one of them (the master node) connects with all the others (workers). The workers  still own only one function $f_i$ of the sum-cost $f$ in \eqref{eq:problem}. The application of  \texttt{ACC-SSONATA} to such a setting boils down to customize the inner algorithm \texttt{SONATA}--we use the instance in \cite[Algorithm 3]{sun2019distributed}, which is reported in Algorithm~\ref{alg:SONATA-star} below (applied to \eqref{eq:problem}) for convenience. Note that the consensus step is now replaced by the exact average performed by the master node (see \texttt{(S.4)}), based upon reception of the local optimization variables  from the workers. Also, there is no need of the gradient-tracking mechanism, as   the master node can directly broadcast to the workers the aggregate gradient $\nabla f(x^k)$. Notice  that \texttt{SONATA-star} can be interpreted as a special instance of \texttt{SONATA} described in Algorithm~\ref{alg:SONATA} (up to a proper initialization) if the weight matrix $W$ therein is set to  $W=[1,0_{m,m-1}][1/m,0_{m,m-1}]^\top$, whose associated $\rho=\|W-11^\top/m\|=0$.  

 Equipped with \texttt{SONATA-star}, \texttt{ACC-SONATA-star} reduces to Algorithm~\ref{alg:ACCSONATA-star} below.
 Convergence as discussed in the main text is a consequence of Theorem~\ref{thm:cata_net} and Theorem~\ref{crl:cata_net_lin}. Note that the condition of $\rho$ is trivially satisfied as $\rho=0$.

  
  \begin{algorithm}[h]
\caption{\texttt{SONATA-star}$(\left\{ f_i \right\}_{i\in[m]},\, (x_i^0)_{i\in [m]},  T)$}
 {\bf Input}: $\left\{ f_i(x) \right\}_{i\in[m]},\, \,r(x)$ [cf. \eqref{eq:problem}];
 
 \hspace{1.2cm} $(x_i^0)_{i\in [m]}$ 
 [initialization points],
 
 \hspace{1.2cm} $T>0$ 
 [\# iterations];\\
 {\bf Output}: $x^T$;

\textbf{for} {$k=0,1,2,\ldots,T-1$} \textbf{do}\medskip 

  \texttt{ (S.1):}  Each worker $i$ evaluates $\nabla f_i(\bx^k)$ and sends it to the master node;\medskip 
 
 \texttt{ (S.2):}  The master   broadcasts  $\nabla f(\bx^k)=1/m \sum_{i=1}^m \nabla f_i(\bx^k)$ to the workers;\medskip 
 
  \texttt{ (S.3):} Each worker $i$ computes   
$$   x_i^{k+1/2} \triangleq {}   \argmin_{x_i \in \mathbb{R}^d} {\tilde{f}_i (x_i ;x^k) + \big( \nabla f(x^k) - \nabla f_i (x^k) \big)^\top (x_i - x^k)} + r(x_i),$$
 \qquad \qquad and sends  $ x_i^{k+1/2}$ to the master;\medskip 
 
 \texttt{ (S.4):}  The master computes the average \vspace{-0.2cm}
 \begin{equation*}x^{k+1}= \frac{1}{m}\sum_{i=1}^m  x_i^{k+1/2},\vspace{-0.2cm}\end{equation*} 
 \qquad \qquad and sends it back to the workers.
 
\textbf{end for}\label{alg:SONATA-star}
\end{algorithm}

 \begin{algorithm}[!ht]
\caption{\texttt{Accelerated SONATA-star}}\label{alg:ACCSONATA-star} 
 {\bf Input}: $\beta$, $\mu$, $\delta>0$, $\alpha = \sqrt{\mu/(\mu+\delta)}$; 
 
 \hspace{1.1cm} $x^0=z^0  =0$; \\
 {\bf Output}: $x^K$ \\
\textbf{for} {$k=0,1,2,\ldots,K-1$} \textbf{do}\bigskip 

\,\,\,\texttt{Set:} \qquad $ f_i^k(x) = f_i(x) + \frac{\delta}{2}\norm{x-z^k}^2
,$\bigskip 


\,\,\texttt{(S.1) Inner loop via  \texttt{SONATA-star}:} 
    \begin{align*}   x^{k+1} = 
     \texttt{SONATA-star} \Big( \left\{ f_i^k \right\}_{i\in[m]},\,  x^k,  \, T \Big);
    \end{align*}
    \,\,\texttt{(S.2)     Extrapolation step:} 
\[ z^{k+1} = x^{k+1} + \frac{1-\alpha}{1+\alpha} \,(x^{k+1} - x^k).
\]
\textbf{end for}
\end{algorithm}

 \section{SOLVING AGENTS' SUBPROBLEMS INEXACTLY}\label{app_inexact}
There are applications wherein   the local optimization problems in the Algorithm~\ref{alg:SONATA}
\begin{align*}
      {x}_i^{k+1/2} = \argmin_{x\in \mathbb{R}^d}  \,\,\tilde{f}_i(x; x_i^k)  + \inn{y_i^k - \nabla f_i(x_i^k)}{x - x_i^k}+r(x),
\end{align*}
do not have a closed-form solution or cannot be   solved efficiently to arbitrary precision, especially when the surrogate function \eqref{surrogates_f_i} is adopted.  In this section, we discuss how to modify  \texttt{ACC-SONATA-F} (and by-product \texttt{ACC-SONATA-L})   to accommodate  computations of  inexact  solutions of   agents' subproblems.  
We prove that, by carefully choosing the {\it inexact  criterion} for solving  approximately  the local optimization subproblems, the communication complexity of the resulting inexact \texttt{ACC-SONATA-F}, termed \texttt{Inexact ACC-SONATA-F} (Algorithm~\ref{alg:inexact_SF}), matches that of \texttt{ACC-SONATA-F} as in   \eqref{eq:num_comm_sim} (see Theorem~\ref{thm:ine_cata_f}).  We also study  the computational complexity of \texttt{Inexact-ACC-SONATA-F} (see Theorem~\ref{thm:ine_cata_f_cp}).


We begin introducing the inexact instance of the SONATA algorithm, termed \texttt{Inexact-SONATA}, and described in Algorithm~\ref{alg:inexact_SF}. Notice the presence therein of an additional input sequence $\left\{\xi^k \right\}_{k=0}^{T-1} =   \left\{(\xi_i^k)_{i\in [m]} \right\}_{k=0}^{T-1}$ (to be properly chosen), determining the accuracy the solution of the agents' subproblems (\ref{eq:ine_loc_opt}) in \texttt{(S.1)} is estimated.   

   
 \begin{algorithm}[h] 
\caption{\texttt{Inexact-SONATA}\label{alg:inexact_SF}$(\left\{ f_i \right\}_{i\in[m]},\, x^0, \, y^0,\, T, \, \left\{\xi^k \right\}_{k=0}^{T-1})$}
 {\bf Input}: $\left\{ f_i(x) \right\}_{i\in[m]},\, \,r(x)$ [cf. \eqref{eq:problem}];
 
 \hspace{1.2cm} $x^0=(x_i^0)_{i\in [m]}$ 
 [initialization points],
 
 \hspace{1.2cm} $y^0=(y_i^0)_{i\in [m]}$ [gradient-tracking initialization],
 
 \hspace{1.2cm} $T>0$ 
 [\# iterations],
 
 \hspace{1.2cm} $\left\{\xi^k \right\}_{k=0}^{T-1} =   \left\{(\xi_i^k)_{i\in [m]} \right\}_{k=0}^{T-1} $ 
 [inexactness parameters];\\
 {\bf Output}: $x^T=\big(x_i^T\big)_{i\in [m]},\,  \, y^T=\big( y_i^T\big)_{i\in [m]}$; 

\textbf{for} {$k=0,1,2,\ldots,T-1$} \textbf{do}

  \texttt{ (S.1) Local computations:} for all $i\in [m]$, \vspace{-0.2cm}
  \begin{equation}\label{eq:ine_loc_opt}
  \begin{aligned}
      & \tilde{x}_i^{k+1/2} \approx \argmin_{x\in \mathbb{R}^d}  \,\, u_i^k(x) \triangleq \tilde{f}_i(x; x_i^k)  + \inn{y_i^k - \nabla f_i(x_i^k)}{x - x_i^k}+r(x), \\
      & \textbf{s.t.} \quad u_i^k(\tilde{x}_i^{k+1/2}) - \min_{x\in \mathbb{R}^d} u_i^k(x) \leq \xi_i^k;
  \end{aligned}
  \end{equation}
 
 \texttt{ (S.2) Communications:}  for all $i\in [m]$,\vspace{-0.2cm}
   \begin{align*}
    & x_i^{k+1} = \sum_{j=1}^m w_{ij} \tilde{x}_j^{k+1/2}, \\
    & y_j^{k+1} = \sum_{j=1}^m w_{ij} \big( y_j^k + \nabla f_j(x_j^{k+1}) - \nabla f_j(x_j^k)\big).
   \end{align*} 
\textbf{end for}
\end{algorithm}

 \begin{algorithm}[!ht]
\caption{\texttt{Inexact Accelerated SONATA}}\label{alg:ine_ACCSONATA} 
 {\bf Input}: $\beta$, $\mu$, $\delta>0$, $\alpha = \sqrt{\mu/(\mu+\delta)}$, $\left\{\xi^{k,t} \right\}$; 
 
 \hspace{1.1cm} $x_i^0=z_i^0 =z_i^{-1} =0$, $y_i^0 = \nabla f_i(x_i^0)$ \\
 {\bf Output}: $x^K = (x_i^K)_{i\in [m]}$ \\
\textbf{for} {$k=0,1,2,\ldots,K-1$} \textbf{do}\smallskip 

\,\,\,\texttt{Set:} \qquad $ f_i^k(x) = f_i(x) + \frac{\delta}{2}\norm{x-z_i^k}^2;$\medskip 
 
\,\,\texttt{(S.1) Inner loop via  \texttt{SONATA}:} 
    \begin{align*}  \big(x^{k+1},\,  y^{k+1}\big) =  \texttt{Inexact-SONATA} \Big( \left\{ f_i^k \right\}_{i\in[m]},\,  x^k, \, y^k+ \delta\,  \mybrace{z^{k-1} - z^k}, \, T, \, \left\{\xi^{k,t} \right\}_{t=0}^{T-1}\Big);
    \end{align*}
    \,\,\texttt{(S.2)     Extrapolation step:} 
\[ z_i^{k+1} = x_i^{k+1} + \frac{1-\alpha}{1+\alpha} \,(x_i^{k+1} - x_i^k),\quad \forall i \in[m].
\]
\textbf{end for}
\end{algorithm}
 
Equipped with   \texttt{Inexact-SONATA} Algorithm, the inexact version of \texttt{ Acc-SONATA} is presented in   Algorithm~\ref{alg:ine_ACCSONATA}, where in the inner loop is now  invoked \texttt{Inexact-SONATA}. 

The next theorem studies convergence of \texttt{Inexact ACC-SONATA}; we focus on \texttt{Inexact ACC-SONATA-F}, i.e., the instance of \texttt{Inexact ACC-SONATA} using   \eqref{surrogates_f_i} as 
   surrogate functions   in the step \texttt{(S.1)} of \texttt{Inexact SONATA}.  With a properly chosen accuracy sequence, we show that \texttt{Inexact ACC-SONATA-F} inherits the same communication complexity of its exact counterpart, \texttt{ACC-SONATA-F}.

\begin{theorem}\label{thm:ine_cata_f}
Consider problem \eqref{eq:problem} under Assumption~\ref{assump:class}, with optimal value function $u^\star$ and $\beta> \mu$ w.l.o.g.. 
Let $\{x^k \triangleq (x_i^k)_{i\in [m]}\}$ be the sequence generated by the Algorithm~\ref{alg:ine_ACCSONATA} under Assumption~\ref{assump:weight}, with\vspace{-0.1cm} \begin{equation}
\rho \leq  \mathcal{O} \left(\left(1+\frac{\kappa-1}{\beta/\mu}\right)^{-2}\right),\vspace{-0.1cm}\end{equation}   and the following tuning: \vspace{-0.1cm}\begin{equation}\label{eq:xi_decay}  \delta=\beta-\mu,\quad 
  T = \mathcal{O} \mybrace{ \log {\beta}/{\mu}}, \quad \xi_i^{k,t} = \mathcal{O} \mybrace{\mybrace{1-c\,\sqrt{ {\mu}/{\beta}}}^k \mybrace{16/17}^t}, \,\, \forall \,i \in [m], 
\end{equation} 
  and agents' surrogate functions \eqref{surrogates_f_i} in \textnormal{\texttt{Inexact-SONATA}}. Recall the optimality gap $\Delta(x^k)$ defined in \eqref{eq:opt_gap_def};
 then, there holds \vspace{-0.2cm}
\begin{equation*}   \Delta(x^k) = \mathcal{O} \mybrace{\mybrace{1-c\,\sqrt{ \frac{\mu}{\beta}}}^k},
\end{equation*} where $c\in (0,1)$ is some universal constant. 
Therefore,    $\Delta(x^K) \leq \varepsilon$, $\varepsilon>0$, in   
\begin{equation}\label{eq:K_total_inexact}
\mathcal{O}\left(\sqrt{ \frac{\beta}{\mu}}\cdot T\cdot \log  \frac{1}{\varepsilon}\right)
\end{equation}
total  (inner plus outer) communication steps.
\end{theorem} 
\begin{proof}
See Appendix~\ref{sec:ine_asf}.
\end{proof}

As anticipated, the  above result shows that if the subproblems (\ref{eq:ine_loc_opt}) are solved with increasing accuracy, as specified by the decay of $\{\xi^{k,t}\}$ in \eqref{eq:xi_decay}, the total number of communication steps (as in \eqref{eq:K_total_inexact}) for \texttt{Inexact ACC-SONATA-F} to reach an $\varepsilon$-solution  of (P)   matches that of \texttt{ACC-SONATA-F}, despite the presence of computation errors.  
We  discuss next the computation complexity of \texttt{Inexact ACC-SONATA-F}.  

Suppose we use a solution method $\mathcal{M}$ to solve the local optimization \eqref{eq:ine_loc_opt}; and let $T^{k,t}$ be the number of iterations required by $\mathcal{M}$ to solve \eqref{eq:ine_loc_opt} within the precision $\xi^{k,t}$  (at the  inner iteration $t$ of the outer step $k$ of \texttt{Inexact ACC-SONATA-F}). 
Then, in the setting of Theorem~\ref{thm:ine_cata_f},  the total number of steps taken by $\mathcal{M}$ for $\Delta(x^K) \leq \varepsilon$ 
reads $\sum_{k=0}^{K-1} \sum_{t=0}^{T-1} T^{k,t}.$  We provide next an explicit expression of such a  number  when $\mathcal{M}$ is a linearly convergent method. 
\begin{theorem}\label{thm:ine_cata_f_cp}
Let $\{x^k \triangleq (x_i^k)_{i\in [m]}\}$ be the sequence generated by the Algorithm~\ref{alg:ine_ACCSONATA} under Assumption~\ref{assump:weight}, in the same setting of Theorem~\ref{thm:ine_cata_f}.  Suppose that $\mathcal{M}$ is such that   $$T^{k,t}=\mathcal{O}\left( \kappa_{\mathcal{M}} \log  {1}/{\xi_i^{k,t}}\right),$$ 
for some  $\kappa_{\mathcal{M}}\geq 1$. Then, to reach $\Delta(x^K) \leq \varepsilon$, $\varepsilon>0$, the total number of iterations  taken by $\mathcal{M}$    reads
\begin{equation}\label{eq:total_it}
\mathcal{O}\left( \kappa_{\mathcal{M}} \, \frac{\beta}{\mu}  \, \log \frac{\beta}{\mu} \, \mybrace{\log  \frac{1}{\varepsilon}}^2 \right).
\end{equation}
\end{theorem}
\begin{proof}
See Appendix~\ref{sec:ine_asf}.
\end{proof}

We can make \eqref{eq:total_it} a bit more explicit depending on the choice of $\mathcal{M}$. Note that the condition number of  $u_i^k(x)$ in \eqref{eq:ine_loc_opt} (based on the  surrogate   \eqref{surrogates_f_i}) is $$\frac{\tilde{L}_{\text{mx}}}{\tilde{\mu}_{\text{mn}}} = 1 + \frac{L+\beta}{2\beta - \mu}.$$ Therefore, if $\mathcal{M}$ is the proximal gradient algorithm with stepsize $2/(\tilde{L}_{\text{mx}}+\tilde{\mu}_{\text{mn}})$, we have $$\kappa_{\mathcal{M}} =  {1 + \frac{L+\beta}{2\beta - \mu}},$$ and thus  \eqref{eq:total_it}  becomes 
\begin{equation}\label{eq:compexity_M}
\tilde{\mathcal{O}}\left( \left(1+ \frac{\kappa+\beta/\mu}{ 2\beta/\mu - 1}\right)\,\cdot \frac{\beta}{\mu} \cdot \mybrace{\log  \frac{1}{\varepsilon}}^2 \right)
\end{equation}
  number of total gradient evaluations/agent, while for $\mathcal{M}$ 
being the accelerated proximal gradient (with nominal tuning),  \eqref{eq:total_it}  reads  
\begin{equation}\label{eq:compexity_acc_M}
\tilde{\mathcal{O}}\left( \sqrt{1+ \frac{\kappa+\beta/\mu}{ 2\beta/\mu - 1}}\,\cdot \frac{\beta}{\mu} \cdot \mybrace{\log  \frac{1}{\varepsilon}}^2 \right)
\end{equation}
 number of total gradient evaluations, where $\tilde{\mathcal{O}}$ hides log-factors.

\subsection{Proof of Theorem~\ref{thm:ine_cata_f} and Theorem~\ref{thm:ine_cata_f_cp}}\label{sec:ine_asf}
\subsubsection{Sketch of the Proof}
The proof for Theorem~\ref{thm:ine_cata_f} and Theorem~\ref{thm:ine_cata_f_cp} is organized in the following three steps:
  
\textbf{Step 1:} We first establish  convergence   for   \texttt{Inexact-SONATA}.  Specifically, we prove the connections among the optimality measures and the consensus/tracking error in \eqref{prop1}, \eqref{eq:ine_con_tk} and \eqref{prop3}, in the presence of the inexactness parameter $\xi_i^{k}$'s;

\textbf{Step 2:} we show that by carefully choosing the inexactness parameters for every inner iteration of every outer loop, the number of inner iterations can still be a constant as $T = \mathcal{O} \mybrace{ \log \frac{\beta}{\mu} }$; and finally 

\textbf{Step 3:} combines the results in \textbf{Step 1} and \textbf{Step 2} to determine  the overall communication and computational complexity of \texttt{Inexact ACC-SONATA-F}.

\medskip

\textbf{Basic definitions and notation.} We begin introducing some definitions and basic facts that will be used throughout the proof. 

We denote by $\tilde{\mu}_i$ and $\tilde{L}_i$  the strong convexity and smoothness constants of  the surrogate function $\tilde{f}_i(\bullet\,; x_i^k)$, respectively.    We define:
\begin{equation}\label{eq:definitons}
g^k \triangleq \frac{1}{m} \sum_{i} u(x_i^k) - u^\star,\quad 
d_i^{k} \triangleq  \tx_i^{k+1/2}-x_i^{k},\quad \tau_i^{k} \triangleq \nabla f(x_i^{k})-y_i^{k},\quad \tilde{\mu}_{\text{mn}} \triangleq \min_{i\in [m]} \tilde{\mu}_i, \quad \tilde{L}_{\text{mx}} \triangleq \max_{i\in [m]} \tilde{L}_i,
\end{equation} 
and the concatenation of local variables
$
d^{k} \triangleq \left[{d_1^{k}}^\top,\cdots,{d_m^{k}}^\top \right]^\top,$  $\tau^{k} \triangleq \left[{\tau_1^{k}}^\top,\cdots,{\tau_m^{k}}^\top \right]^\top.
$
By \cite{sun2019distributed}, we know that there exist $\{D_i^\ell,\, D_i^u\}_{i\in [m]}$ such that 
\begin{align*}
    &D_i^{\ell}\, \mathbf{I} \preceq \nabla^{2}\tf_{i}(x;y)-\nabla^{2}f(x) \preceq D_i^{u}\, \mathbf{I}, \quad \forall x,y \in \mathbb{R}^d.
\end{align*}
We denote $D_i \triangleq \max \left\{\abs{D_i^\ell},\abs{D_i^u} \right\},$ $D_{\text{mx}} = \max_{i\in[m]} D_{i},$ and $D_{\text{mn}}^\ell \triangleq \min_{i\in [m]} D_{i}^\ell$.

We will leverage the following basic facts on subdifferential calculus; see, e.g., \cite{rockafellar2015convex}. The $\epsilon$-subdifferential of a convex function $f$, whit $\epsilon>0$,  is defined as
\[
\partial_\epsilon f(x) \triangleq \left\{\, g \, \big\vert \,\, f(y) \geq (f(x) -\epsilon ) + g^\top(y-x)  ,\,\, \forall\, y\in \mathbb{R}^d \right\}.
\]
A direct consequence of this definition is the following fact about $\epsilon$-minimizers, $x_\varepsilon$, of $f$ (assumed to be convex and proper),  i.e., 
\begin{equation}\label{eq:subd_1_prel}
x_\epsilon:\,\, f(x_\epsilon)-\min_{x\in \mathbb{R}^d} f(x)\leq \epsilon \quad \Leftrightarrow \quad      0\in \partial_\epsilon f(x_\epsilon).
\end{equation}
If, in addition, $f$ is also $L$-smooth, we have the following. 

\begin{lemma}
Let $f:\mathbb{R}^d\to \mathbb{R}$ be convex and $L$-smooth. Then,
\begin{align}\label{eq:subd_2}
\partial_\epsilon f(x) \subseteq \left\{\, \nabla f(x) + \chi \, \Big\vert \,\, \norm{\chi}^2 \leq 2L\,\epsilon \right\}.\end{align}
\end{lemma}

\subsubsection{Step 1: Convergence Results of \texttt{Inexact-SONATA}}
we present the convergence results of  $\texttt{Inexact-SONATA}$ to solve a general convex optimization problem 
\begin{equation}
\min_{x\in \mathbb{R}^d}\,\, u(x) \triangleq f(x) + r(x),\quad f(x)\triangleq \frac{1}{m} \sum_{i=1}^m f_i(x), 
\end{equation}
with $f(x)$ being $\Upsilon$-strongly convex and $\mathcal{L}$-smooth.  
We use single time index $(\bullet)^k$ to denote any variable $(\bullet)$ at the iteration $k$ of the Algorithm~\ref{alg:inexact_SF}.
To facilitate the discussion, we denote the smooth part of $u_i^k(x)$ in \eqref{eq:ine_loc_opt}  as $$s_i^k(x) \triangleq  \tilde{f}_i(x; x_i^k)  + \inn{y_i^k - \nabla f_i(x_i^k)}{x - x_i^k},$$ and thus write $$u_i^k(x) = s_i^k(x) + r(x).$$  

Our first step is to establish an inexact descent property of $u$ at each $\tx_i^{k+1/2}$, as stated in the following lemma, which   can be deemed as a counterpart of   \cite[Lemma~3.1]{sun2019distributed}, in the presence of  inexact computations of the solutions of the subproblems in   \texttt{(S.1)}.

\begin{lemma}\label{lm:ine_desc}
Let $\{x_i^k\}$ be the sequence generated by the Algorithm~\ref{alg:inexact_SF}; there holds
\begin{equation}\label{eq:ine_lemma1}
u(\tx_i^{k+1/2}) \leq u(x_i^{k})-\frac{1}{2}\left(\tilde{\mu}_{\text{mn}}+D_{\text{mn}}^\ell -\epsilon_1-\epsilon_2\right)\norm{d_{i}^{k}}^2 + \frac{1}{2\epsilon_1}\norm{\tau_{i}^{k}}^2+\mybrace{\frac{\tilde{L}_{\text{mx}}}{\epsilon_2}+1}\xi_{i}^{k}.
\end{equation}
$\epsilon_1, \epsilon_2$ with $\epsilon_1 + \epsilon_2 < \tilde{\mu}_{\text{mn}}+D_{\text{mn}}^l$ are parameters to be determined.
 \end{lemma}
\begin{proof}
At step \texttt{(S.1)}, $\tx_i^{k+1/2}$ is a $\xi_i^k$-optimal solution of $\min_{x\in \mathbb{R}^d} u_i^k(x)$.    Using \eqref{eq:subd_1_prel}, this implies that    $\tx_i^{k+1/2}$ satisfies  
\begin{align}\label{eq:subd_1}
0 \in \partial_{\xi_i^k} \, u_i^k(\tilde{x}_i^{k+1/2}) \subset \partial_{\xi_i^k} \, s_i^k(\tilde{x}_i^{k+1/2}) + \partial_{\xi_i^k} \, r(\tilde{x}_i^{k+1/2}).
\end{align}
Using \eqref{eq:subd_1}, \eqref{eq:subd_2} and the fact that  $s_i^k(x)$ is $\tilde{L}_i$-smooth, we infer that there exists a $\chi_i^k \in \mathbb{R}^{d}$ with
\begin{equation}\label{epDe}
     \norm{\chi_i^{k}}^2 \leq 2 \tilde{L}_i \xi_i^{k},
\end{equation} 
such that,
$$ h_i^k \triangleq \nabla s_i^k (\tx_i^{k+1/2}) + \chi_i^{k} \in \partial_{\xi_i^k} \, s_i^k(\tilde{x}_i^{k+1/2})\quad \text{and}
 -h_i^k \in \partial_{\xi_i^k} \, r(\tilde{x}_i^{k+1/2}).$$
Therefore, for any $x\in\mathbb{R}^{d},$ we have
\begin{equation}\label{inexactness}
\begin{aligned}
     r(x)-r(\tx_i^{k+1/2}) + \xi_i^{k} \geq \inn{x - \tx_i^{k+1/2}}{-h_i^k} =  \inn{\tx_i^{k+1/2}-x}{\nabla \tf_i(\tx_i^{k+1/2},x_i^{k})+y_i^{k}-\nabla f_i(x_i^{k})+\chi_i^{k}}.
\end{aligned}
\end{equation}

Using \eqref{inexactness} with $x = x_i^k$ leads to
\begin{equation}\label{eq:inexact_opt_cd}
\begin{aligned}
    r(x_i^{k})-r(\tx_i^{k+1/2}) + \xi_i^{k} \geq \inn{d_i^{k}}{\,\,  \nabla \tf_i(\tx_i^{k+1/2};x_i^{k})+y_i^{k}-\nabla f_i(x_i^{k})+\chi_i^k }  = \inn{d_i^{k}}{\,\,y_i^{k}+\widetilde{H}_i^{k}d_i^{k}+\chi_i^k}
\end{aligned}
\end{equation} 
with $\widetilde{H}_i^{k} \triangleq  \int_{0}^{1}\nabla^{2}\tf_i(\theta\tx_i^{k+1/2}+(1-\theta)x_i^{k};x_i^{k}){\rm d}\theta$.  Then, we get
\begin{align*}
    & f(\tx_i^{k+1/2}) \overset{(a)}{=}  f(x_i^{k}) + \inn{\nabla f(x_i^k)}{d_i^{k}}+\inn{H_i^{k}d_i^{k}}{d_i^{k}}  = f(x_i^{k})+\inn{\tau_i^{k}}{d_i^{k}} + \inn{y_i^{k}}{d_i^{k}}+\inn{H_i^{k}d_i^{k}}{d_i^{k}}\\
    & \overset{(b)}{\leq} f(x_i^{k})+\inn{\tau_i^{k}}{d_i^{k}} + \inn{H_i^{k}d_i^{k}}{d_i^{k}} +
    r(x_i^{k})-r(\tx_i^{k+1/2})-\inn{\widetilde{H}_i^{k}d_i^{k}}{d_i^{k}}-\inn{d_i^{k}}{\chi_i^k}+\xi_i^{k}\\
    & \overset{(c)}{\leq} f(x_i^{k})+\inn{\tau_i^{k}}{d_i^{k}} -\frac{1}{2}\mybrace{D_{\text{mn}}^\ell+\tilde{\mu}_{\text{mn}}}\norm{d_i^{k}}^2 +
    r(x_i^{k})-r(\tx_i^{k+1/2}) -\inn{d_i^{k}}{\chi_i^k}+\xi_i^{k}\\
    & \leq f(x_i^{k}) -\frac{1}{2}(D_{\text{mn}}^\ell+\tilde{\mu}_{\text{mn}})\norm{d_i^{k}}^2 + r(x_i^{k})-r(\tx_i^{k+1/2}) + \frac{1}{2}\left(\epsilon_1\norm{d_i^{k}}^2+\epsilon_1^{-1}\norm{\tau_i^{k}}^2\right) \\
    &\quad + \frac{1}{2}\left(\epsilon_2\norm{d_i^{k}}^2+\epsilon_2^{-1}\norm{\chi_i^k}^2\right) + \xi_i^{k},
\end{align*}
where in  $(a)$ we used the Taylor's formula with $H_i^{k} \triangleq \int_{0}^{1}(1-\theta)\nabla^{2}f(\theta \tx_i^{k+1/2}+(1-\theta)x_i^{k}){\rm d}\theta$; in $(b),$ we upper bound $\inn{y_i^{k}}{d_i^{k}}$ through \eqref{eq:inexact_opt_cd}; and in $(c),$ $H_i^{k} - \widetilde{H}_i^{k} \preceq - \frac{1}{2} \mybrace{D_{\text{mn}}^\ell+\tilde{\mu}_{\text{mn}}} \mathbf{I}$, obtained by setting $\alpha = 1$ in \cite[(32)]{sun2019distributed}.  We obtain $\eqref{eq:ine_lemma1}$ using $\eqref{epDe}$.
\end{proof}

Recalling the definition of the optimality gap $g^k$ as  in  \eqref{eq:definitons} and using \eqref{eq:ine_lemma1} and the convexity of $u$, we get
\begin{equation}\label{eq2:lemma1}
g^{k+1} \leq g^k-\frac{1}{2m}\left(\tilde{\mu}_{\text{mn}}+D_{\text{mn}}^\ell-\epsilon_1-\epsilon_2\right)\norm{d^{k}}^2 + \frac{1}{2m\epsilon_1}\norm{\tau^k}^2+\mybrace{\frac{\tilde{L}_{\text{mx}}}{\epsilon_2}+1} \overm\sum_{i=1}^m \xi_i^k.
\end{equation}

As   second step, we derive below the lower bound of $\norm{d^{k}}$, which is the counterpart of  \cite[Lemma~3.2]{sun2019distributed} when inexact solutions are allowed in agents' subproblems.  
\begin{lemma}
The following lower bound holds for $\norm{d^{k}}^2$:
\begin{equation}\label{eq:ine_lb_d}
    \overm\norm{d^{k}}^2 \geq \frac{\Upsilon}{D_{\text{mx}}^{2}}\left(g^{k+1}-\frac{2}{m\mu}\norm{\tau^{k}}^{2}-\overm\left(1+\frac{4\tilde{L}_{\text{mx}}}{\Upsilon}\right)\summ\xi_j^{k}\right).
\end{equation}
\end{lemma}
\begin{proof}
Applying \eqref{inexactness} with $x = x^\star$ leads to
\begin{equation}\label{eq1:lemma2}
    r(x^{\star})-r(\tx_i^{k+1/2}) + \xi_i^{k} \geq \inn{\tx_i^{k+1/2}-x^{\star}}{\nabla\tf_i(\tx_i^{k+1/2};x_i^{k})+y_i^{k}-\nabla f_i(x_i^{k})+\chi_i^k}
\end{equation}
By the $\Upsilon$-strongly convexity of $f$, we have
\begin{equation}\label{eq2:lemma2}
    \begin{aligned}
    & u(x^{\star}) \geq u(\tx_i^{k+1/2})+r(x^{\star})-r(\tx_i^{k+1/2})+\inn{\nabla f(\tx_i^{k+1/2})}{x^{\star}-\tx_i^{k+1/2}}+\frac{\Upsilon}{2}\norm{x^{\star}-\tx_i^{k+1/2}}^{2}\\
    & \overset{\eqref{eq1:lemma2}}{\geq} u(\tx_i^{k+1/2})+\inn{x^{\star}-\tx_i^{k+1/2}}{\nabla f(\tx_i^{k+1/2})-\nabla\tf_i(\tx_i^{k+1/2};x_i^{k})-y_i^{k}+\nabla f_i(x_i^{k})-\chi_i^k} + \frac{\Upsilon}{2}\norm{x^{\star}-\tx_i^{k+1/2}}^2-\xi_i^{k}\\
    & \geq u(\tx_i^{k+1/2})-\frac{1}{2\Upsilon}\norm{\nabla f(\tx_i^{k+1/2})-\nabla f(x_i^{k})+\nabla f_i(x_i^{k})-\nabla \tf_i(\tx_i^{k+1/2};x_i^{k})+\tau_i^{k}-\chi_i^k}^{2}-\xi_i^{k}\\
    &\geq u(\tx_i^{k+1/2})-\frac{D_i^2}{\Upsilon}\norm{d_i^{k}}^{2}-\frac{2}{\Upsilon}\norm{\tau_i^{k}}^{2}-\frac{2}{\Upsilon}\norm{\chi_i^k}^2-\xi_i^{k} \overset{\eqref{epDe}}{\geq } u(\tx_i^{k+1/2})-\frac{D_i^2}{\Upsilon}\norm{d_i^{k}}^{2}-\frac{2}{\Upsilon}\norm{\tau_i^{k}}^{2}- \mybrace{1+ \frac{4 \tilde{L}_i}{\Upsilon}}  \xi_i^{k} 
    \end{aligned}
\end{equation}
Summing the inequalities over $i$, and using the convexity of $u(x)$ lead to the conclusion.
\end{proof}
Combining \eqref{eq2:lemma1} and \eqref{eq:ine_lb_d} to cancel out $\norm{d^k}^2$, and noticing that $$ \overm\norm{\tau^{k}}^{2} \leq 4L_{\text{mx}}^{2}\norm{\xb^{k}}^{2}+2\norm{\yb^{k}}^{2} ,$$ we obtain
\begin{equation}\label{prop1}
g^{k+1} \leq \sigma_{0} g^k+\eta_{0} \mybrace{4L_{\text{mx}}^{2}\norm{\xb^{k}}^2+2\norm{\yb^{k}}^2}+\eta_{1}\, \frac{1}{m}\summ \xi_j^{k},    
\end{equation}
with
\begin{equation}
 \begin{aligned}
 &\sigma_{0} \triangleq \frac{2D_{\text{mx}}^{2}}{2D_{\text{mx}}^{2}+\Upsilon\mybrace{\tilde{\mu}_{\text{mn}}+D_{\text{mn}}^\ell-\epsilon_1-\epsilon_2}},\qquad  \eta_{0} \triangleq \sigma_{0}\left(\frac{1}{2\epsilon_1}+\frac{\tilde{\mu}_{\text{mn}}+D_{\text{mn}}^\ell-\epsilon_1-\epsilon_2}{D_{\text{mx}}^2}\right),\\
 & \eta_{1} \triangleq \sigma_{0}\left(\frac{\tilde{L}_{\text{mx}}}{\epsilon_2}+1+\frac{\Upsilon \mybrace{\tilde{\mu}_{\text{mn}}+D_{\text{mn}}^\ell-\epsilon_1-\epsilon_2} }{2D_{\text{mx}}^2}\left(1+\frac{4L_{\text{mx}}^2}{\Upsilon}\right)\right).
 \end{aligned}
\end{equation}

On the other hand, we recall the following result on the consensus and tracking error from \eqref{3.5}:
\begin{equation}\label{eq:ine_con_tk}
\begin{aligned}
 4 L_{\text{mx}}^2 \norm{x_\perp^{k+1}}^2 + 2 \norm{y_\perp^{k+1}}^2 \leq \rho^2\,  \mybrace{32 L_{\text{mx}}^2 \norm{x_\perp^{k}}^2 + 6 \norm{y_\perp^{k}}^2 + 14 L_{\text{mx}}^2 \frac{1}{m} \norm{d^{k}}^2 }.
\end{aligned}
\end{equation}


Equipped with  \eqref{prop1} and \eqref{eq:ine_con_tk}, we proceed to   bound $\norm{d^{k}}^2$ as a function of the optimality gap   $g^k$ and the consensus/tracking error to close the loop.  To this end, we give the following result, which is the counterpart  of \cite[Proposition~3.6]{sun2019distributed} in the presence of inexact solutions.
\begin{proposition}
The following upper bound hold for $\norm{d^k}^2$:
\begin{equation}\label{prop3}
  \overm\norm{d^{k}}^{2} \leq \sigma_{1}g^k+\eta_2\norm{\yb^{k}}^{2} +\frac{\eta_3}{m}\summ \xi_j^{k},
\end{equation}
with
\[
\sigma_1 \triangleq   \frac{2}{\Upsilon}\left(2+\frac{8D_{\text{mx}}^{2}}{\tilde{\mu}_{\text{mn}}^{2}}+\frac{32L_{\text{mx}}^{2}}{\tilde{\mu}_{\text{mn}}^{2}}\right), \quad
\eta_2 \triangleq  \frac{8}{\tilde{\mu}_{\text{mn}}^{2}}, \quad 
\eta_3 \triangleq  \left(\frac{16L_{\text{mx}}}{\tilde{\mu}_{\text{mn}}^{2}}+\frac{4}{\tilde{\mu}_{\text{mn}}}\right).
\]

\end{proposition}
\begin{proof}
Applying \eqref{inexactness} with $x = x^\star$ and invoking the optimality condition of $x^\star$, we get
\begin{equation*}
    \begin{aligned}
        & r(x^{\star})-r(\tx_i^{k+1/2})-\inn{\tx_i^{k+1/2}-x^{\star}}{\nabla\tf_i(\tx_i^{k+1/2};x_i^{k})+y_i^{k}-\nabla f_i(x_i^{k})+\chi_i^k}+\xi_i^{k} \geq 0,\\
        &\inn{\nabla f(x^{\star})}{\tx_i^{k+1/2}-x^{\star}}+r(\tx_i^{k+1/2})-r(x^{\star}) \geq 0.
    \end{aligned}
\end{equation*}
Combining the above two, we get
\begin{equation*}
    \begin{aligned}
        & 0 \leq \inn{\nabla f(x^{\star})-y_i^{k}+\nabla f_i(x_i^{k})-\nabla \tf_i(\tx_i^{k+1/2};x_i^{k})-\chi_i^k \pm \bar{y}^{k}}{\,\, \tx_i^{k+1/2}-x^{\star}}+\xi_i^{k}\\
        & \leq \inn{\nabla f(x^{\star})-\overm\summ\nabla f_j(x_j^{k})+\nabla f_i(x_i^{k})-\nabla \tf_i(\tx_i^{k+1/2};x_i^{k})}{\,\, \tx_i^{k+1/2}-x^{\star}}\\
        & \quad + \norm{\bar{y}^{k}-y_i^{k}}\norm{\tx_i^{k+1/2}-x^{\star}}+\norm{\chi_i^k}\norm{\tx_i^{k+1/2}-x^{\star}}+\xi_i^{k}\\
        & \leq \inn{\nabla f(x^{\star})-\nabla f(x_i^{k})+\nabla f_i(x_i^{k})\pm \nabla \tf_i(x^{\star};x_i^{k})-\nabla \tf_i(\tx_i^{k+1/2};x_i^{k})}{\,\, \tx_i^{k+1/2}-x^{\star}}\\
        & \quad + \norm{\bar{y}^{k}-y_i^{k}}\norm{\tx_i^{k+1/2}-x^{\star}}+\norm{\chi_i^k}\norm{\tx_i^{k+1/2}-x^{\star}}+\norm{\nabla f(x_i^{k})-\overm\summ\nabla f_j(x_j^{k})}\norm{\tx_i^{k+1/2}-x^{\star}}+\xi_i^{k}\\
        & = \inn{ \int_{0}^{1}\mybrace{\nabla^{2}f(\theta x^\star+(1-\theta)x_i^{k}) - \nabla^{2}\tilde{f}_i(\theta x^\star+(1-\theta)x_i^{k};x_i^{k}) } \mybrace{x^\star - x_i^{k}}{\rm d}\theta}{\,\, \tx_i^{k+1/2}-x^{\star}}\\
        &\quad + \inn{\nabla \tf_i(x^{\star};x_i^{k})-\nabla \tf_i(\tx_i^{k+1/2};x_i^{k})}{\,\, \tx_i^{k+1/2}-x^{\star}}  + \norm{\bar{y}^{k}-y_i^{k}}\norm{\tx_i^{k+1/2}-x^{\star}}\\ 
        & \quad +\left(\overm\summ L_{j}\norm{x_j^{k}-x_i^{k}}\right)\norm{\tx_i^{k+1/2}-x^{\star}}+\norm{\chi_i^k}\norm{\tx_i^{k+1/2}-x^{\star}}+\xi_i^{k} \\
        &\leq D_i\norm{x^{\star}-x_i^{k}}\norm{\tx_i^{k+1/2}-x^{\star}}-\tilde{\mu}_{i}\norm{\tx_i^{k+1/2}-x^{\star}}^{2}+\norm{\bar{y}^{k}-y_i^{k}}\norm{\tx_i^{k+1/2}-x^{\star}}\\
        &\quad +\left(\overm\summ L_{j}\norm{x_j^{k}-x_i^{k}}\right)\norm{\tx_i^{k+1/2}-x^{\star}}+\norm{\chi_i^k}\norm{\tx_i^{k+1/2}-x^{\star}}+\xi_i^{k}\\
        &\leq \frac{D_i}{2q_1}\norm{x_i^{k}-x^{\star}}^2 +\frac{1}{2q_2}\norm{\bar{y}^{k}-y_i^{k}}^{2}+\frac{1}{2q_3}\left(\overm\summ L_{j}\norm{x_j^{k}-x_i^{k}}\right)^{2}+\frac{1}{2q_4}\norm{\chi_i^k}^{2}\\
        &\quad +\left(\frac{1}{2}\left(D_{i}q_{1}+q_{2}+q_{3}+q_{4}\right)-\tilde{\mu}_i\right)\norm{\tx_i^{k+1/2}-x^{\star}}^{2} +\xi_i^{k}
    \end{aligned}
\end{equation*}
Setting $q_1 = \frac{\tilde{\mu}_{i}}{4D_{i}}, q_2 = q_3 = q_4 = \frac{\tilde{\mu}_i}{4}$ in the above and plugging in the following inequalities,  
\begin{equation*}
\begin{aligned}
    &\norm{\tx_i^{k+1/2}-x^{\star}}^{2} \geq - \norm{x^{\star}-x_i^{k}}^2+ \frac{1}{2}\norm{d_i^{k}}^2,\\
    & \norm{x^{\star}-x_i^{k}}^{2} \leq \frac{2}{\Upsilon}\mybrace{u(x_i^{k})-u(x^{\star})},\\
    & \norm{x_j^{k}-x_i^{k}}^2 \leq 2\norm{x_j^{k}-x^{\star}}^2+2\norm{x_i^{k}-x^{\star}}^2,
\end{aligned}
\end{equation*}
we get
\begin{equation*}
    \begin{aligned}
        \norm{d_i^{k}}^2 &\leq \left(2+\frac{8D_{i}^2}{\tilde{\mu}_{i}^{2}} + \frac{16 L_{\text{mx}}^2 }{\tilde{\mu}_{i}^{2}} \right)\norm{x^{\star}-x_i^{k}}^{2}+\frac{8}{\tilde{\mu}_{i}^{2}}\norm{\bar{y}^{k}-y_i^{k}}^{2}+\frac{16}{\tilde{\mu}_{i}^{2}}\frac{L_{\text{mx}}^2}{m}\summ \norm{x_j^{k}-x^{\star}}^2 +\frac{8}{\tilde{\mu}_{i}^2}\norm{\chi_i^k}^2+\frac{4}{\tilde{\mu}_i}\xi_i^{k}.
    \end{aligned}
\end{equation*}
Using $\eqref{epDe}$ and summing the above over $i\in [m]$, we get
\begin{equation*}
    \begin{aligned}
        \overm\norm{d^{k}}^{2} &\leq \frac{2}{\Upsilon}\left(2+\frac{8D_{\text{mx}}^{2}}{\tilde{\mu}_{\text{mn}}^{2}}+\frac{32L_{\text{mx}}^{2}}{\tilde{\mu}_{\text{mn}}^{2}}\right)g^k+\frac{8}{\tilde{\mu}_{\text{mn}}^{2}}\norm{\yb^{k}}^{2} +\left(\frac{16L_{\text{mx}}}{\tilde{\mu}_{\text{mn}}^{2}}+\frac{4}{\tilde{\mu}_{\text{mn}}}\right) \overm \summ \xi_j^{k}.
    \end{aligned}
\end{equation*}

\end{proof}

Now we are ready to combine \eqref{prop1}, \eqref{eq:ine_con_tk} and \eqref{prop3} to obtain the contraction results for the \texttt{Inexact-SONATA}.  
Beforehand, we notice that when \texttt{Inexact-SONATA} is adopted inside Algorithm~\ref{alg:ine_ACCSONATA} with $\delta = \beta - \mu,$ the aforementioned parameters can be set as
\begin{align*}
\Upsilon = \beta, \quad D_{\text{mn}}^\ell = 0, \quad D_{\text{mx}} = 2\beta, \quad \tilde{\mu}_{\text{mn}} = 2\beta - \mu, \quad L_{\text{mx}} = L+2\beta - \mu, \quad \tilde{L}_{\text{mx}} = L+3\beta - \mu.
\end{align*}
In addition, we choose $$\epsilon_1 = \epsilon_2 = \frac{2\beta - \mu}{4}.$$  One can then easily check in $\eqref{prop1}$, $\sigma_{0} \leq \frac{16}{17}$, $\eta_{0}\leq \frac{36}{17\beta}$.  Multiplying \eqref{eq:ine_con_tk} by $\frac{72}{17\beta}$ and \eqref{prop3} by $\frac{1008 L_{\text{mx}}^2 \rho^2}{17 \beta}$, and combining the obtained inequalities with \eqref{prop1} yield

\begin{equation}\label{descent_E2} 
\begin{aligned} & g^{k+1}+\frac{288L_{\text{mx}}^{2}}{17\beta} \norm{x_{\bot}^{k+1}}^{2}+\frac{144}{17\beta} \norm{y_{\bot}^{k+1}}^{2} \\ 
& \leq  \left( \frac{16}{17} + \frac{1008 \, L_{\text{mx}}^2 \, \rho^2}{17\beta} \sigma_1 \right)g^{k}+ \frac{144 L_{\text{mx}}^2}{17\beta} \mybrace{1+ 16 \rho^2 }\norm{x_{\bot}^{k}}^{2} + \frac{72}{17\beta} \mybrace{1+6\rho^2 + 14L_{\text{mx}}^2 \rho^2 \eta_2} \norm{y_{\bot}^{k}}^{2} \\ 
&\quad +\left(\eta_1+ \frac{1008\, L_{\text{mx}}^2 \, \rho^2}{17\beta} \eta_3 \right)\overm\summ\xi_{j}^{k}. 
\end{aligned} \end{equation}

Thus, with 
\begin{align}
    \rho^2 \leq \min \mybrace{\frac{\beta}{2016\,L_{\text{mx}}^2 \, \sigma_1},\,\, \frac{8}{17 \mybrace{3+7 \, L_{\text{mx}}^2 \, \eta_2}}},
\end{align}
we have 
\begin{equation}\label{inexact_SF:descent}
    g^{k+1}+\frac{288L_{\text{mx}}^{2}}{17\beta} \norm{x_{\bot}^{k+1}}^{2}+\frac{144}{17\beta} \norm{y_{\bot}^{k+1}}^{2} \\  \leq \frac{33}{34} \mybrace{g^{k}+\frac{288L_{\text{mx}}^{2}}{17\beta} \norm{x_{\bot}^{k}}^{2}+\frac{144}{17\beta} \norm{y_{\bot}^{k}}^{2} }+\eta_4\overm\summ\xi_j^{k}, 
\end{equation}
where  $$\eta_4 \triangleq \eta_1+ \frac{1008\, L_{\text{mx}}^2 \, \rho^2}{17\beta} \eta_3.$$

\subsubsection{Step 2: \texttt{Inexact-SONATA} as an Inner Algorithm in \texttt{Inexact ACC-SONATA-F}}
In this section, we use double time index $(\bullet)^{k,t}$ to denote any variable $(\bullet)$ of \texttt{Inexact SONATA} in the $t$-th inner iteration of the $k$-th outer iteration of the Algorithm~\ref{alg:ine_ACCSONATA}.  Recalling the definition of $P^k$ in \eqref{eq:LYAPU_0}, we define 
\begin{align}\label{eq:def_xi}
\xi_{i}^{k,t} \triangleq \frac{P^0}{\eta_4} \, \mybrace{1-c\,\sqrt{ \frac{\mu}{\beta}}}^k\, \mybrace{\frac{16}{17}}^t,
\quad \text{for} \,\, \forall \, i\in[m], \, k\geq 0,\, t\geq 0.
\end{align}

To comply with the notations in \eqref{eq:def_ekt}, we define 
\begin{equation}\label{eq:def_ekt_2}
    e^{k,t} \triangleq  c_x \norm{x_\perp^{k,t}}^2 + c_y \norm{y_\perp^{k,t}}^2,\quad \text{with} \quad c_x \triangleq \frac{288L_{\text{mx}}^{2}}{17\beta} \quad  \text{and} \quad  c_y \triangleq \frac{144}{17\beta}.
\end{equation}
According to the discussion in Section~\ref{sec_Step1}, if one can guarantee
\begin{equation}\label{eq:termination_cp}
    g^{k,T} + e^{k,T} \leq \epsilon^{k+1},\quad   k=0,1, \ldots,   
\end{equation} with  $\epsilon^k \triangleq P^0\cdot \mybrace{1-c\cdot \alpha}^k$, then $P^k = \mathcal{O} \mybrace{\mybrace{1-c\cdot \alpha}^k}$.
By \eqref{inexact_SF:descent}, we have
\begin{align*}
    g^{k,t+1} + e^{k,t+1} \leq \frac{33}{34} \mybrace{g^{k,t} + e^{k,t}} + P^0\mybrace{1-c\,\sqrt{ \frac{\mu}{\beta}}}^k \mybrace{\frac{16}{17}}^t, \,\, \forall \, k\geq 0,\,\, t\geq 0.
\end{align*}
In particular, for a fixed $k$, applying the above telescopically on $t$ yields
\begin{align*}
    & g^{k,T} + e^{k,T} \leq \mybrace{\frac{33}{34}}^T \mybrace{g^{k,0} + e^{k,0}} + P^0 \mybrace{1-c\,\sqrt{ \frac{\mu}{\beta}}}^k \sum_{t=0}^T \mybrace{\frac{33}{34}}^{T-t} \mybrace{\frac{16}{17}}^t \\
    & = \mybrace{\frac{33}{34}}^T \mybrace{ g^{k,0} + e^{k,0} + P^0 \mybrace{1-c\,\sqrt{ \frac{\mu}{\beta}}}^k \sum_{t=0}^T \mybrace{\frac{32}{33}}^{t}}  \leq  \mybrace{\frac{33}{34}}^T \mybrace{ g^{k,0} + e^{k,0} + 33 \,P^0 \mybrace{1-c\,\sqrt{ \frac{\mu}{\beta}}}^k }.
\end{align*}
Therefore, according to Lemma~\ref{lm:good_init}, for $k\geq 1$, the number of inner  loops needed for \eqref{eq:termination_cp} can be bounded as
\begin{align*}
 T & \leq \ceil{34\,\log\frac{g^{k,0} + e^{k,0} + 33\,P^0 \mybrace{1-c\,\sqrt{ \frac{\mu}{\beta}}}^k}{\ep^{k+1}}}\\
& \overset{\eqref{termination_zero_bound}}{\leq}  \ceil{34 \, \log \frac{P^0 \mybrace{2 (1-c\cdot \alpha)^k + \frac{576(\beta-\mu)^2}{\mu\beta} c_2\, (1-c\cdot \alpha)^{k-1} + 33 (1-c\cdot \alpha)^k }}{P^0\, (1-c\cdot \alpha)^{k+1}}} \\
    & = \ceil{34 \, \log \frac{35 (1-c\cdot \alpha) + \frac{576(\beta-\mu)^2}{\mu\beta} c_2 }{(1-c\cdot \alpha)^2}} = \mathcal{O} \mybrace{\log\frac{\beta}{\mu}}.
\end{align*}
 For $k=0$,
due to $g^{0,0}+e^{0,0}\leq P^0,$ we have
$T = \ceil{34 \, \log \frac{34}{1-c\cdot \alpha}}$, which is smaller than the RHS of the above. 
 

\subsubsection{Step 3: Complexity of $\texttt{Inexact ACC-SONATA-F}$}
The total number of computation steps    taken by $\mathcal{M}$ as in Theorem~\ref{thm:ine_cata_f_cp} can be  readily obtained as follows: 
\begin{align*}
    & \sum_{k=0}^{K-1} \sum_{t=0}^{T-1} T^{k,t} = \kappa_{\mathcal{M}} \sum_{k=0}^{K-1} \sum_{t=0}^{T-1} \log \frac{1}{\xi_1^{k,t}} \overset{\eqref{eq:def_xi}}{\leq } \kappa_{\mathcal{M}} \sum_{k=0}^{K-1} \sum_{t=0}^{T-1}  \mybrace{\log \mybrace{\frac{\eta_4}{P^0}} + k \, \log \mybrace{\frac{1}{1 - c}} + t \,\log \frac{17}{16} } \\
    & \leq \kappa_{\mathcal{M}} \mybrace{KT \log \mybrace{\frac{\eta_4}{P^0}} + T \frac{K(K-1)}{2} \log \mybrace{\frac{1}{1 - c}} +  K \frac{T(T-1)}{2} \log \frac{17}{16}}.
\end{align*}
Since $K = \mathcal{O} \mybrace{\sqrt{\frac{\beta}{\mu}} \log \frac{1}{\epsilon}}$ and $T = \mathcal{O} \mybrace{\log \frac{\beta}{\mu}},$ we have
\begin{align*}
    \sum_{k=0}^{K-1} \sum_{t=0}^{T-1} T^{k,t} =  \mathcal{O} \mybrace{\kappa_{\mathcal{M}} \, K^2 \, T} = \mathcal{O}\left( \kappa_{\mathcal{M}} \, \frac{\beta}{\mu}  \, \log \frac{\beta}{\mu} \, \mybrace{\log  \frac{1}{\varepsilon}}^2 \right).
\end{align*}
In addition, the total communication complexity is $K\cdot T = \mathcal{O}\left(\sqrt{ \frac{\beta}{\mu}}\cdot T\cdot \log  \frac{1}{\varepsilon}\right).$

\end{document}